\documentclass[11pt,a4paper]{amsart}

\usepackage{amsmath}
\usepackage{amsfonts}
\usepackage{amssymb}
\usepackage{amsthm}
\usepackage[usenames, dvipsnames]{color}
\usepackage{tikz}
\usepackage{tikz-cd}

\usepackage[font=small]{caption}
\usepackage[font=small]{subcaption}
\usepackage{wrapfig}
\usepackage{float} 
\usepackage{multicol}
\usepackage{empheq}
\usepackage{tikzit}

\tikzstyle{paire}=[fill=white, draw=none, shape=rectangle]
\tikzstyle{text_box}=[fill=white, draw=black, shape=rectangle]
\tikzstyle{decision}=[rectangle, minimum width=3cm, minimum height=1cm, text centered, draw=black]
\tikzstyle{arrow_flou}=[fill=none, draw={rgb,255: red,191; green,191; blue,191}, -Stealth, dashed]
\tikzstyle{clear_text}=[fill=white, draw=none, shape=rectangle, text={rgb,255: red,150;green,150; blue,150}]
\tikzstyle{start_point}=[fill=black, draw=black, shape=circle, radius=1pt]
\tikzstyle{ldn_text}=[fill=none, draw=black, shape=rectangle]
\tikzstyle{label}=[fill=white, draw=none, shape=circle]

\tikzstyle{simx}=[draw={rgb,255: red,146; green,179; blue,255}, ->, very thick]
\tikzstyle{simy}=[draw={rgb,255: red,255; green,92; blue,92}, ->, very thick]
\tikzstyle{usimx}=[draw={rgb,255: red,146; green,179; blue,255}, -, very thick]
\tikzstyle{usimy}=[draw={rgb,255: red,255; green,92; blue,92}, -, very thick]
\tikzstyle{light_grid}=[-, draw={rgb,255: red,210; green,210; blue,210}, thin, dashed]
\tikzstyle{walk_step}=[-, draw={rgb,255: red,162; green,43; blue,209}, -Stealth, very thick]
\tikzstyle{limits}=[-, pattern={Lines[angle=45,distance={5pt}]}, pattern color={rgb,255: red,180; green,180; blue,180}, draw=none]
\tikzstyle{arrow}=[thick, ->, >=Stealth]
\tikzstyle{fill_y_0}=[-, fill opacity=0.1, tikzit fill=none, fill={rgb,255: red,0; green,0; blue,255}]
\tikzstyle{fill_x_0}=[-, fill opacity=0.1, tikzit fill=none, fill={rgb,255: red,255; green,0; blue,0}]

\input{figures/article.tikzdefs}
\usepackage{pgfplots}
\pgfplotsset{compat=1.13} 

\calclayout

\usepackage[utf8]{inputenc}

\definecolor{liens}{rgb}{1,0,0}
\usepackage[colorlinks=true, linkcolor=blue,
hyperfootnotes=true,citecolor=blue,urlcolor=blue]{hyperref}

\usepackage{thmtools}

\newtheorem{thm}{Theorem}[section]
\newtheorem{cor}[thm]{Corollary}
\newtheorem{lem}[thm]{Lemma}
\newtheorem{prop}[thm]{Proposition}
\newtheorem{conj}[thm]{Conjecture}

\theoremstyle{definition}
\newtheorem{defi}[thm]{Definition} 
\numberwithin{figure}{section}
\numberwithin{table}{section}

\newtheorem{rem}[thm]{Remark}
\newtheorem{exa}[thm]{Example}

\newtheorem{assumption}[thm]{Assumption}
\renewcommand{\thmcontinues}[1]{continued}
\numberwithin{equation}{section}

\def\imag{\rm{Im}}
\def\Qb{\overline{\mathbb{Q}}}

\def\Cmul{\C_{\text{mul}}}
\def\C{\mathbb{C}}

\def\kinv{k_{\text{inv}}}
\def\Ktld{\widetilde{K}}

\def\id{\text{id}}
\def\invS{1/S(x,y)}

\def\eqdef{\overset{\mathrm{def}}{=}}

\def\levX{\mathcal{X}}
\def\levY{\mathcal{Y}}

\def\Gmod{\calG_{\lambda}}
\def\st{\,|\,}          
\def\sts{\,/\,}         
\def\id{\mathrm{id}}
\def\wgammax{\widetilde{\gamma_x}}
\def\wgammay{\widetilde{\gamma_y}}
\def\sinv{\omega}

\def\Res{\operatorname{Res}}

\def\cO{\mathcal{O}}

\def\Z{\mathbb{Z}}
\def\C{\mathbb{C}}

\def\Q{\mathbb{Q}}
\def\N{{\mathbb N}}
\def\K{{\mathbb K}}

\def\P1{\mathbb{P}^{1}}
\def\Gal{\mathrm{Gal}}
\def\Aut{\mathrm{Aut}}
\def\beq{\begin{equation}}
\def\eeq{\end{equation}}
\def\Et{E_t}

\def\P2{\mathbb{P}^{2}}
\def\GCD{\operatorname{gcd}}

\def\K{\mathbb{K}}

\def\P1{\mathbb{P}^{1}}

\def\calC{{\mathcal{C}}}
\def\calS{\mathcal{S}}
\def\calW{\mathcal{W}}

\def\calH{{\mathcal{H}}}

\def\calG{{\mathcal{G}}}

\def\calO{{\mathcal{O}}}

\def\calP{{\mathcal{P}}}

\def\Gal{\mathrm{Gal}}

\def\Gal{{\rm Gal}}

\def\Ker{{\rm{Ker}}}

\def\P{\mathbb{P}}

\newcommand\xqed[1]{%
  \leavevmode\unskip\penalty9999 \hbox{}\nobreak\hfill
  \quad\hbox{#1}}
\newcommand\exqed{\xqed{$\square$}}

\title[Galoisian structure of large steps walks in the quadrant]{Galoisian structure of large steps walks \\ in the  quadrant}

\author{Pierre Bonnet}
\address{Laboratoire Bordelais de Recherche en Informatique, 351, cours de la
Libération F-33405 Talence, France}
\email{pierre.bonnet@u-bordeaux.fr}
\author{Charlotte Hardouin}
\address{Institut de Mathématiques de Toulouse, Université Paul Sabatier, 118, route de Narbonne, 31062 Toulouse, France}
\email{hardouin@math.univ-toulouse.fr}

\begin{document}

\maketitle
\begin{abstract}
The enumeration of walks in the quarter plane confined in the first
quadrant has attracted a lot of attention over the past fifteenth
years. The generating functions associated to small steps models
satisfy a functional equation in two catalytic variables.  For such
models, Bousquet-Mélou and Mishna defined a group called
\emph{the group of the walk} which turned out to be central in the
classification of small steps models. In particular, its action on the
catalytic variables yields a set of change of variables compatible
with the structure of the functional equation. This particular set
called the \emph{orbit} has been generalized to models with arbitrary
large steps by Bostan, Bousquet-Mélou and Melczer.  However, the orbit
had till now no underlying group.

  In this article, we endow the orbit with the action of a Galois
group, which extends the group of the walk to models with large
steps. Within this Galoisian framework, we generalized the notions of
\emph{invariants} and \emph{decoupling}. This enable us to develop a
general strategy to prove the algebraicity of models with small
backward steps. Our constructions lead to the first proofs of
algebraicity of weighted models with large steps, proving in
particular a conjecture of Bostan, Bousquet-Mélou and Melczer, and
allowing us to find new algebraic models with large steps.
\end{abstract}

\section{Introduction} \label{sect:walks}

We consider $2$-dimensional lattice weighted walks confined to the
quadrant $\N^2$ as illustrated in Figure~\ref{fig:g2_mod}. In recent
years, the enumeration of such walks has attracted a lot of attention
involving many new methods and tools.  This question is ubiquitous
since lattice walks encode several classes of mathematical objects in
discrete mathematics (permutations, trees, planar maps, \dots), in
statistical physics (magnetism, polymers, \dots), in probability
theory (branching processes, games of chance \dots), in operations
research (birth-death processes, queueing theory).

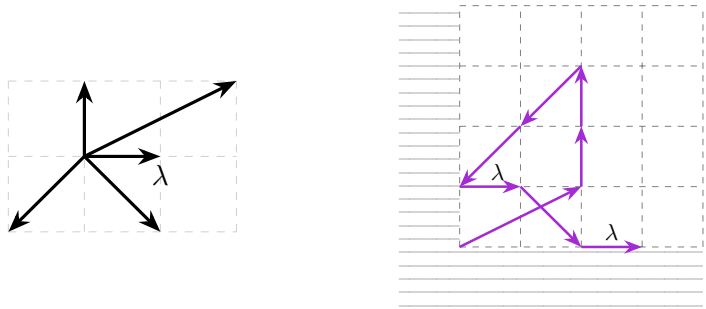
\begin{figure}[ht] 
    \centering
    \begin{subfigure}[c]{0.3\textwidth}
      \centering
      \scalebox{1}{\begin{tikzpicture}
	\begin{pgfonlayer}{nodelayer}
		\node [style=none] (0) at (0, 0) {};
		\node [style=none] (1) at (-2, -2) {};
		\node [style=none] (2) at (2, -2) {};
		\node [style=none] (3) at (2, 0) {};
		\node [style=none] (4) at (4, 2) {};
		\node [style=none] (5) at (0, 2) {};
		\node [style=none] (lamb) at (2, -0.5) {$\lambda$};
		\node [style=none] (6) at (0.5, 1.75) {};
		\node [style=none] (7) at (-1.25, -1.75) {};
	\end{pgfonlayer}
	\begin{pgfonlayer}{edgelayer}
	\draw[step=2.0,dashed,thin,color={rgb,255: red,210; green,210; blue,210}] (-2,-2) grid (4,2);
		\draw [style={axis_grid}] (0.center) to (5.center);
		\draw [style={axis_grid}] (0.center) to (4.center);
		\draw [style={axis_grid}] (0.center) to (3.center);
		\draw [style={axis_grid}] (0.center) to (2.center);
		\draw [style={axis_grid}] (0.center) to (1.center);
	\end{pgfonlayer}
\end{tikzpicture}}
    \end{subfigure}
    \begin{subfigure}[c]{0.4\textwidth}
      \centering
      \scalebox{0.8}{\begin{tikzpicture}[scale=1]
	\begin{pgfonlayer}{nodelayer}
		\node [style=none] (0) at (-2, -2) {};
		\node [style=none] (1) at (-2, 8) {};
		\node [style=none] (2) at (0, 8) {};
		\node [style=none] (3) at (0, 0) {};
		\node [style=none] (4) at (8, 0) {};
		\node [style=none] (5) at (8, -2) {};
		\node [style=none] (6) at (4, 2) {};
		\node [style=none] (7) at (2, 4) {};
		\node [style=none] (8) at (0, 2) {};
		\node [style=none] (10) at (4, 6) {};
		\node [style=none] (11) at (4, 4) {};
		\node [style=none] (12) at (2, 2) {};
		\node [style=none] (13) at (4, 0) {};
		\node [style=none] (14) at (6, 0) {};
		\node [style=none] (15) at (4.75, 3) {};
		\node [style=none] (16) at (4.75, 5) {};
		\node [style=none] (17) at (2, 5.5) {};
		\node [style=none] (18) at (0.25, 3.75) {};
		\node [style=none] (19) at (1.25, 2.5) {$\lambda$};
		\node [style=none] (20) at (5, 0.5) {$\lambda$};
	\end{pgfonlayer}
	\begin{pgfonlayer}{edgelayer}
	\draw[step=2,color=gray,dashed] (0,0) grid (8,8);
		\draw [style=limits] (2.center)
			 to (3.center)
			 to (4.center)
			 to (5.center)
			 to (0.center)
			 to (1.center)
			 to cycle;
		\draw [style={walk_step}] (3.center) to (6.center);
		\draw [style={walk_step}] (6.center) to (11.center);
		\draw [style={walk_step}] (11.center) to (10.center);
		\draw [style={walk_step}] (10.center) to (7.center);
		\draw [style={walk_step}] (7.center) to (8.center);
		\draw [style={walk_step}] (8.center) to (12.center);
		\draw [style={walk_step}] (12.center) to (13.center);
		\draw [style={walk_step}] (13.center) to (14.center);
	\end{pgfonlayer}
\end{tikzpicture}}
      \label{fig:walk2D}
    \end{subfigure}
 \caption{The weighted model $\Gmod$ along with an example of a walk
of size $8$, total weight $\lambda^2$ and ending at
$(3,0)$}\label{fig:g2_mod}
\end{figure}

Given a finite set $\calS$ of allowed steps in $\Z^2$ and a family of
$\calW=(w_{s})_{s \in \calS}$ of non-zero weights, the combinatorial
question consists in enumerating the weighted lattice walks in $\N^2$
with steps in $\calS$. A weighted lattice walk or path of length $n$
consists of $n+1$ points whose associated translation vectors belong
to $\mathcal{S}$. Its weight is the product of the weights of all
translation vectors encountered walking the path.  To enumerate these
objects, we study the generating function \[
  Q(X,Y,t)= \sum_{i,j,n} q^{(i,j)}_n X^i Y^j t^n
\]
where $q^{(i,j)}_n$ is the sum of the weights of all walks in $\N^2$
of $n$ steps taken in $\calS$ that start at $(0,0)$ and end at
$(i,j)$.  One natural question for this class of walks is to decide
where $Q(X,Y,t)$ fits in the classical hierarchy of power series:
\[ \mbox{ algebraic } \subset D\mbox{-finite} \subset
  D\mbox{-algebraic}. \]
Here, one says that the series $Q(X,Y,t)$ is $D$-finite if it
satisfies a linear differential equation in each variable $X$,$Y$,$t$,
over $\Q(X,Y,t)$ and $D$-algebraic if it satisfies a polynomial
differential equation in each of the variables $X,Y,t$ over
$\Q(X,Y,t)$.\\

\noindent \textbf{Walks with small steps.}
For \emph{unweighted small steps} walks (that is $\calS \subset
\{-1,0,1\}^2$ and weights all equal to $1$), the 
classification of the generating function is now complete. It
required almost ten years of research and the contribution of many
mathematicians, combining a large variety of tools: elementary power
series algebra \cite{BMM}, computer algebra \cite{KauersBostan},
probability theory \cite{DenisovWachtel}, complex uniformization
\cite{KurkRasch}, Tutte invariants \cite{BBMR16} as well as
differential Galois theory \cite{DHRS}.

In \cite{BMM}, Bousquet-Mélou and Mishna associated with a model
$\calW$ a certain group $G$ of birational transformations which plays
a crucial role in the nature of $Q(X,Y,t)$. Indeed, the series
$Q(X,Y,t)$ is $D$-finite if and only if $G$, called here the
\emph{classical group of the walk}, is a finite group (see \cite{BMM,
KauersBostan, KurkRasch, MR09, DreyfusHardouinRoquesSingerGenuszero}).

When the group $G$ is finite, the algebraic nature of the generating
function is intrinsically related to the existence of certain rational
functions in $X,Y,t$ called in this paper \emph{Galois invariants} and
\emph{Galois decoupling pairs}. These notions were introduced in
\cite{BBMR16} where the authors proved that the finiteness of the
group $G$ is equivalent to the existence of non-trivial Galois
invariants (see \cite[Theorem~4.6]{BBMR16}) and found that the algebraicity
of the model is equivalent to the existence of Galois invariants and
decoupling pairs for the fraction $XY$ (see \cite[Section 4]{BBMR16}).

\begin{figure}[ht]
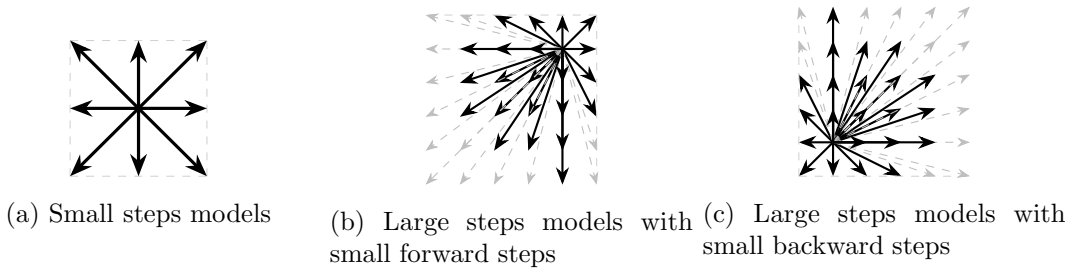

    \centering
    \begin{subfigure}[t]{0.3\textwidth}
        \ctikzfig{small_step}
        \caption{Small steps models}
        \label{fig:small_step_models}
    \end{subfigure}
    \begin{subfigure}[t]{0.3\textwidth}
        \ctikzfig{lsteprev}
        \caption{Large steps models with small forward steps}
        \label{fig:large_step_bo}
    \end{subfigure}
    \begin{subfigure}[t]{0.3\textwidth}
        \ctikzfig{lstep}
        \caption{Large steps models with small backward steps}
        \label{fig:large_step_bo_rev}
    \end{subfigure}
    \caption{Models of walks}
\end{figure}

\noindent \textbf{Walks with arbitrarily large steps.}
Compared to the case of small steps walks,
the classification of walks with arbitrarily large steps is still at its infancy.
In \cite{bostan2018counting}, Bostan, Bousquet-Mélou and Melczer lay
the foundation of the study of large steps walks. To this purpose,
they attach to any model with large steps, a graph called
\emph{the orbit of the walk} whose edges are pairs of algebraic elements over
$\Q(x,y)$.  When all the steps of the walk are small, the
\emph{ orbit of the walk} coincides with the orbit of $(x,y)$ under the action of
the group $G$ of birational transformations introduced in \cite{BMM}.

  Bostan, Bousquet-Mélou and Melczer started a thorough classification
of the $13 110$ nonequivalent models with steps in $\{1, 0, -1,
-2\}^2$ (which are instances of walks with \emph{small forward steps},
see Figure~\ref{fig:large_step_bo}). They ended up with a partial
classification of the differential nature of the associated generating
functions (see \cite[Figure~7]{bostan2018counting}).  Among the $240$
models with finite orbit, they were able to prove $D$-finitness for
all but $9$ models via orbit sums constructions or Hadamard
products. For the $12 870$ models with an infinite orbit, they were
able to prove non-$D$-finitness for all but $16$ models by exhibiting
some wild asymptotics for the associated generating functions. \\

\noindent \textbf{Content of the paper.}
 When the steps set contains at least one large step, the authors of
\cite{bostan2018counting} deplored that, within their study, the group
of the walk `` is lost, but the associated orbit survives''.  In this
paper, we show that one can generalize the notion of \emph{group of
the walk} to models with large steps as well as many objects and
results related to the small steps framework.  The novelty of our
approach lies in the use of tools from graph theory, in particular
graph homology and their combination with a Galois theoretic
approach. We list below our contributions.

\begin{itemize}
\item We attach to any model $\calW$ a group $G$, which we call the
\emph{group of the walk}. This group generalizes the `` classic''
\emph{group of the walk} in many ways.  First, $G$ is the group of
automorphisms of a certain field extension. It is generated by Galois
automorphisms and extends thereby the definition of the
\emph{classical group of the walk} as in \cite[Section 2.4]{FIM} (see
Theorem~\ref{thm:ko_gal} below). Moreover, we also prove that the
\emph{orbit of the walk} is the orbit of $(x,y)$ under the faithful
action of $G$ viewed as a group of graph automorphisms (see
Theorem~\ref{thm:group_trans}). Finally, Section~\ref{ap:geometry}
studies the geometric interpretation of the group
$G$ as group of birational transformations of a certain algebraic
curve.

\item The Galoisian structure of the \emph{group of the walk} enables
us to characterize algebraically the existence of \emph{Galois
invariants}.  To any model $\calW$, one can attach a \emph{kernel
polynomial} $\Ktld(X,Y)$ in $\C[X,Y,t]$. A pair of \emph{Galois
invariants} consists in a pair of rational fractions $(F(X),G(Y))$ in
$\C(X,t) \times \C(Y,t)$ such that
\[\Ktld(X,Y) R(X,Y) = F(X) -G(Y),\] where $R$ is a rational fraction
in $\C(X,Y,t)$ whose denominator is not divisible by $\Ktld$.  We
prove that the existence of non-trivial \emph{Galois invariants} is
equivalent to the finiteness of the group $G$, itself equivalent to
the finiteness of the orbit (see Theorem~\ref{thm:fried}). This
extends to any model of walk the result of \cite{BBMR16} for small
steps walks.  Finally, we give an explicit way of obtaining a
non-trivial pair of Galois invariants out of the data of a finite
orbit (see Section~\ref{sect:effectiveinvariants}).  We give here a geometric and Galoisian interpretation  of the 
 question of separating variables in principal bivariate
polynomial ideals, as studied in~\cite{BuchacherKauersPogudin}.

\item This Galoisian setting also sheds a new light on the notion
  of decoupling.  Given a rational fraction $H(X,Y)$
in $\C(X,Y,t)$, a \emph{Galois decoupling} for $H$ is a pair
$(F(X),G(Y))$ in $\C(X,t) \times \C(Y,t)$ such that
\[\Ktld(X,Y) R(X,Y) = H(X,Y)-F(X) -G(Y),\] where $R$ is a rational
fraction in $\C(X,Y,t)$ whose denominator is not divisible by $\Ktld$.

When the orbit of the model $\calW$ is finite, we
give a general criterion to characterize the existence of the Galois
decoupling of any rational fraction $H$, and an explicit
expression of such decoupling when it exists.
This amounts to evaluate $H$
on some well chosen linear combination of pairs of the orbit
(Theorem~\ref{thm:decoupling_orb}), obtained
  using the Galoisian structure of the orbit and graph homology.
This combination is obtained explicitely, making the procedure
constructive. Moreover, it admits an efficient implementation under a small
assumption depending only on the graph structure of the orbit (see
Section~\ref{subsect:effective_decoupling}). We checked
  this assumption on all the finite orbits for models with steps in
  in $\{-1,0,1,2\}^2$ classified in~\cite{bostan2018counting},
  and other known families of finite orbits.
This construction generalizes \cite[Theorem~4.11]{BBMR16} to the
large steps case.

As an application, we study the existence of Galois decoupling for the
function $XY$ for weighted models with steps in $\{-1,0,1,2 \}^2$. The
finite \emph{orbit-types} (which correspond to the graph structures of
the \emph{orbit of the walk} of these models) have been classified in
\cite{bostan2018counting}. For these orbit-types, we are able to give
an efficient procedure to test the existence of the Galois decoupling
of any given rational fraction, and construct it when it exists. We
applied these procedures to $XY$ for every unweighted model with steps
in $\{-1,0,1,2 \}^2$ and finite orbit (see
Proposition~\ref{prop:XYdecouple}). We also exhibit a new family of
models $\calH_n$ with large steps for which we were able to find
multiple $(a,b)$ such that $X^aY^b$ admits a Galois decoupling. This
corresponds to the counting problem for walks starting at $(a-1,b-1)$
(see Appendix~\ref{sect:otheralgmodels}).

\item Generating functions associated to models with small backward
steps (see Figure~\ref{fig:large_step_bo_rev}) satisfy a functional
equation in two catalytic variables of the form
\[\Ktld(X,Y)Q(X,Y,t) = XY +F(X) +G(Y) ,\] where $F(X)$ (resp. $G(Y)$)
involves only the section $Q(X,0,t)$ (resp. $Q(0,Y,t)$) of the
generating function.  In \cite{BBMR16} for small steps walks and
\cite{BousquetMelouThreequadrant} for walks confined in the
three-quadrant, the authors develop a strategy to prove (when it
holds) the algebraicity of the generating function. When $XY$ admits a
Galois decoupling pair and when there exist nontrivial Galoisian
invariants, they were able to obtain from the functional equation
above two functional equations in one catalytic variable each, whose
solutions are respectively the sections $Q(X,0,t)$ and
$Q(0,Y,t)$. Since solutions of polynomial equations in one catalytic
variable are known to be algebraic by \cite{BMJ}, one concludes to the
algebraicity of the generating function $Q(X,Y,t)$. Thanks to our
systematic approach to Galoisian invariants and decoupling, we apply
their strategy to prove the algebraicity of the generating function
$Q(X,Y,t)$ of the model $\Gmod$ for general $\lambda$.  In particular,
we prove that the excursion generating function $Q(0,0,t)$ of $\Gmod$
is algebraic of degree $32$ over $\Q(\lambda)(t)$. By a specialization
argument, we then conclude to the algebraicity of the excursion
generating functions of $\calG_0$ ($\lambda=0$) and $\calG_1$
($\lambda=1$). Since these excursion generating functions coincide
with those of the reversed models of $\calG_0$ and $\calG_1$, we prove
a conjecture on two models with small forward steps by Bostan,
Bousquet-Mélou and Melczer on \cite[Page 57]{bostan2018counting}.

\end{itemize}

The paper is organized as follows.
In Section \ref{subsect:alg_strat}, we present a strategy illustrated on the example of $\Gmod$ to prove
the algebraicity of models with small backward steps based on the
Galois theoretic tools developed later on in the paper.
In Section \ref{sect:orbit}, we recall the construction of the \emph{orbit of the
walk} and define the group of the walk as a group of field
automorphisms.  Section \ref{sect:ratinv} is concerned with the notion
of \emph{pairs of Galois invariants} and their properties.
In Section \ref{sect:decoupling}, we define the notion of \emph{Galois decoupling
of the pair $(x,y)$} in the orbit and prove the unconditional
existence of such a decoupling when the orbit is finite. This yields a
criterion to test the decoupling of any rational fraction, including
$XY$. We also study the implementation of our decoupling test via the
notion of \emph{level lines} of the graph of the orbit, allowing a
more effective computation. Section~\ref{ap:geometry} is for the
geometry inclined reader since it presents the Riemannian geometry
behind the large steps models.

Note that, in this paper, we consider a weighted model $\calW$ which
is entirely determined by a set of directions $\calS$ together with a
set of weights $(w_s)_{s \in \mathcal{S}}$. The weights are always
non-zero and they belong to a certain field extension of $\Q$  which is not
  necessarily algebraic, allowing the choice of  indeterminate weights. Without
loss of generality, one can assume that $\Q\left(w_s,s \in
\calS\right) \subset \C$. For ease of presentation, we consider
polynomials, rational fractions with coefficients in $\C$. However,
the reader must keep in mind that our results are valid if one replace
$\C$ by the algebraic closure of $\Q\left(w_s,s \in
\mathcal{S}\right)$.

\section{A step by step proof of algebraicity} \label{subsect:alg_strat}

In this section, we fix a weighted model $\calW$ with small backward
steps.  We explain how one can combine the approach of Bousquet-Mélou
and Jehanne on equations with one catalytic variables \cite{BMJ} and
the notion of Galois decoupling and invariants of a model to study the
algebraicity of the generating functions for models with small
backwards steps. This strategy is not yet entirely algorithmic and
follows the one developed in the small steps case in \cite[Section 5]
{BBMR16} and in \cite{BousquetMelouThreequadrant} for walks in the
three-quadrant.  We summarize its main steps in
Figure~\ref{fig:alg_strat}. In subsection
\ref{subsect:algebraicitystrategy}, we apply this strategy to prove
that the generating function of the weighted model $\Gmod$ defined in
Example \ref{ex:gessel2_alg_1} is algebraic.  Therefore, the same
holds for its excursion series. Since excursion series are preserved
under central symmetry, the excursion series of the reversed model of
$\Gmod$ is also algebraic. Thereby, we prove two of the four
conjectures of Bostan, Bousquet-Mélou and Melczer on \cite[Section
8.4.2]{bostan2018counting}. More precisely, we prove that the
excursion series $Q(0,0,t)$ of two models which are obtained by
reversing the step sets of $ \mathcal{G}_0$ and $ \mathcal{G}_1$ are
algebraic. In Appendix~\ref{sect:otheralgmodels}, we apply this
strategy to a new family of models $\calH_n$ and prove that the
generating functions counting walks starting at $(a,b)$ are algebraic
for various starting points $(a,b)$.

\subsection{Walks and functional equation in two catalytic variables}\label{subsect:walkfuncequ}

Recall that we do not  only study the number of walks of size $n$ that corresponds
to the series $Q(1,1)$.  We  record in the enumeration
the coordinates
  where these walks end, encoded in the generating function as the exponents associated
  with the variables $X$ and $Y$.
The variables $X$ and $Y$ in $Q(X,Y,t)$ are called
\emph{catalytic}, as they provide an easy way to write a functional
equation for $Q(X,Y,t)$ from the recursive description of walks: either
a walk is the trivial walk (with no steps), either one adds a step to
an existing walk, provided the new walk does not leave the quarter
plane. This is that boundary constraint which forces to consider the
final coordinates $(i,j)$ of the walk to form a functional equation.
This inductive description  yields a functional equation for the
generating function $Q(X,Y,t)$.

Thereby, we encode the model $\calW$ in two Laurent polynomials which
are the \emph{step polynomial} of the model $S(X,Y) = \sum_{(i,j) \in
\calS} w_{i,j} X^i Y^j$ and the \emph{kernel polynomial} $K(X,Y,t) = 1
- tS(X,Y)$. This Laurent polynomial can be normalized into a
polynomial $\Ktld(X,Y,t) = X^{m_x} Y^{m_y} K(X,Y,t)$ where $-m_x,-m_y$
are the smallest moves of the walk in the $X$ and $Y$-direction. By an
abuse of terminology, we also call $\Ktld$ the kernel polynomial. We
shall sometimes write $Q(X,Y),K(X,Y)$ and $ \Ktld(X,Y)$ instead of
$Q(X,Y,t), K(X,Y,t), \Ktld(X,Y,t)$ in order to lighten the
notation. We now illustrate the construction of the functional
equation for the model $\Gmod$.

\begin{exa}[The model $\Gmod$] \label{ex:gessel2_alg_1}
  Consider the weighted model
  \[
    \Gmod = \{(-1,-1), (0,1),
    (1,-1), (2,1), ((1,0), \lambda)\}
  \]
  together with its  step polynomial
  $S(X,Y) = \tfrac{1}{XY} + Y + \tfrac{X}{Y} + X^2Y + \lambda X$,
  and kernel polynomial $\Ktld(X,Y,t) = XY - t(1 + XY^2+X^2+X^3Y^2+\lambda X^2Y)$. The weight
  $\lambda$ is a nonzero complex number.

  Now, to form a functional equation,
  observe that the steps $(1,0)$, $(0,1)$ and $(2,1)$ can be concatenated to any existing walk, whereas the step $(1,-1)$ can only be concatenated to a walk
  that does not terminate on the $X$-axis, and the step
  $(-1,-1)$ can only be concatenated to a walk that does not
  terminate on the $X$-axis or the $Y$-axis.
  These conditions translate directly into the following functional equation:
  \begin{align}\label{eq:gessel2_alg_1}
    Q(X,Y) &= 1 + t Y Q(X,Y) + tX^2YQ(X,Y) + \lambda tXQ(X,Y) \nonumber \\
           &+ t \tfrac{X}{Y} \left(Q(X,Y) - Q(X,0)\right) \nonumber \\
           &+ t \tfrac{1}{XY} \left(Q(X,Y) - Q(X,0) - Q(0,Y) + Q(0,0) \right).
  \end{align}

   Note  that  we  can express
   the generating function for walks ending on
    the $X$-axis, the $Y$-axis or at $(0,0)$ as specializations of  the  generating function
    $Q(X,Y)$. For instance, the series
    $Q(X,Y)- Q(X,0)$ counts the  walks
    that do not   end on the $X$-axis.

    Grouping terms in $Q(X,Y)$ to the left-hand side and multiplying
    by $XY$ to have polynomial coefficients, we finally obtain the following
    equation for $Q(X,Y)$:
    \[
        \Ktld(X,Y)Q(X,Y) = XY - t(X^2+1)Q(X,0) - t Q(0,Y) + t Q(0,0).
    \]
\end{exa}

The general form of the functional equation satisfied by the
generating function of a weighted model might be quite complicated
\cite[Equation (11)]{bostan2018counting}.
For models with small backward steps, the functional equation satisfied by
$Q(X,Y)$ simplifies as follows:
\begin{equation} \label{eq:small_steps_backwards}
    \Ktld(X,Y) Q(X,Y) = XY + A(X) + B(Y),
  \end{equation}
where $A(X)=\Ktld(X,0)Q(X,0) +t \epsilon Q(0,0)$ and
$B(Y)=\Ktld(0,Y)Q(0,Y)$ where $\epsilon$ is $1$ if $(-1,-1)$ belongs
to $\mathcal{S}$ and $0$ otherwise.  Thus,
\eqref{eq:small_steps_backwards} only involves the sections $Q(X,0)$
and $Q(0,Y)$ which makes it easier to study.

\begin{rem}
  One may ask whether there exists a weighting of the steps set of the
model $\Gmod$ which would still yield an algebraic generating
function.
  Consider the model \[
    \calS = \{((-1,-1),\mu),((0,1),\mu),(2,1),((1,0),\lambda),(1,-1)\}\]
  which consists in adding a nonzero weight $\mu$ to its two leftmost steps.
  Consider a lattice walk on this model taking
  $a$ times the step $(-1,-1)$, $b$ times the step $(0,1)$, $c$ times
  the step $(2,1)$, $d$ times the step $(1,0)$ and $e$ times
  the step $(1,-1)$. This lattice path contributes to the generating function  $Q(X,Y,t)$ via the monomial \[
    \lambda^d\, \mu^{a+b} \, X^{2c+d+e-a} \, Y^{b+c-a-e} \, t^{a+b+c+d+e}.
  \]
  One then remarks that the knowledge of the exponents of $\lambda$, $X$, $Y$ and $t$
  completely determines the exponent of $\mu$, for \[
    a+b = -\tfrac{1}{4} d - \tfrac{1}{2} \left(2c+d+e-a\right)
    + \tfrac{1}{4} \left(b+c-a-e\right) + \tfrac{3}{4} \left(a+b+c+d+e\right).
  \]
  Thus, the series $Q(X,Y,t)$ for $\Gmod$ can be expressed
  as $Q'(X \mu^{-\tfrac{1}{2}}, Y \mu^{\tfrac{1}{4}}, t \mu^{\tfrac{3}{4}})$,
  where $Q'$ is the generating function for walks using the steps
  $\{(-1,-1),(0,1),(1,-1),(2,1),((1,0),\lambda \mu^{-\tfrac{1}{4}})\}$
  (that is $\calG_{\lambda \mu^{-\frac{1}{4}}}$).
  Thus, the weight $\mu$ is   combinatorially redundant
  when considering the full generating series $Q(X,Y)$ or the excursion
  generating series $Q(0,0)$, and this also implies that the nature
  of the model $\calS$  is only  determined the nature of $\Gmod$.
  A similar redundancy occurs when one weights with the same nonzero weight $\mu$
    the two steps going upwards in $\Gmod$, or with the same nonzero weight $\mu$
    the two steps going downwards in $\Gmod$.
  Apart from these three redundant generalizations of $\Gmod$, we were
not able to find a weighting of the steps set of $\Gmod$ leading to a
finite orbit and thereby to a potential algebraic generating function.
\end{rem}

\subsection{Algebraicity strategy}\label{subsect:algebraicitystrategy}

In \cite{BMJ}, Bousquet-Mélou and Jehanne proved the algebraicity of
power series solution of \emph{well founded} polynomial equations in
one catalytic variable. Their method has been further extended
recently to the case of systems of discrete differential equations by
Notarantonio and Yurkevich in \cite{NotaYurk}. These algebraicity
results are in fact particular cases of an older result in commutative
algebra of Popescu \cite{Popescu} but the strength of the strategy
developed in \cite{BMJ, NotaYurk} lies in the   effectiveness of their
approach.

In this subsection, we recall the algebraicity strategy developped in
\cite[Section 4]{BMJ} to deduce two polynomial equations in one
catalytic variable from the data of a polynomial equation in two
catalytic variables a decoupling pair and a pair of invariants.  We
illustrate this strategy on the model $\Gmod$. Since we alternate
general discussions and their illustration on our running example
$\Gmod$, we use $\square$ in this subsection to notify the end of the
examples.

Let $\mathbb{L}$ be a field of characteristic zero. For an unknown
bivariate function $F(u,t)$ denoted for short $F(u)$, we consider the
functional equation
\begin{equation}\label{eq:bousquetmeloujehanne} F(u) = F_0(u) + t\,
Q\left(F(u), \Delta F(u), \Delta^{(2)} F(u), \dots, \Delta^{(k)} F(u),
t, u\right),
\end{equation} where $F_0(u) \in \mathbb{L}[u]$ is given explicitly
and $\Delta$ is the \emph{discrete derivative}: $\Delta
F(u)=\frac{F(u)-F(0)}{u}$. One can easily show that the equation
\eqref{eq:bousquetmeloujehanne} has a unique solution $F(u,t)$ in
$\mathbb{L}[u][[t]]$, the ring of formal power series in $t$ with
coefficients in the ring $\mathbb{L}[u]$. Such an equation is called
\emph{well-founded}. Here is one of the main results of \cite{BMJ}.
\begin{thm}[Theorem~3 in \cite{BMJ}] \label{thm:alg_one_variable} The
formal power series $F(u,t)$ defined by
\eqref{eq:bousquetmeloujehanne} is algebraic over $\mathbb{L}(u,t)$.
\end{thm}

We shall use Theorem~\ref{thm:alg_one_variable} as a black box in
order to establish the algebraicity of power series solutions of a
polynomial equation in one catalytic variable.

In order to eliminate directly trivial algebraic models, we make the
following assumption on the step sets.  Write $-m_x$, $M_x$
(resp. $-m_y$, $M_y$) for the smallest and largest move in the $x$
direction (resp. $y$ direction) of the model $\calW$ (the $m_x$,
$M_x$, $m_y$ and $M_y$ are non-negative). Now, consider the class of
models where one of these quantities is zero. All the models in this
class are algebraic. Indeed, the corresponding models are essentially
one dimensional.  More precisely, if $M_x = 0$, one shows that a walk
based upon such a model is included in the half-line $x=0$. Similarly,
if $m_x = 0$, then the walks on this model have only the $y$
constraint. Reasoning analogously to \cite[Section 2.1]{BMM} or
\cite[Section 6]{bostan2018counting}, one proves that the series is
algebraic. Thus, we may assume from now on that none of these
parameters are zero so that \textbf{ $S(X,Y)$ is not univariate.}
Moreover, analogously to \cite[\S 8.1]{bostan2018counting}, we exclude
upper diagonal models, that is, models for which $(i,j) \in
\mathcal{S}$ satisfy $j \geq i$ as well as their symmetrical, the
lower diagonal models. Indeed, these models are automatically
algebraic.

This assumption being made, the series $Q(X,Y)$ satisfies naturally an
equation with two catalytic variables, and therefore does not fall
directly into the conditions of
Theorem~\ref{thm:alg_one_variable}. However, the functional equation
\eqref{eq:small_steps_backwards} implies that the generating function
$Q(X,Y)$ is algebraic over $\C(X,Y,t)$ if and only if the series
$A(X)$ and $B(Y)$ are algebraic over $\C(X,t)$ and $\C(Y,t)$
respectively.  Therefore, we set ourselves to find two well founded
polynomial equations with one catalytic variable: one for $A(X)$ and
the other for $B(Y)$.

In order to produce these two equations from the functional equation
\eqref{eq:small_steps_backwards}, we now present a method inspired by
Tutte \cite{TutteSurvey} which was further adapted by Bernardi,
Bousquet-Mélou and Raschel in the context of small steps walks
\cite{BBMR16} and by Bousquet-Mélou in the context of three quadrant
walks \cite{BousquetMelouThreequadrant}.  We reproduce here the method
of \cite{BousquetMelouThreequadrant} which relies on suitable notion
of $t$-invariants and an \emph{Invariant Lemma} for multivariate power
series. The strategy developed in \cite{BousquetMelouThreequadrant} is
an adaptation for formal power series of the approach already
introduced in Section 4.3 in \cite{BBMR16}.

\begin{defi}
    We denote by $\C(X,Y)((t))$ the field of Laurent series in $t$
with coefficients in the field $\C(X,Y)$. The subring
$\Cmul(X,Y)((t))$ of $\C(X,Y)((t))$ formed by the series of the
form
\[ H(X,Y,t) = \sum_t \frac{p_n(X,Y)}{a_n(X) b_n(Y)} t^n, \]
where $p_n(X,Y) \in \C[X,Y]$, $a_n(X) \in \C[X]$ and $b_n(Y) \in \C[Y]$.
\end{defi}

\begin{defi}[Definition~2.4 in \cite{BousquetMelouThreequadrant}]
    \label{defi:poles_of_bounded_order}
    Let $H(X,Y,t)$ be a Laurent series in $\Cmul(X,Y)((t))$. The series $H$ is
    said to have \emph{poles of bounded order at $0$} if the collection of its
    coefficients (in the $t$-expansion) have poles of bounded order at $X=0$
    and $Y=0$. In other words, this means that, for some natural numbers $m$ and
    $n$, the coefficients in $t$ of the series $X^m Y^n H(X,Y)$ have no pole at
    $X=0$ nor at $Y=0$.
\end{defi}

Given a model  $\calW$, one can use the notion of poles of bounded order at zero to construct an equivalence
relation in the ring $\Cmul(X,Y)((t))$. To this purpose, we slightly adapt
Definition~2.5 in \cite{BousquetMelouThreequadrant} to encompass the large step
case. Moreover, in the following definition, we consider  division by $\Ktld$
and not by $K$ as in \cite{BousquetMelouThreequadrant} but one easily checks
that  Definition~\ref{defi:analytic_inv} below and  Definition~2.3 in
\cite{BousquetMelouThreequadrant} coincide.

\begin{defi}[$t$-equivalence] \label{defi:analytic_equiv}
    Let $F(X,Y)$ and $G(X,Y)$ be two Laurent series in $\Cmul(X,Y)((t))$.
    We say that these series are \emph{$t$-equivalent},
    and we write $F(X,Y) \equiv G(X,Y)$ if the series
    $\frac{F(X,Y) - G(X,Y)}{\Ktld(X,Y)}$ has poles of bounded order at 0.
\end{defi}

The $t$-equivalence is compatible with the ring  operations on Laurent series applied pairwise as stated below.
\begin{prop}[Lemma~2.5 in \cite{BousquetMelouThreequadrant}] \label{prop:analytic_op}
    If $A(X,Y) \equiv B(X,Y)$ and $A'(X,Y) \equiv B'(X,Y)$, then
    $A(X,Y) + B(X,Y) \equiv A'(X,Y) + B'(X,Y)$ and
    $A(X,Y) B(X,Y) \equiv A'(X,Y) B'(X,Y)$.
\end{prop}

The notion of $t$-equivalence allows us  to define  the notion of
$t$-invariants as follows.

\begin{defi}[$t$-Invariants (Definition~2.3 in \cite{BousquetMelouThreequadrant})] \label{defi:analytic_inv}
    Let $I(X)$ and $J(Y)$ be two Laurent series in $t$  with coefficients lying
    respectively in $\C(X)$ and $\C(Y)$.
    If $I(X) \equiv J(Y)$, then the pair $(I(X), J(Y))$ is said to be
    a \emph{pair of  $t$-invariants} (with respect to the model $\calW$).
\end{defi}

By Proposition~\ref{prop:analytic_op}, pairs of $t$-invariants are
also preserved under sum and product applied pairwise. We now state
the main result on $t$-invariants
\cite[Lemma~2.6]{BousquetMelouThreequadrant} whose proof originally
for small steps models passes directly to the large steps
context.\footnote{In \cite{BousquetMelouThreequadrant}, Lemma~2.6
requires that the coefficients in the $t$-expansion of $\frac{I(X) -
J(Y)}{K(X,Y)}$ vanish at $X=0$ and $Y=0$.  This is equivalent to the
condition stated in Lemma~\ref{lem:inv_Lemma}.}:

\begin{lem}[Invariant Lemma] \label{lem:inv_Lemma}
    Let $(I(X), J(Y))$ be a pair of  $t$-invariants. If
    the coefficients in the $t$-expansion  of
    $\frac{I(X) - J(Y)}{\Ktld(X,Y)}$ have no pole
    at $X=0$ nor $Y=0$, then there exists a Laurent series $A(t)$
    with coefficients in $\C$ such that $I(X) = J(Y) = A(t)$.
  \end{lem}

Note that the equations $I(X) = A(t)$ and $J(Y) = A(t)$ involve only  one catalytic
variable. In other words, the  Invariant Lemma allows us to produce
nontrivial equations with one catalytic variable from one pair of
$t$-invariants satisfying a certain analytic regularity.

Still assuming that the negative steps are small,we can now try to combine the
notion of $t$-invariants and the Invariant Lemma with the functional
equation satisfied by $Q(X,Y)$ in order to obtain two equations in one
catalytic variable for $Q(X,0,t)$ and $Q(0,Y,t)$.  

First, we find a pair of $t$-invariants which involves the specializations $Q(X,0)$
and $Q(0,Y)$ of $Q(X,Y)$.
One way to obtain such a pair of $t$-invariants is by looking at
\eqref{eq:small_steps_backwards}, namely:
\begin{equation}\label{eq:funceqwrittenasdecouplingsections}
    \Ktld(X,Y) Q(X,Y) = XY + A(X) + B(Y).
\end{equation}
Assume that there exist some  fractions $F(X)$  in $\C(X,t)$, $G(Y)$
in $\C(Y,t)$, and $H(X,Y)$ in $\C(X,Y,t)$ having poles of bounded order at 0
such that that $XY$ can be written as
\[
    XY = F(X) + G(Y) + \Ktld(X,Y) H(X,Y).
\]
We call such a relation a \emph{$t$-decoupling} of $XY$.
Combining the  $t$-decoupling of $XY$ with
\eqref{eq:funceqwrittenasdecouplingsections}, one obtains the following
rewriting
\[
    \Ktld(X,Y) \left( Q(X,Y) - H(X,Y) \right) = \left(F(X) + A(X)\right)
    + \left(G(Y) + B(Y)\right).
\]
Note now that the right-hand side has separated variables from the
$t$-decoupling of $XY$. Since $Q(X,Y)$ is a generating function for
walks in the quarter plane, the coefficients of its $t$-expansion are
polynomials in $\C[X,Y]$ (the coefficient of $t^n$ is $ \sum_{i,j \ge
0} q^{(i,j)}_n X^i Y^j$), so the power series $Q(X,Y)$ has poles of
bounded order at $0$. By assumption on $H(X,Y)$, this is also the case
for the series $Q(X,Y) - H(X,Y)$.  Therefore, $(I_1(X), J_1(Y)) =
(F(X)+A(X), -G(Y) - B(Y))$ is a pair of $t$-invariants. It is
noteworthy that this pair of $t$-invariants involves the sections
$Q(X,0)$ and $Q(0,Y)$.

We must note that the writing of $XY$ as the sum of two univariate
fractions modulo $\Ktld$ was the only condition to the existence of
the pair $(I_1,J_1)$. In Section \ref{sect:decoupling}, we introduce
the notion of \emph{Galois decoupling } of $XY$ which is weaker though
easier to test than the notion of $t$-decoupling.  A criterion to test
the existence of a Galois decoupling for $XY$ or, more generally, for
any rational fraction in $\Qb(X,Y)$ and the computation of a Galois
decoupling pair if it exists are among the main results of this paper,
and are covered in full generality in Section
\ref{sect:decoupling}. Provided the orbit of the walk defined in
Section \ref{sect:orbit} is finite, our Galois decoupling procedure is
entirely algorithmic. Thus, one can search for a $t$-decoupling of
$XY$ by first looking for a Galois decoupling and then by checking if
this Galois decoupling is a $t$-decoupling. We now illustrate this
step on the model $\Gmod$:
\begin{exa}[The model $\Gmod$] \label{ex:gessel2_alg_2}

  Recall that the functional equation \eqref{eq:funceqwrittenasdecouplingsections}
  obtained for $\Gmod$ is:
    \[
         \Ktld(X,Y)Q(X,Y) = XY  +A(X) +B(Y),
      \]
    with $A(X)=  - t(X^2+1)Q(X,0)  + t  Q(0,0)$ and $B(Y)= -tQ(0,Y)$.
    One can check that $XY$ admits a $t$-decoupling of the following form:
    \[
      XY = - \frac{3 \lambda X^2 t   -   \lambda t  - 4 X}{4 t (X^2 + 1) }
      + \frac{- \lambda Y - 4}{4 Y} -  \frac{\Ktld(X,Y)}{(X^2 + 1 )Y t} .
\]

    Combining this identity with the functional equation, one obtains the
    following pair of $t$-invariants:
    \[
      (I_1(X),J_1(Y)) = \left(
        \frac{3 \lambda  t \,X^{2}-\lambda  t -4 X}{-4 t \,X^{2}-4 t}-t \left(X^{2}+ 1\right) Q \! \left(X , 0\right)+t   Q \! \left(0, 0\right),
        t   Q \! \left(0, Y\right)+\frac{\lambda  Y +4}{4 Y}\right).  \] 
    \exqed
\end{exa}

A priori, the pair of $t$-invariants $(I_1(X), J_1(Y))$ that can be
obtained through the combination of the functional equation and a
decoupling equation does not satisfy the conditions of
Lemma~\ref{lem:inv_Lemma}, as the coefficients of the $t$-expansion of
$\frac{I_1(X) - J_1(Y)}{\Ktld(X,Y)}$ might have poles at $0$. In order
to remove these poles, we want to combine the pair $(I_1(X), J_1(Y))$
with a second pair of $t$-invariants $(I_2(X), J_2(Y))$ by means of
Proposition~\ref{prop:analytic_op}, where $I_2(X)$ and $J_2(Y)$ will
be assumed to be respectively in $\C(X,t)$ and $\C(Y,t)$. In order to
obtain this second pair of $t$-invariants, we rely once again on a
weaker notion of invariants: the \emph{Galois invariants} which are
introduced in Section \ref{sect:ratinv}. Theorem~\ref{thm:fried} below
shows that the existence of a non-constant pair of Galois invariants
is equivalent to the finitness of the \emph{orbit of the walk}.
Currently, the pole elimination between the two pairs of
$t$-invariants requires a case by case treatment. We detail it for our
running example $\Gmod$.

\begin{exa}[The model $\Gmod$] \label{ex:gessel2_alg_3}
  The pair $(I_2(X),J_2(Y))$ below is a pair of $t$-invariants for $\Gmod$:
  \begin{align*}
    (I_2,J_2) = \left(\frac{\left(-\lambda^{2} \,X^{3}- \,X^{4}-X^{6}+ X^{2}+1\right) t^{2}-X^{2} \lambda  \left(X^{2}-1 \right) t +X^{3}}{t^{2} X \left(X^{2}+ 1\right)^{2}},
    \frac{-  t \,Y^{4}+\lambda  t Y +Y^{3}+t}{Y^{2} t}\right). 
  \end{align*}
  Analogously to the $t$-decoupling, we first search for a pair of
  Galois invariants, which amounts to use the semi-algorithm presented
  in Section \ref{sect:ratinv}, and then check that this pair is a pair
  of $t$-invariants.

  As we now have two pairs of $t$-invariants $P_1=(I_1(X), J_1(Y))$
  and $P_2 = (I_2(X),J_2(Y))$, we perform some algebraic combinations
  between them in order to eliminate their poles. To lighten notation,
  we write the component-wise operations on the pairs $P_i$ of
  $t$-invariants. Computations can be checked in the joint Maple
  worksheet (also on this \href{https://www.labri.fr/perso/pbonnet/}{webpage}).

  Consider the Taylor expansions of the first coordinates:
  \begin{align*}
    I_1(X) &= \frac{\lambda}{4}+O\! \left(X \right), \\
    I_2(X) &=  X^{-1}+O\! \left(X \right).
  \end{align*}
  Out of these two pairs  of $t$-invariants, we first produce a third  pair of $t$-invariants
  without a pole at $X=0$ as follows:
  \[
    P_3=(I_3,J_3) := P_2 \left( P_1 - \frac{\lambda}{4} \right).
  \]
  The first coordinate of the pairs $P_1$ and $P_3$ do not have a pole
at $X=0$.  The Taylor expansion of their second coordinates $J_1(Y)$
and $J_3(Y)$ at $Y=0$ is as follows:
  \begin{align*}
    J_3(Y) &= Y^{-3}+\left(t Q \! \left(0,0\right)+\lambda \right) Y^{-2}+ t \left(Q \! \left(0, 0\right) \lambda +  \frac{\partial^2 Q}{\partial Y^2} \left(0, 0 \right)\right) Y^{-1}+O\! \left(Y^{0}\right),\\
    J_1(Y) &= Y^{-1}+O\! \left(Y^{0}\right).
  \end{align*}
  In order to produce a pair of $t$-invariants satisfying the assumption of the Invariant Lemma, we need to combine $P_1$ and $P_3$ in order
  to eliminate the pole at $Y=0$. Note that, since the first coordinate
  of $P_1$ and $P_3$ have no pole at zero, the first coordinate of any
  sum or product between these two pairs  have no pole at $X=0$.  Using the simple pole at $Y=0$ of $J_3$, we produce a new pair $P_4$ whose coordinates have no pole at $X$ and $Y$ equal zero by setting
  \[
    P_4=(I_4,J_4):= P_3 -P_1^{3}+\left(2 t Q \! \left(0, 0\right)-\frac{\lambda}{4}\right) P_1^{2}+\left(2   t \frac{\partial^2 Q}{\partial Y^2} \left(0, 0 \right)-t^{2} Q \! \left(0, 0\right)^{2}+\frac{5 \lambda^{2}}{16}\right) P_1.
  \]

  It remains to  check that $\frac{I_4(X)-J_4(Y)}{\Ktld(X,Y)}$
  has no poles at $X=0$ and $Y=0$. This is done in the joint Maple worksheet
  (also on this \href{https://www.labri.fr/perso/pbonnet/}{webpage}). Therefore, the Invariant
  Lemma  yields the existence of a series
  $C(t)$ in $\C((t))$ such that $I_4(X) = C(t)$ and $J_4(Y) = C(t)$. \exqed
\end{exa}

Once we have found a pair of $t$-invariants
satisfying the conditions of
Lemma~\ref{lem:inv_Lemma}, we end up with two nontrivial polynomial equations
in one catalytic variable involving the sections $Q(X,0)$ and $Q(0,Y)$.
If these equations
are well-founded, then Theorem~\ref{thm:alg_one_variable} allows us to
conclude that the
series $Q(X,0)$ and $Q(0,Y)$ are algebraic
over  $\C(X,t)$ and $\C(Y,t)$ respectively, and therefore that $Q(X,Y)$
is algebraic over $\C(X,Y,t)$.

\begin{exa}[The model $\Gmod$] \label{ex:gessel2_alg_4}
  The value of $C(t)$ can be deduced from the values of $Q(0,Y)$ and
  its derivatives at $0$ by looking at the Taylor expansion of $J_4(Y)$
  at $Y=0$. The verification that the polynomial equations $I_4(X)=C(t)$
  and $J_4(Y)= C(t)$ are well-founded is done in the Maple worksheet
  (also on this \href{https://www.labri.fr/perso/pbonnet/}{webpage}).
  We only give here the form of the well-founded equation for
  $F(Y):=Q(0,Y)$:

  \begin{align}\label{eq:polyeqonecatavar}
    F(Y) &= 1 + t \left(
           t^{2} Y F \! \left(Y\right)\left({\Delta^{(1)}F} \! \left(Y
           \right)\right)^{2} + \lambda t F \! \left(Y\right){\Delta^{(1)} F} \!
           \left(Y \right)+t \left({\Delta^{(1)}F} \! \left(Y \right)\right)^{2}
           \right. \\ &\left. +2 t F \! \left(Y\right){\Delta^{(2)}F} \! \left(Y
                        \right) +Y F \! \left(Y \right)+\lambda {\Delta^{(2)}F} \! \left(Y
                        \right) +2 {\Delta^{(3)}F} \! \left(Y \right) \right). \nonumber
  \end{align}

  Theorem~\ref{thm:alg_one_variable} with $\mathbb{L}=\Q(\lambda)$
  implies that the generating function of the weighted model $\Gmod$ is
  algebraic over $\Q(\lambda)(X,Y,t)$. Moreover, one can show that, at
  any step of our reasoning, one may have taken the weight $\lambda$ to
  be zero.  In particular, the generating function of the model
  $\calG_0$ is algebraic. Thus, the excursion generating functions
  $Q(0,0)$ of the reverse models of $\calG_0$ and $\calG_1$ are
  algebraic over $\Q(t)$. In Appendix~\ref{sect:bmj}, we apply the
  method of Bousquet-Mélou and Jehanne to the polynomial equation
  \eqref{eq:polyeqonecatavar} to find an explicit minimal polynomial of
  degree $32$ over $\Q(\lambda,t)$ for $Q(0,0)$ of the model $\Gmod$.
  \exqed
\end{exa}

For unweighted small steps models, the results of \cite{BMM,
KauersBostan, KurkRasch, DreyfusHardouinRoquesSingerGenuszero,
MelcMish} show that the generating function is algebraic in the
variables $X$ and $Y$ if and only if the model admits some non-trivial
Galois invariants and $XY$ has a Galois decoupling.  For weighted
models with small steps, \cite[Corollary 4.2]{dreyfus2019differential}
and \cite[Theorem 4.6 and Theorem 4.11]{BBMR16} imply that the
existence of non-trivial Galois invariants and of a Galois decoupling
pair for $XY$ yield the algebraicity of the generating functions. We
conjecture that the reverse implication is also true yielding an
equivalence which should also be valid in the large steps case. The
general strategy we used in this section is summarized in Figure
\ref{fig:alg_strat} and motivates the above conjecture. It is the
first attempt at finding uniform proofs for the algebraicity of
generating functions of large steps models.

 The strategy detailed above is entirely algorithmic, except for the
fact that Galois invariants and decoupling are $t$-invariants and
$t$-decoupling and that they yield polynomial equations in one
catalytic variable satisfying the conditions of Theorem
\ref{thm:alg_one_variable}. Nonetheless, we think that this last step
could be made constructive via for instance the generalization of the
notion of weak invariants \cite[Section 5.2]{BBMR16} to the large
steps framework.  The rest of the paper is devoted to the systematic
and algorithmic study of the notions of pairs of Galois invariants and
decoupling.

\begin{figure}[h!]
    \scalebox{0.6}{\begin{tikzpicture}[node distance=1.5cm,scale=0.5]

\node (decpl) [decision] {Galois decoupling of $XY$};
\node (funeq1) [decision, above of=decpl] {Functional equation for $Q(X,Y)$};
\node (link) [style=none,right of=funeq1,xshift=3cm] {};
\node (inv1) [decision, right of=link, xshift=5cm] {Pair of $t$-invariants involving the functions $Q(X,0)$
and $Q(0,Y)$};
\node (inv2) [decision, below of=inv1] {Pair of Galois invariants};

\node (inv3) [decision, right of=inv1, xshift=10cm] {Pair of $t$-invariants without
poles};

\node (funeq2) [decision, below of=inv3] {Algebraic equations for $Q(X,0)$ and $Q(0,Y)$};
\node (funeq3) [decision, below of=funeq2] {Algebraic equation for $Q(X,Y)$};

\draw [in=180, out=0, looseness=1.50] (decpl) to (link);
\draw [arrow] (funeq1) -- (inv1);
\draw [arrow] (inv1) -- (inv3);
\draw [in=180,out=0, looseness=1.50] (inv2) to (inv3);
\draw [arrow] (inv3) -- (funeq2);
\draw [arrow] (funeq2) -- (funeq3);

\end{tikzpicture}}
    \caption{Summary of the strategy for proving algebraicity}
 \label{fig:alg_strat}
\end{figure}
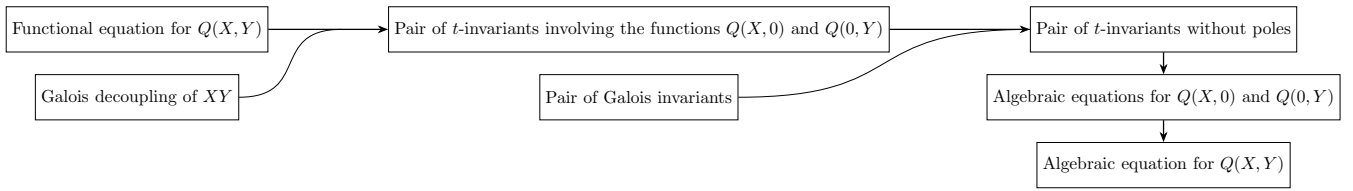

\section{The orbit of the walk and its Galoisian structure} \label{sect:orbit}

In the context of small steps models, the \emph{group of the walk}
(which we qualify \emph{classic} in this paper for disambiguation) has
been initially introduced in  \cite[Section 3]{BMM}. It is the group generated by two birational
involutions $\Phi$ and $\Psi$ of $\C \times \C$  defined as follows.
Assuming that the model has at least a
negative and a positive $X$ and $Y$-steps, one writes its step
polynomial $S(X,Y)=\sum_{ (i,j) \in \mathcal{S}}w_{(i,j)} X^i Y^j$ as
\[ S(X,Y)= A_{-1}(X)\frac{1}{Y} +A_0(X) + A_1(X)Y=B_{-1}(Y)\frac{1}{X}
+B_0(X) + B_1(Y)X,\] where the $A_i$ and $B_i$'s are Laurent
polynomials.  The birational transformations $\Phi$ and $\Psi$ are
then defined as
\[ \Phi:(x,y) \mapsto \left( \frac{B_{-1}(y)}{x B_1(y)},y \right)
\mbox{ and } \Psi:(x,y) \mapsto \left(x, \frac{A_{-1}(x)}{y A_1(x)}
\right). \]

 When the classic group of the walk is infinite, its action on the
variables $X$ and $Y$ produces an infinite amount of singularities for
the generating function $Q(X,Y)$ proving that the series is
not D-finite (see \cite{MelcMish} or \cite{KurkRasch} for
instance). When the group of the walk is finite, one can describe in
certain cases the generating function as a diagonal of a
rational function, called the (alternating) orbit sum.  To such a
group, one can attach a graph, called the \emph{orbit}, whose
vertices are the orbit in $\C(x,y)^2$ of the pair $(x, y)$ under the
action of $\Phi$ and $\Psi$ and whose edges correspond
to the action of $\Phi$ and $\Psi$ (see \cite[Section 3]{BMM}).

In \cite{bostan2018counting}, the authors generalized the notion of
the orbit of the walk to arbitrary large steps models but did not
attempt to find a group of transformations which generates this
orbit. In this section, we show how one can associate to a weighted
model $\calW$ a group, called in this paper \emph{the group of the
walk}, which is generated by Galois automorphisms of two field
extensions.  In this section, we prove that the group of the walk acts
faithfully and transitively on the orbit analogously to the classic
group.  When the orbit is finite, this group is itself presented as a
Galois group. We interpret in the next two sections the notions of
 invariants and decoupling in this Galoisian framework. Moreover,
for finite orbits, one can interpret the group of the walk as
a group of automorphisms of an algebraic curve (see Appendix~
\ref{ap:geometry}). This point of view generalizes the notion of the
classic group of the walk in the small step case used in
\cite{KurkRasch, DHRS,DreyfusHardouinRoquesSingerGenuszero}.

From now on, we fix $\calW$ a weighted model, and we assume that the
step polynomial $S(X,Y)$ is not univariate, which is the case when
considering models with both positive and negative steps in each
direction as in Section \ref{subsect:walkfuncequ}.  In order to
distinguish the coordinates of the orbit from the coordinates of
$X,Y,t$ of the functional equation in Section
\ref{subsect:walkfuncequ}, we introduce two new variables $x$ and $y$
that are taken algebraically independent over $\C$. We also denote by
$k$ the field $\C(S(x,y))$.  As $x$, $y$ and $S(x,y)$ satisfy by
definition the polynomial relation $\Ktld(x,y,\invS) = 0$, the
condition that $S$ is not univariate implies the following lemma.

\begin{lem} \label{lem:alg_rel_xyS}
  The variables $x$, $y$ and the polynomial
  $S(x,y)$ satisfy the following relations:
  \begin{enumerate}
  \item $x$ and $S(x,y)$ are algebraically independent over $\C$, and
    so are $y$ and $S(x,y)$,
  \item $x$ is algebraic over $k(y)$ and $y$ is algebraic over
    $k(x)$.
  \end{enumerate}
\end{lem}

The orbit as well as the associated group, Galois invariants
and Galois decoupling pairs are constructed for $S(X,Y)$ arising from
a model of walk.  These constructions should pass directly to the case
where  \textbf{$S(X,Y)$ is an arbitrary  non-univariate rational fraction  in $\C(X,Y)$} by letting
$\Ktld(X,Y,t)$ be $(1-tS(X,Y))Q(X,Y)$ with $S=\frac{P}{Q}$ for $P,Q$
two relatively prime polynomial in $\C[X,Y]$.

In Section \ref{subsect:theorbit}, we recall the definition of the
orbit of a model $\calW$ with large steps. We give it a Galois
structure in Section \ref{subsect:galoisetxorbit}.  In Section
\ref{sect:grouporbit}, we define the group of the walk and prove that
it acts faithfully and transitively by graph automorphisms on the
orbit. Finally, we investigate the evaluation of fractions in
$\C(X,Y,t)$ on the orbit.

\subsection{The orbit}\label{subsect:theorbit}

We recall below the definition of the orbit  introduced in  \cite[Section 3]{bostan2018counting},
and we also fix once and for all an algebraic closure $\K$ of $\C(x,y)$.

\begin{defi}[Definition~3.1 in \cite{bostan2018counting}]
  Let $(u,v)$ and $(u',v')$ be in $\K \times \K$.

  If $u = u'$ and $S(u,v) =
  S(u',v')$, then the pairs $(u,v)$ and $(u',v')$ are called \emph{$x$-adjacent},
  and we  write $(u,v) \sim^x (u',v')$.
  Similarly, if $v = v'$ and $S(u,v) =
  S(u',v')$, then the pairs $(u,v)$ and $(u',v')$ are called \emph{$y$-adjacent},
  and we write $(u,v) \sim^y (u',v')$.
  Both relations are equivalence relations on $\K \times
  \K$.

  If the pairs $(u,v)$ and $(u',v')$ are either $x$-adjacent or
  $y$-adjacent, they are called \emph{adjacent}, and we write
  $(u,v) \sim (u',v')$.  Finally, denoting by $\sim^*$ the reflexive transitive
  closure of $\sim$, the \emph{orbit of the walk}, denoted by $\cO$,
  is the equivalence class of the pair $(x,y)$ under the relation $\sim^*$.
\end{defi}

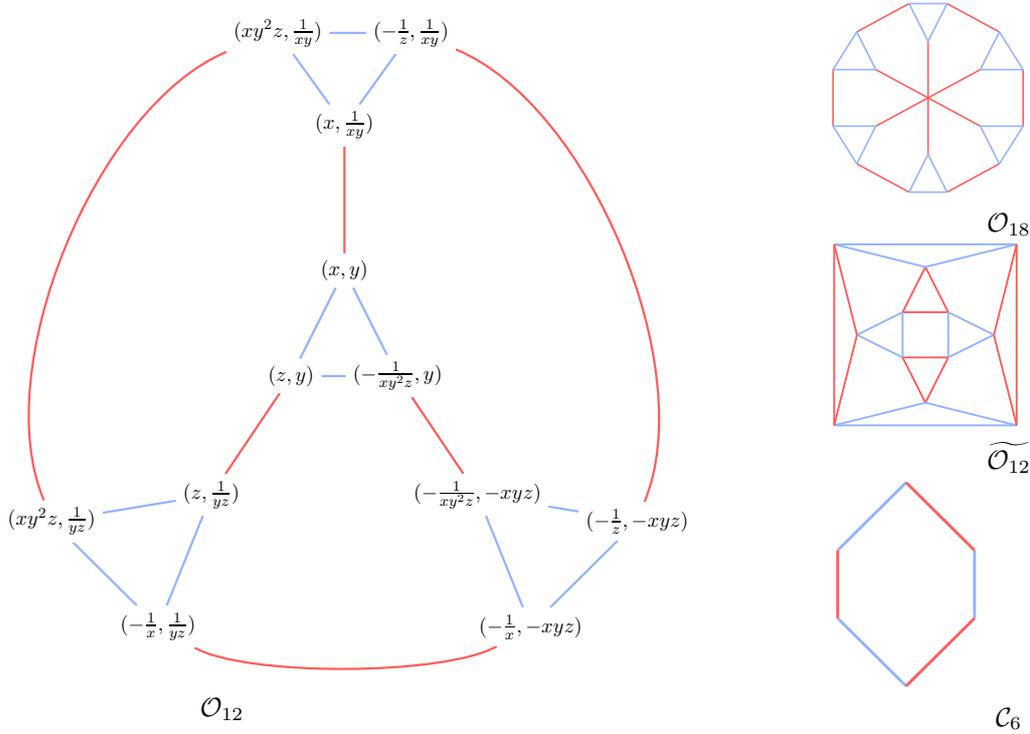
\begin{figure}[h!]
  \begin{subfigure}[c]{0.4\linewidth}
    \scalebox{0.7}{\begin{tikzpicture}[scale=2]
	\begin{pgfonlayer}{nodelayer}
		\node [style=paire] (0) at (0, 4) {$(x,\frac{1}{xy})$};
		\node [style=paire] (1) at (-1.25, 5.75) {$(xy^2z,\frac{1}{xy})$};
		\node [style=paire] (2) at (1.25, 5.75) {$(-\frac{1}{z},\frac{1}{xy})$};
		\node [style=paire] (3) at (0, 1.25) {$(x,y)$};
		\node [style=paire] (4) at (-1, -0.75) {$(z,y)$};
		\node [style=paire] (5) at (1, -0.75) {$(-\frac{1}{xy^2z},y)$};
		\node [style=paire] (6) at (-2.5, -3) {$(z,\frac{1}{yz})$};
		\node [style=paire] (7) at (-5.5, -3.5) {$(xy^2z,\frac{1}{yz})$};
		\node [style=paire] (8) at (-3.5, -5.5) {$(-\frac{1}{x},\frac{1}{yz})$};
		\node [style=paire] (9) at (2.5, -3) {$(-\frac{1}{xy^2z},-xyz)$};
		\node [style=paire] (10) at (3.5, -5.5) {$(-\frac{1}{x},-xyz)$};
		\node [style=paire] (11) at (5.5, -3.5) {$(-\frac{1}{z},-xyz$)};
	\end{pgfonlayer}
	\begin{pgfonlayer}{edgelayer}
		\draw [style=usimx] (4) to (3);
		\draw [style=usimx] (3) to (5);
		\draw [style=usimx] (5) to (4);
		\draw [style=usimx] (1) to (0);
		\draw [style=usimx] (0) to (2);
		\draw [style=usimx] (2) to (1);
		\draw [style=usimx] (6) to (7);
		\draw [style=usimx] (7) to (8);
		\draw [style=usimx] (8) to (6);
		\draw [style=usimx] (9) to (10);
		\draw [style=usimx] (10) to (11);
		\draw [style=usimx] (11) to (9);
		\draw [style=usimy] (4) to (6);
		\draw [style=usimy] (5) to (9);
		\draw [style=usimy, bend right, looseness=0.50] (8) to (10);
		\draw [style=usimy] (3) to (0);
		\draw [style=usimy, bend right=45, looseness=0.75] (1) to (7);
		\draw [style=usimy, bend left=45, looseness=0.75] (2) to (11);
	\end{pgfonlayer}
\end{tikzpicture}}
    \caption*{$\cO_{12}$}
    \label{fig:o12}
  \end{subfigure}
  \hfill
  \begin{subfigure}[c]{0.3\linewidth}
    \scalebox{0.5}{\begin{tikzpicture}
	\begin{pgfonlayer}{nodelayer}
		\node [style=none] (0) at (0, 3) {};
		\node [style=none] (1) at (1, 5) {};
		\node [style=none] (2) at (-1, 5) {};
		\node [style=none] (3) at (-2.75, 1.5) {};
		\node [style=none] (4) at (-3.75, 3.5) {};
		\node [style=none] (5) at (-5, 1.5) {};
		\node [style=none] (6) at (-2.75, -1.5) {};
		\node [style=none] (7) at (-5, -1.5) {};
		\node [style=none] (8) at (-3.75, -3.5) {};
		\node [style=none] (9) at (-1, -5) {};
		\node [style=none] (10) at (1, -5) {};
		\node [style=none] (11) at (0, -3) {};
		\node [style=none] (12) at (2.75, -1.5) {};
		\node [style=none] (13) at (3.75, -3.5) {};
		\node [style=none] (14) at (5, -1.5) {};
		\node [style=none] (15) at (2.75, 1.5) {};
		\node [style=none] (16) at (3.75, 3.5) {};
		\node [style=none] (17) at (5, 1.5) {};
	\end{pgfonlayer}
	\begin{pgfonlayer}{edgelayer}
		\draw [style=usimy] (0.center) to (11.center);
		\draw [style=usimy] (3.center) to (12.center);
		\draw [style=usimy] (6.center) to (15.center);
		\draw [style=usimy] (2.center) to (4.center);
		\draw [style=usimy] (5.center) to (7.center);
		\draw [style=usimy] (8.center) to (9.center);
		\draw [style=usimy] (13.center) to (10.center);
		\draw [style=usimy] (14.center) to (17.center);
		\draw [style=usimy] (16.center) to (1.center);
		\draw [style=usimx] (1.center) to (0.center);
		\draw [style=usimx] (0.center) to (2.center);
		\draw [style=usimx] (2.center) to (1.center);
		\draw [style=usimx] (4.center) to (5.center);
		\draw [style=usimx] (5.center) to (3.center);
		\draw [style=usimx] (3.center) to (4.center);
		\draw [style=usimx] (16.center) to (15.center);
		\draw [style=usimx] (15.center) to (17.center);
		\draw [style=usimx] (17.center) to (16.center);
		\draw [style=usimx] (14.center) to (12.center);
		\draw [style=usimx] (12.center) to (13.center);
		\draw [style=usimx] (13.center) to (14.center);
		\draw [style=usimx] (11.center) to (9.center);
		\draw [style=usimx] (9.center) to (10.center);
		\draw [style=usimx] (10.center) to (11.center);
		\draw [style=usimx] (6.center) to (7.center);
		\draw [style=usimx] (7.center) to (8.center);
		\draw [style=usimx] (8.center) to (6.center);
	\end{pgfonlayer}
\end{tikzpicture}}
    \caption*{$\cO_{18}$}
    \scalebox{0.6}{\begin{tikzpicture}
	\begin{pgfonlayer}{nodelayer}
		\node [style=none] (0) at (0, 3) {};
		\node [style=none] (1) at (-1, 1) {};
		\node [style=none] (2) at (1, 1) {};
		\node [style=none] (3) at (-4, 4) {};
		\node [style=none] (4) at (4, 4) {};
		\node [style=none] (5) at (-3, 0) {};
		\node [style=none] (6) at (-1, -1) {};
		\node [style=none] (7) at (1, -1) {};
		\node [style=none] (8) at (3, 0) {};
		\node [style=none] (9) at (0, -3) {};
		\node [style=none] (11) at (-4, -4) {};
		\node [style=none] (12) at (4, -4) {};
	\end{pgfonlayer}
	\begin{pgfonlayer}{edgelayer}
		\draw [style=usimy] (0.center) to (1.center);
		\draw [style=usimy] (1.center) to (2.center);
		\draw [style=usimy] (2.center) to (0.center);
		\draw [style=usimy] (7.center) to (6.center);
		\draw [style=usimy] (6.center) to (9.center);
		\draw [style=usimy] (9.center) to (7.center);
		\draw [style=usimx] (2.center) to (7.center);
		\draw [style=usimx] (7.center) to (8.center);
		\draw [style=usimx] (8.center) to (2.center);
		\draw [style=usimx] (5.center) to (1.center);
		\draw [style=usimx] (1.center) to (6.center);
		\draw [style=usimx] (6.center) to (5.center);
		\draw [style=usimx] (0.center) to (3.center);
		\draw [style=usimx] (0.center) to (4.center);
		\draw [style=usimx] (4.center) to (3.center);
		\draw [style=usimx] (9.center) to (11.center);
		\draw [style=usimx] (11.center) to (12.center);
		\draw [style=usimx] (9.center) to (12.center);
		\draw [style=usimy] (4.center) to (8.center);
		\draw [style=usimy] (12.center) to (8.center);
		\draw [style=usimy] (4.center) to (12.center);
		\draw [style=usimy] (3.center) to (11.center);
		\draw [style=usimy] (11.center) to (5.center);
		\draw [style=usimy] (5.center) to (3.center);
	\end{pgfonlayer}
\end{tikzpicture}}
    \caption*{$\widetilde{\cO_{12}}$}
    \scalebox{0.9}{\begin{tikzpicture}
	\begin{pgfonlayer}{nodelayer}
		\node [style=none] (0) at (0, 3) {};
		\node [style=none] (1) at (-2, 1) {};
		\node [style=none] (2) at (-2, -1) {};
		\node [style=none] (3) at (0, -3) {};
		\node [style=none] (4) at (2, -1) {};
		\node [style=none] (5) at (2, 1) {};
	\end{pgfonlayer}
	\begin{pgfonlayer}{edgelayer}
		\draw [style=usimx] (1.center) to (0.center);
		\draw [style=usimx] (5.center) to (4.center);
		\draw [style=usimx] (2.center) to (3.center);
		\draw [style=usimy] (0.center) to (5.center);
		\draw [style=usimy] (4.center) to (3.center);
		\draw [style=usimy] (2.center) to (1.center);
	\end{pgfonlayer}
\end{tikzpicture}}
    \caption*{$\calC_6$}
    \label{fig:figrandom}
  \end{subfigure}
  \caption{A sample of finite orbits \label{fig:sample_orbits}}
\end{figure}

The orbit $\cO$ has a graph structure: the vertices are
the elements of the orbit and the edges are adjacencies, colored here by
their adjacency type. The $x$-adjacencies are represented
in red and the $y$-adjacencies in blue. As the $x$ and $y$ adjacencies come from
equivalence relations, the monochromatic connected components of $\cO$
are \emph{cliques} (any two vertices of such a component are connected
by an edge). Moreover, by definition of the transitive closure,
the graph $\cO$ is \emph{connected}, that is, every two vertices of the graph
are connected by a path.
In the sequel, we denote by $\cO$ either the set of pairs  in the orbit
or the induced graph. The structure considered should be clear from the  context.
For a model $\calW$, its \emph{orbit type} corresponds to the class of
its orbit modulo graph isomorphisms.

\begin{exa}
  For small steps models, the orbit when finite is always isomorphic
  to a cycle whose vertices all belong to $\C(x,y)^2$. Example
  $\calC_6$ in Figure~\ref{fig:sample_orbits} is for instance
  the unlabelled orbit of the unweighted small steps model
  $\calS=\{(-1,0),(0,1),(1,-1)\}$ \cite[Example 2]{BMM}.
\end{exa}

The orbit type being preserved when one reverses the model, Section
$10$ in \cite{bostan2018counting} lists the distinct orbit types for
models with steps in $\{-1,0,1,2\}^2$ with at least one large
step. For these models, the finite orbit types are exactly $\cO_{12},
\widetilde{\cO_{12}}$ and $\cO_{18}$ in Figure \ref{fig:sample_orbits}
and the cartesian product orbit-types of the \emph{Hadamard
models} that correspond to a step polynomial of the form $R(X)
+P(X)Q(Y)$ (see \cite[Section 6]{bostan2018counting} or Section
\ref{subsubsect:Hadamard}).

\begin{exa}[The model $\Gmod$] \label{ex:gessel2_alg_5}
For  $\Gmod$,  the polynomial
$\Ktld(Z,y,\invS)$ is reducible over $k(x,y)[Z]$ and factors
as $\frac{(Z-x)yP(Z)}{x^3y^2 + (\lambda y + 1)x^2 + y^2 x + 1}$ where  
 \[P(Z)=x \,y^{2} Z^{2}+ (x^{2} y^{2} +\lambda x y + x)Z - 1. \]  Thus,   an element $(z,y) \in \K^2$ distinct from
$(x,y)$ is $y$-adjacent to $(x,y)$  if and only if $z$ is a root of $P(Z)$.  Its roots  are of the form
$z,\frac{-1}{xy^2z}$ by the relation between the roots and the
coefficients of a degree two polynomial. One can then show that the
orbit $\cO_{12}$ in Figure~\ref{fig:sample_orbits} is the orbit of the
model $\Gmod$.  Since none of the vertices depend on $\lambda$, the graph $\cO_{12}$ is
also the orbit of the  model  $\mathcal{G}_0$.
\end{exa}

Finally, we would like to discuss the finiteness of the orbit. For small steps
walks, the finiteness of the orbit depends only on the order of $\Phi \circ
\Psi$.  Some number theoretic considerations on the    torsion subgroup
of the Mordell-Weil group of a rational elliptic
surface prove that this order, when finite, is bounded by $6$, which provides a
very easy algorithm to test the finiteness of the group of the walk. This bound is valid for any choice of weights contained in an algebraically closed field of characteristic zero (see
\cite[Remark 5.1]{HS} and  \cite[Corollary~8.21]{SchuttShioda}).
For models with arbitrarily large steps,  there is  currently no general criterion to
determine whether the orbit is finite or not, but only a semi-algorithm \cite[Section 3.2]{bostan2018counting}.  We hope that analogously to the
small steps case a geometric interpretation of the notion of orbit will  provide
some bounds on the potential diameter of the orbit.

\subsection{The Galois extension of the orbit}\label{subsect:galoisetxorbit}

In the remaining of the article, we  denote by $k(\cO)$ the subfield of
$\K$ generated over $k=\C(S(x,y))$ by all coordinates of the orbit
$\cO$. Note that $k(\cO)$ coincides with $\C(\cO)$ since $x,y$
belong to the orbit.

We start this subsection with some terminology on field
extensions. Our main reference is \cite{Szamuely} which is a concise
exposition of the Galois theory of field extensions of finite
and infinite degree.
A field extension $M \subset L$ is denoted by
$L|M$.  The \emph{ degree  of the field extension  $L|M$ } is the dimension of
$L$ as $M$-vector space. When this degree is finite, we denote it
$[L:M]$. For $L|M$ and $L'|M$ two field extensions, an \emph{$M$-algebra
homomorphism } of $L$ into $L'$ is a ring homomorphism from $L$ to $L'$
that is the identity on $M$.  An \emph{ algebraic closure  of a field $M$ } is
an algebraic extension of $M$ that is algebraically closed. Let us
recall some of its properties.
\begin{prop}[Proposition~1.1.3 in  \cite{Szamuely}]\label{prop:algebraicclosure}
 Let $M$ be a field.
\begin{enumerate}
\item  There exists an algebraic closure $\overline{M}$ of $M$. It is unique up to isomorphism.
\item For an algebraic extension $L$ of $M$, there exists an embedding
from $L$ to $\overline{M}$ leaving $M$ elementwise fixed. Moreover,
any $M$-algebra homomorphism from $L$ into $\overline{M}$ can be
extended to an $M$-algebra isomorphism of $\overline{L}$ to
$\overline{M}$.
\end{enumerate}
\end{prop}

The field $\K$ introduced in Section \ref{subsect:theorbit} is an
algebraic closure of $\C(x,y)$.  By definition of the orbit,
$k(\cO)=\C(\cO)$ is an algebraic field extension of
$\C(x,y)$. Moreover, since $y$ is algebraic over $k(x)$ and $x$ is
algebraic over $k(y)$ by Lemma~\ref{lem:alg_rel_xyS}, then $\C(x,y)$
is an algebraic field extension of $k(x)$ and $k(y)$. Therefore,
$k(\cO)$ is algebraic over $k(x)$ and $k(y)$.  Proposition
\ref{prop:algebraicclosure} implies that $\K$ is an algebraic closure
of $k(x), k(y)$ and $k(\cO)$.

Let $L|M$ be a field extension. Any $M$-algebra endomorphism of
$L$ is an automorphism and we denote by $\Aut(L|M)$ the set of
$M$-algebra endomorphisms of $L$. An algebraic field extension $L|M$
is said to be \emph{Galois} if the set
$L^{\Aut(L|M)}$ of elements of $L$ that remain fixed under the action
of $\Aut(L|M)$ coincides with $M$ (see \cite[Definition
1.2.1]{Szamuely}). In this case, $\Aut(L|M)$ is denoted by
$\Gal(L|M)$ and called the \emph{Galois group} of $L|M$.
By \cite[Proposition~1.2.4]{Szamuely}, an algebraic
field extension $L|M$ is Galois if and only if, fixing an algebraic
closure $\overline{M}$ of $M$, we have $\sigma(L) \subset L$ for any
automorphism $\sigma$ in $\Aut(\overline{M}|M)$ \footnote{Since we are
in characteristic zero, the separable closure of $M$ coincides with
the algebraic closure of $M$ (see \cite[page 12]{Szamuely}).}. The
Galois group $\Gal(L|M)$ of a finite Galois extension $L|M$ has order
$[L:M]$ \cite[Corollary~1.2.7]{Szamuely}. It is clear that any
sub-extension $L|M'$ \footnote{By
subextension, we mean that $M \subset M' \subset L$. }  of a Galois extension $L|M$  is Galois.
Finally, we recall the following result.

\begin{lem}[Lemma~1.22 in \cite{Szamuely}] \label{lem:gal_trans_roots}
  Let $L|M$ be a Galois extension and $\mu \in M[X]$ an irreducible
  polynomial with some root $\alpha$ in $L$. Then $\mu$ splits in $L$, and the
  group $\Gal(L|M)$ acts transitively on its roots.
\end{lem}

We let any $\C$-algebra endomorphism $\sigma$ of $\K$ act on $\K\times\K$
coordinate-wise by \[\sigma \cdot (u,v) \eqdef (\sigma(u), \sigma(v)).\] The following lemma establishes  the compatibility of the equivalence relation
$\sim^*$ with the action of $\C$-algebra endomorphisms of $\K$.
\begin{lem} \label{lem:hom_adj}
  Let $(u,v)$ and $(u',v')$ be two pairs in $\K \times \K$
    and $\sigma \colon \K \rightarrow \K$ be
    a $\C$-algebra endomorphism.
    Then $(u,v) \sim^x (u',v')$ (resp. $(u,v) \sim^y (u',v')$) implies that
    $\sigma \cdot (u,v) \sim^x \sigma \cdot (u',v')$ (resp.  $\sigma\cdot (u,v) \sim^y \sigma \cdot(u',v')$). The same holds therefore for $\sim^*$.
\end{lem}
\begin{proof}
  Since $\sigma$ is a $\C$-algebra endomorphism, we have
  $\sigma S(u,v)= S( \sigma u, \sigma v)$ for any $u, v$ in $\K$.
  Therefore, if $(u,v) \sim^x (u,v')$ then
  $S(\sigma(u),\sigma(v)) = \sigma (S(u,v)) = \sigma ( S(u,v')) = S(\sigma(u), \sigma(v'))$,
  so $\sigma \cdot (u,v) \sim^x \sigma \cdot (u,v')$.
    The same argument applies if $(u,v) \sim^y (u',v)$. The general case of $(u,v) \sim^* (u',v')$ follows by induction.
\end{proof}

As a direct corollary, we find the following lemma which ensures the setwise stability of the orbit under certain endomorphisms of $\K$.
\begin{lem} \label{lem:orb_normal}
    Let $\sigma_x \colon \K  \rightarrow \K$
    be a $k(x)$-algebra endomorphism. Then, for all $(u,v)$ in the orbit,
    $\sigma_x \cdot (u,v)$ is  in the orbit. Similarly, the orbit is also
    stable under $k(y)$-algebra endomorphisms of $\K$.
\end{lem}
\begin{proof}
    Let $(u,v)$ be in the orbit, i.e. $(u,v) \sim^* (x,y)$.
    By Lemma~\ref{lem:hom_adj}, we find that  \[\sigma_x \cdot (u,v) \sim^* \sigma_x \cdot (x,y) = (x, \sigma_x(y)).\]
    By transitivity, we only need to prove that $(x,\sigma_x(y))$ is in the
    orbit.
    This is true because $S(x, \sigma_x(y)) = \sigma_x S(x, y) = S(x,y)$
    since $\sigma_x$ fixes $\C(x,S(x,y))$
    so $(x,\sigma_x(y)) \sim^x (x,y)$.
\end{proof}

The above two lemmas imply that any $k(x)$ or $k(y)$-algebra
automorphism of $\K$ induces a permutation of the vertices of $\cO$
which preserves the colored adjacencies, and is therefore a
\emph{graph automorphism} of $\cO$. The stability result of
Lemma~\ref{lem:orb_normal} translates as a field theoretic statement.
\begin{thm} \label{thm:ko_gal}
    The extensions $k(\cO)|k(x), k(\cO)|k(y)$ and $k(\cO)|k(x,y)$ are Galois.
\end{thm}
\begin{proof}
  We first prove that $k(\cO) | k(x)$ is a Galois extension.  Recall
  that the field extension $k(\cO)|k(x)$ is algebraic and $\K$ is an
  algebraic closure of $k(\cO)$ and $k(x)$. Thus, we only need to prove
  that $\sigma(k(\cO)) \subset k(\cO)$ for every automorphism $\sigma$
  in $\Aut(\K|k(x))$. This follows directly from
  Lemma~\ref{lem:orb_normal}. The proof for $k(\cO)|k(y)$ is entirely
  symmetric and the field extension $k(\cO) |k(x,y)$ is Galois as
  subextension of $k(\cO)|k(x)$.
\end{proof}

Theorem~\ref{thm:ko_gal} gives a Galoisian framework to the orbit,
which will be central in our study of Galois invariants and
decoupling. Remark that the algebraic extension $k(\cO)|k(x,y)$ may be
of infinite degree.  In Figure~\ref{fig:extension}, we represent
the different Galois extensions involved in Theorem~\ref{thm:ko_gal}
and we denote their Galois groups $G_x = \Gal(k(\cO) | k(x))$, $G_y =
\Gal(k(\cO) | k(y))$ and $G_{xy} = \Gal(k(\cO) | k(x, y))$. Note that
$G_{xy}=G_y \cap G_x$.

\begin{figure}[ht]
    \centering
    \includegraphics[scale=0.8]{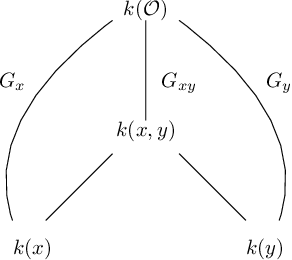}
    \caption{The field extensions attached to the orbit}
 \label{fig:extension}
\end{figure}

\begin{exa}\label{exa:Galoisgroupsmallsteps}
 For small steps models, we have
$k(\cO)=k(x,y)=\C(x,y)$. Moreover, the field extensions $k(\cO)|k(x)$
and $k(\cO)|k(y)$ are both of degree $2$ so that $G_x$ and $G_y$ are
groups of order $2$ and thereby isomorphic to $\Z/2\Z$. In the
notation of the beginning of Section \ref{sect:orbit}, consider the
endomorphisms $\phi,\psi$ of $\C(x,y)$ defined as follows: for $f(x,y)
\in \C(x,y)$, we set $\phi(f)= f(\Phi(x,y))$ and
$\psi(f)=f(\Psi(x,y))$. It is easily seen that $\psi \in G_x$ and that
$\phi \in G_y$ and that they both are non-trivial involutions. Thus,
we have $G_{xy} = 1$, $G_x = \langle \psi \rangle \simeq \Z/2\Z$
and $G_y = \langle \phi \rangle \simeq Z/2\Z$.
\end{exa}

Example \ref{ex:gessel2_alg_6} below shows that, unlike the small steps
case, one has to go to a finite non-trivial field extension of
$\C(x,y)$ in order to build the orbit of a large steps model.

\begin{exa}[The model $\Gmod$] \label{ex:gessel2_alg_6}

  In the case of $\Gmod$, we have $k(\cO) = \C(x,y,z)$ where $z$ is a
root of the polynomial $P(Z)=x y^{2} Z^{2}+ (x^{2} y^{2} +\lambda x y
+ x)Z -1$ (see Example \ref{ex:gessel2_alg_5}).  Its discriminant is \[x \left(x^{3} y^{4}+2 \lambda
x^{2} y^{3}+\lambda^{2} x \,y^{2}+2 x^{2} y^{2}+2 \lambda x y +4 
y^{2}+x \right)\] which cannot be a square in $k(x,y)=\C(x,y)$ because of the irreducible factor $x$.  Thus, the polynomial $P(Z)$ is irreducible in
$\C(x,y)[Z]$.  Therefore, $z
\notin k(x,y)$ and the extension $k(\cO)|k(x,y)$ is of degree $2$.  As
above, we find that $G_{xy} \simeq \Z/2\Z$.

   The field extension $k(\cO) | k(y)$ is of degree $6$ so that
its Galois group is either $S_3$ or $\Z/6\Z$. In this last case, the
group $G_{xy}$ would be a normal subgroup of $G_y$. As
$k(x,y)=k(\cO)^{G_{xy}}$, the extension $k(x,y)|k(y)$ would be Galois
by \cite[Theorem~1.2.5]{Szamuely}.  This is impossible since the root
$z$ of $\Ktld(Z,y,\invS)$ is not in $k(x,y)$. Hence, we find that
$G_{y} \simeq S_3$.

 The extension $k(\cO)|k(x)$ is of degree $4$. Its Galois group
is of order four and therefore either isomorphic to $\Z/2\Z \times
\Z/2\Z$ or to $\Z/4\Z$. If $G_x$ were $\Z/4\Z$ then there would exist
a $k(x)$-algebra endomorphism $\sigma$ of $k(\cO)$ of order $4$. Since
$G_{xy}$ is of order $2$, the automorphism $\sigma$ can not fix $y$
and we must have $\sigma(y)=\frac{1}{xy}$, which is the other root of
$\Ktld(x,Y,\invS) \in k(x)[Y]$. Since the orbit is setwise invariant
by $\sigma$ and $(z,y)$ is in the orbit, the same holds for $
\sigma(z,y)=(\sigma(z),\frac{1}{xy})$.  From the description of the
orbit of $\Gmod$ in Figure~\ref{fig:sample_orbits}, we find that
$\sigma(z) \in \{ \frac{1}{z}, xy^2z \}$. In both cases, we find
that $\sigma^2(z)=z$ which implies that $\sigma^2$ is the identity on
$k(\cO)$. A contradiction.  Hence, we conclude that the group $G_x$ is
isomorphic to $\Z/2\Z \times \Z/2\Z$.
\end{exa}

\subsection{The group of the walk}\label{sect:grouporbit}

In this section, we  prove that the orbit $\cO$ is the orbit of the pair $(x,y)$
under the action of a certain group which generalizes the one
introduced in the small steps case by Bousquet-Mélou and Mishna
\cite[Section 3]{BMM}.

\begin{defi} For a model $\calW$ with non-univariate step polynomial,  we denote by $G$
the subgroup of $\Aut(k(\cO)|k)$
generated by $G_x$ and $G_y$, and we call it the
\emph{group of the walk}. \end{defi}

Recall the discussion on small steps models.
 Example \ref{exa:Galoisgroupsmallsteps} shows that 
the group $G$ is generated by the automorphisms $\psi,\phi$,
as they generate $G_x$ and $G_y$.
Thereby, $G$ is isomorphic to the classic group of the walk $\left<\Phi, \Psi\right>$.

As explained in Section \ref{subsect:galoisetxorbit}, every element
of $G$ induces a \emph{graph automorphism} of $\cO$, that is, a
permutation of the vertices of $\cO$ which preserves the colored
adjacencies on the orbit $\cO$.  In Theorem~\ref{thm:group_trans}
below, we prove that there exists a finitely generated subgroup of $G$
whose action on $\cO$ is faithful and transitive, which is a
notable property of the classic group of the walk. It is clear that the group $G$ acts faithfully on the orbit
$\cO$. Indeed, if an element $\sigma$ of $G$ is the identity on any
element of the orbit then $\sigma$ is the identity on
$k(\cO)$. Therefore, $\sigma$ is the identity.  The construction of a
finitely generated subgroup of $G$ with a transitive action on the
orbit requires a bit more work. We first prove two lemmas on the polynomial $\Ktld(X,Y,t)$.

\begin{lem} \label{lem:ker_irred}
    The kernel polynomial $\Ktld(X,Y,t)$ is irreducible in
$\C[X,Y,t]$. Therefore, it is irreducible as a polynomial in
  $\C(X,t)[Y],\C(Y,t)[X]$ and $\C(t)[X,Y]$.
\end{lem}
\begin{proof}
    The kernel polynomial is a degree 1 polynomial in $t$, therefore
it is irreducible in $\C(X,Y)[t]$. Moreover, its content is one
by construction.  Therefore, by Gauss Lemma \cite[chap. V par. 6
Theorem~10]{Lang}, the kernel polynomial is irreducible in
$\C[X,Y][t]=\C[X,Y,t]$. Since $S(X,Y)$ is not univariate, the
polynomial $\Ktld$ does not belong to $\C[X,t]$, Gauss Lemma~asserts
that $\Ktld$ being irreducible in $\C[X,t][Y]$ is also irreducible in
$\C(X,t)[Y]$. The same reasoning holds for the irreducibility of
$\Ktld$ in $\C(Y,t)[X]$. It is clear that since $\Ktld$ is irreducible
in $\C[X,Y,t]$ and not in $\C(t)$, it is irreducible in $\C(t)[X,Y]$.
\end{proof}

\begin{lem} \label{lem:ker_spec_irred}
    The specializations of the kernel polynomial $\Ktld\left(x,Y,\invS\right)$
    and $\Ktld\left(X,y,\invS\right)$ are respectively irreducible as
    polynomials in $k(x)[Y]$ and in $k(y)[X]$.
\end{lem}
\begin{proof}
    We only prove the first assertion by symmetry of the roles of $x$ and $y$.
    Consider the $\C$-algebra homomorphism ${\phi : \C[X,t] \rightarrow
    k(x)}$ defined by $\phi(X) = x$ and $\phi(t) = \invS$. Since $S(X,Y)$ is not univariate, the fractions $x$ and $\invS$ are algebraically
independent over $\C$. Therefore the morphism $\phi$ is one-to-one,
so it extends to a field isomorphism $\phi : \C(X,t) \rightarrow
k(x)$ (onto by definition of $k(x)$), which extends to a $\C$-algebra
isomorphism $\phi$ from $\C(X,t)[Y]$ to $k(x)[Y]$. Moreover, by
Lemma~\ref{lem:ker_irred}, $\Ktld(X,Y,t)$ is irreducible as a
polynomial in $\C(X,t)[Y]$. Therefore, since
$\Ktld\left(x,Y,\invS\right) = \phi (\Ktld(X,Y,t))$ and
$\phi(\C(X,t)) = k(x)$, we conclude that the polynomial
    $\Ktld\left(x,Y,\invS\right)$ is irreducible over $k(x)$.
\end{proof}

For large steps models, the extensions $k(\cO)|k(x)$ and $k(\cO)|k(y)$
might be of infinite degree, hence the groups $G_x$ and $G_y$ might
not be finite, not even finitely generated (unlike the small steps
case where they are always cyclic of order $2$).  However, note that
$G_{xy}$ is the stabilizer of the pair $(x,y)$ in the orbit.
Therefore, the action of $G$ on $(x,y)$ factors through the
left quotients $G_x / G_{xy}$ and $G_y / G_{xy}$ which are proved to
be finite in the following lemma.

\begin{lem}
  The group $G_{xy}$ is of finite index in $G_x$ and in $G_y$
  with $[G_x : G_{xy}] = m_y+M_y$ and $[G_y : G_{xy}] = m_x+M_x$.
\end{lem}
\begin{proof}
  The orbit $\Omega$ of $y$ under the action of $G_x$
  is a subset of the roots of the polynomial $\Ktld(x,Z,\invS) \in k(x)[Z]$.
  This polynomial is irreducible by Lemma~\ref{lem:ker_irred}, so
  $G_x$ acts transitively on its roots by Lemma~\ref{lem:gal_trans_roots},
hence $\Omega$ coincides with the set of roots
  of $\Ktld(x,Z,\invS)$ which is a finite set of cardinal $\deg_Y \Ktld = M_y + m_y$.
  Moreover, the stabilizer of $y$ for this action is precisely the group $G_{xy}$. Therefore,
  the quotient $G_x / G_{xy}$ can be identified with $\Omega$, wich
  proves that $G_{xy}$ is of finite index in $G_x$ with $[G_x : G_{xy}] = M_y + m_y$.
  The proof for the subgroup $G_y$ is analogous.
\end{proof}

Therefore, we fix, once and for all, a set
$I_x = \{\id, \iota^x_{1}, \dots, \iota^x_{m_y+M_y}\}$
of representatives of the left cosets of $G_x / G_{xy}$,
and a set $I_y = \{\id, \iota^y_{1}, \dots, \iota^y_{m_x+M_x}\}$ of
representatives of the left cosets of $G_y / G_{xy}$.
By construction, \[G_x = \left<I_x, G_{xy}\right>, G_y =
  \left<I_y, G_{xy}\right>, \mbox{ and } G = \left<I_x, I_y, G_{xy}\right>.\]

We now have all the ingredients to prove the transitivity of the
action of a finitely generated subgroup of $G$ on $\cO$.  We only
recall that the distance between two vertices of a graph is the number
of edges in a shortest path connecting them.

\begin{thm}[Transitivity of the action] \label{thm:group_trans}
    The subgroup of $G$ generated by $I_x$ and $I_y$ acts transitively
    on the orbit $\cO$.
\end{thm}
\begin{proof}
    We show that for all pairs $(u,v)$ of $\cO$ there exists an
    element $\sigma$ in $\langle I_x,I_y\rangle$ such that $\sigma \cdot
    (x,y) = (u,v)$.
    As the graph of the orbit is connected, the proof is done by
induction on the distance between $(x,y)$ and $(u,v)$. If $(u,v)$ is
at distance zero to $(x,y)$ then $(u,v)=(x,y)$ and we set
$\sigma=\id$.

    Let $(u,v)$ be in $\cO$ of positive distance $d$ to $(x,y)$. Then
    there exists a pair $(u',v')$ at distance $d-1$ to $(x,y)$ that is
    adjacent to $(u,v)$. Without loss of generality, one can assume that
    $(u',v')$ is $x$-adjacent to $(u,v)$, that is, $u=u'$.
    By induction hypothesis, there exists $\sigma$ in $\left< I_x,I_y\right>$ such that
    $\sigma \cdot (x,y) = (u,v')$. Therefore,
    since $(u,v') \sim^x (u,v)$, the application of $\sigma^{-1}$ implies
    by Lemma~\ref{lem:hom_adj} that $(x,y) \sim^x (x, \sigma^{-1}(v))$.
    Thus, $y$ and $\sigma^{-1} (v)$ satisfy the equation $S(x,y) -
S(x,Y)$, so they are roots of the polynomial
$\Ktld\left(x,Y,\invS\right)$ which is an irreducible polynomial over
$k(x)$ by Lemma~\ref{lem:ker_spec_irred}.  Therefore, by
    Lemma~\ref{lem:gal_trans_roots},
    there is an element $\sigma_x$ in $G_x$ such
    that $\sigma_x(y)= \sigma^{-1}(v)$.
    Let $\iota^x_{i}$ in $I_x$ be the representative of the left coset
    $\sigma_x G_{xy}$. Then,  $(\sigma\, \iota^x_{i}) \cdot (x,y) = \sigma
    \cdot (\sigma_x \cdot (x,y)) = \sigma \cdot (x, \sigma^{-1} (v)) = (u,v)$.
    This concludes the proof.
\end{proof}

This result shows that the orbit  $\cO$ is actually the
orbit of the pair $(x,y)$ under the action of a finitely generated
subgroup of $G$. As a direct corollary, one finds that the extensions
$k(x,y)|k$ and $k(u,v)|k$ are isomorphic for any pair $(u,v)$ in the
orbit. Indeed, let $\sigma$  in $G$ such that $\sigma\cdot (x,y)=(u,v)$
then $\sigma$ induces a $k$-algebra isomorphism between $k(x,y)$ and
$k(u,v)$.

For large steps models with an infinite orbit, it might be quite
difficult to give a precise description of the automorphisms in $I_x$
and $I_y$.  Indeed, they act as a permutation on the infinite orbit $\cO$
and their action on $x$ or $y$ is not in general given by a rational fraction
in $x$ and $y$ as in the small steps case. When the steps are small or
when the orbit is finite, one is able to give a more precise
description of these generators.

\begin{exa} In the small steps case and in the notation of Example \ref{exa:Galoisgroupsmallsteps}, one can choose $I_x =\{ \id ,\psi \}$ and $I_y=\{ \id, \phi \}$.
\end{exa}

\begin{exa}[The model $\Gmod$] \label{ex:gessel2_alg_7}
For $\Gmod$, the group $G_y$ is isomorphic to $S_3$, $G_{xy}$ to
$\Z/2\Z$ and $G_x$ to $\Z/2\Z$. We give below the expression of
automorphisms  $\iota^x,\iota^y$ and $\tau$ such that
\[I_x =\{ \id, \iota^x\}, \, I_y=\{ \id ,\iota^y\} , G_{xy}=\langle
\tau \rangle. \] They satisfy the relations ${(\iota^x)}^2= {(\iota^y)}^3=
\tau^2 =\id $. We represent below their action on the orbit.

\begin{figure}[h!]
\begin{minipage}{0.5\textwidth}
  \scalebox{0.8}{
  \begin{tikzpicture}[scale=1.5]
	\begin{pgfonlayer}{nodelayer}
		\node [style=paire] (0) at (-8, 4) {$(x,\frac{1}{xy})$};
		\node [style=paire] (1) at (-9.25, 5.75) {$(xy^2z,\frac{1}{xy})$};
		\node [style=paire] (2) at (-6.75, 5.75) {$(-\frac{1}{z},\frac{1}{xy})$};
		\node [style=paire] (3) at (-8, 1.25) {$(x,y)$};
		\node [style=paire] (4) at (-9, -0.75) {$(z,y)$};
		\node [style=paire] (5) at (-7, -0.75) {$(-\frac{1}{xy^2z},y)$};
		\node [style=paire] (6) at (-11, -3) {$(z,\frac{1}{yz})$};
		\node [style=paire] (7) at (-13.5, -3.5) {$(xy^2z,\frac{1}{yz})$};
		\node [style=paire] (8) at (-11.5, -5.5) {$(-\frac{1}{x},\frac{1}{yz})$};
		\node [style=paire] (9) at (-5, -3) {$(-\frac{1}{xy^2z},-xyz)$};
		\node [style=paire] (10) at (-4.5, -5.5) {$(-\frac{1}{x},-xyz)$};
		\node [style=paire] (11) at (-2.5, -3.5) {$(-\frac{1}{z},-xyz)$};
		\node [style=none] (12) at (-7.5, 2.5) {$\iota^x$};
		\node [style=none] (13) at (-9, 0.5) {$\iota^y$};
		\node [style=none] (14) at (-7, 0.5) {$(\iota^y)^2$};
		\node [style=none] (15) at (-11, -1.5) {$\iota^y \iota^x (\iota^y)^{-1}$};
		\node [style=none] (16) at (-4.3, -1.5) {$(\iota^y)^2 \iota^x (\iota^y)^{-2}$};
	\end{pgfonlayer}
	\begin{pgfonlayer}{edgelayer}
		\draw [style=usimx] (4) to (3);
		\draw [style=usimx] (3) to (5);
		\draw [style=usimx] (5) to (4);
		\draw [style=usimx] (1) to (0);
		\draw [style=usimx] (0) to (2);
		\draw [style=usimx] (2) to (1);
		\draw [style=usimx] (6) to (7);
		\draw [style=usimx] (7) to (8);
		\draw [style=usimx] (8) to (6);
		\draw [style=usimx] (9) to (10);
		\draw [style=usimx] (10) to (11);
		\draw [style=usimx] (11) to (9);
		\draw [style=usimy] (4) to (6);
		\draw [style=usimy] (5) to (9);
		\draw [style=usimy, bend right, looseness=0.50] (8) to (10);
		\draw [style=usimy] (3) to (0);
		\draw [style=usimy, bend right=45, looseness=0.75] (1) to (7);
		\draw [style=usimy, bend left=45, looseness=0.75] (2) to (11);
		\draw [style=arrow] (3) to (0);
		\draw [style=arrow] (3) to (4);
		\draw [style=arrow] (3) to (5);
		\draw [style=arrow] (4) to (6);
		\draw [style=arrow] (5) to (9);
	\end{pgfonlayer}
\end{tikzpicture}
  }
\end{minipage}
\hspace{0.05\textwidth}
\begin{minipage}{0.2\textwidth}
$\iota^x(x) = x$\\
$\iota^x(y) = \frac{1}{xy}$\\
$\iota^x(z) = xy^2z$ \\

$\tau(x)=x$ \\
$\tau(y)=y$ \\
$\tau(z)= -\frac{1}{xy^2z}$ \

\end{minipage}
\hspace{5pt}
\begin{minipage}{0.2\textwidth}
$\iota^y(x) = z$\\
$\iota^y(y) = y$\\
$\iota^y(z) = -\frac{1}{xy^2z}$

\end{minipage}
\caption{The elements of $I_x$ and $I_y$ for the model $\Gmod$}

\end{figure}

\end{exa}

\begin{exa} 
In \cite[Notation 3.6]{BuchacherKauersPogudin}, the authors attach to
a ``separated polynomial '' $P(X) +Q(Y)$ where $P(X) \in \C[X]$ of
degree $m$ and $Q(Y) \in \C[Y]$ of degree $n$ a Galois group which can
be seen as a subgroup of $S_m \times S_n$. This Galois group coincides
with the group of the walk of the Hadamard model with step polynomial
$P(X)+Q(Y)$ (see Appendix \ref{sec:groupe-hadam-empl}).
\end{exa}

 Note that an $x$-adjacency in the orbit corresponds to the action of an
element of $G$ that is conjugate to an element of $G_x$. Indeed, for
$(u,v)$ in the orbit, Theorem~\ref{thm:group_trans} yields the
existence of $\sigma \in G$ such that $\sigma(u,v)=(x,y)$. Then, if
$(u,v) \sim^x (u,v')$, Lemma~\ref{lem:hom_adj} proves that
$\sigma(u,v') \sim^x (x,y)$. As explained above, any $x$-adjacency to
$(x,y)$ corresponds to the action of an automorphism in $G_x$ so that
there exists $\sigma_x$ in $G_x$ such that $\sigma(u,v')
=\sigma_x(x,y)$. We conclude that $(u,v')=\sigma^{-1} \sigma_x \sigma
(u,v)$.  In the above example for the model $\Gmod$, one sees that
$(z,y) \sim^x(z,\frac{1}{xz})$ and that  $\iota^y \iota^x (\iota^y)^{-1} \cdot (z,y)=(z,\frac{1}{xz})$. The automorphism $\iota^x$ belong
to $G_x$ but it is not the case of $\iota^y \iota^x (\iota^y)^{-1}$ since
$\iota^y \iota^x (\iota^y)^{-1}(x)=xy^2z$.

 Moreover, the transitivity of the action of $G$ on
the orbit also implies the following minimality result for the
extension $k(\cO)$.

\begin{prop} \label{prop:smallest_normal_ext}
 The field $k(\cO)$ is the smallest field in $\K$ that is a Galois extension of $k(x)$ and a Galois extension of $k(y)$.
\end{prop}
\begin{proof}
    Let $M \subset \K$ be a Galois extension of $k(x)$ and a Galois
extension of $k(y)$. Proposition~\ref{prop:algebraicclosure}
shows that $\K$ is an algebraic closure for $M$. Let $(u,v)$ be an
element of $\cO$. To prove that $k(\cO) \subset M$, we only need to
show that $u$ and $v$ belong to $M$.
By Theorem~\ref{thm:group_trans}, there exists
$\sigma$ in $G$ such that $\sigma \cdot (x,y) = (u,v)$.  Let
us first assume that $\sigma$ belongs to $G_x$. Since $\K$ is an
algebraic closure for $k(x)$, Proposition~\ref{prop:algebraicclosure}
shows that $\sigma$ extends as a $k(x)$-algebra endomorphism of $\K$
still denoted $\sigma$. The field extension $M|k(x)$ is Galois and
$\K$ is an algebraic closure of $M$ so that $\sigma(M) \subset
M$. Since $x$ and $y$ belong to $M$, the same holds for $(u,v)$. The
proof is analogous if $\sigma$ belong to $G_y$. Since $G$ is generated
by $G_x$ and $G_y$, an easy induction concludes that $u=\sigma(x)$ and $v=\sigma(y)$
both belong to $M$ for any $\sigma$ in $G$.
\end{proof}

\subsection{Orbit sums} \label{subsect:orbit sum}

One of the purposes of the orbit is to provide a nice family of
changes of variables, in the sense that the kernel
polynomial $K(X,Y,t)$ is constant on the orbit: for all pairs $(u,v)$
of the orbit, $K(u,v,t) = K(x,y,t)$ (because $S(x,y)=S(u,v)$)
This polynomial being a  factor of the left-hand side
 of the functional equation satisfied by the generating function, one
can evaluate the variables $(X,Y)$ at any pair $(u,v)$ of the orbit
and obtain what is called an \emph{orbit equation}.  Indeed, the
   generating function $Q(X,Y)$
and its sections $Q(X,0)$ and $Q(0,Y)$ belong to the ring of formal
power series in $t$ with coefficients in $\C[X,Y]$ so that their
evaluation at $(u,v)$ belong to the ring $\C[\cO][[t]]$.  Note that
such an evaluation leaves the variable $t$ fixed. The strategy
developed in \cite[Section 4]{bostan2018counting} for models with
small forward steps consists in forming linear combinations of these
orbit equations so that the resulting equation is free from
sections. From the section-free equation, Bostan, Bousquet-Mélou
and Melczer sometimes succeed in isolating the generating function
$Q(X,Y)$ and expressing it as a diagonal of algebraic fractions which
leads to its $D$-finitness by \cite{Lip}. For models with small
backward steps, it is quite easy to produce a section-free equation
from \eqref{eq:small_steps_backwards} when the orbit contains a
cycle. However, it is very unlikely that, for models with small
backward steps and at least one large step, such a section free
equation suffices to characterize the generating function.

In this paper, we want to evaluate the variables $X,Y,t$ at
$(u,v,\invS)$ for $(u,v)$ an element of the orbit. Since
$\Ktld(u,v,\invS)=0$ for any element $(u,v)$ of the orbit $\cO$, such
an evaluation is similar to the kernel method used in \cite{KurkRasch}
for models with small steps. More precisely, let us define a
\emph{$0$-chain} as a formal $\C$-linear combination of elements of
the orbit $\cO$ with finite support. This terminology is borrowed from graph homology (see
Section \ref{sect:decoupling} for some basic introduction). Let
$\gamma =\sum_{(u,v) \in \cO} c_{(u,v)} (u,v)$ be a zero chain. Since
the coefficients $c_{(u,v)}$ are complex and almost all zero, the
evaluation $P_{\gamma}$ of a polynomial $P(X,Y,t) \in \C[X,Y,t]$ at
$\gamma$ is defined as
\[ P_\gamma=\sum_{(u,v) \in \cO} c_{(u,v)} P(u,v),\] and belongs to
$\C[\cO]$. The evaluation of $\Ktld(X,Y,t)$ at any $0$-chain vanishes
so that one can not evaluate a rational fraction in $\C(X,Y,t)$ whose
denominator is divisible by $\Ktld$. This motivate the following
definition.

\begin{defi}\label{defi:regularfraction}
    Let $H(X,Y,t) = \frac{A(X,Y,t)}{B(X,Y,t)}$ be a rational fraction in
    $\C(X,Y,t)$ where $A(X,Y,t)$ and $B(X,Y,t)$ are relatively prime
    polynomials in $\C[X,Y,t]$. We say that $H(X,Y,t)$ is a  \emph{regular
    fraction} if $B(X,Y,t)$ is not divisible by the kernel polynomial
    $\Ktld(X,Y,t)$ in $\C[X,Y,t]$.
\end{defi}

\begin{rem} \label{exa:univ_regular}
  Since $S(X,Y)$ is not univariate, the kernel polynomial involves all three variables
  $X$,$Y$ and $t$, so does a multiple of $\Ktld(X,Y,t)$ (by a simple degree
   argument). Therefore, any fraction in $\C(X,t)$ or $\C(Y,t)$ is regular.
\end{rem}

 We endow the set of regular fractions in $\C(X,Y,t)$ with the
following equivalence relation: two regular fractions $H, G$ are
\emph{ equivalent} if there exists a regular fraction $R$ such that
$H-G=\Ktld (X,Y,t)R$.  We denote by $\mathcal{C}$ the set of
equivalence classes. Since the equivalence relation is compatible with
the addition and multiplication of fractions, one easily notes that
$\mathcal{C}$ can be endowed with a ring structure. Moreover, since
$\Ktld(X,Y,t)$ is irreducible in $\C[X,Y,t]$, any non-zero class is
invertible proving that $\mathcal{C}$ is a field.  Indeed, if $H$ is a
regular fraction that is not equivalent to zero, then one can write
$H=\frac{P}{Q}$ with $P,Q \in \C[X,Y,t]$ relatively prime and $\Ktld$
does not divide $P$ nor $Q$. Thus, the fraction $\frac{Q}{P}$ is
regular and its class in $\mathcal{C}$ is an inverse of the class of
$\frac{P}{Q}$. Moreover, since $\Ktld$ is not univariate, any non-zero
element in $\C(X,t)$ or $\C(Y,t)$ is a regular fraction which is not
equivalent to zero. Therefore, the fields $\C(X,t)$ and $\C(Y,t)$
embed into $\mathcal{C}$. By an abuse of notation, we denote by
$\C(X,t)$ and $\C(Y,t)$ their image in $\mathcal{C}$.

\begin{prop}\label{lem:liftingidentities}
  For a fraction $H$ in $\C(X,Y,t)$ and $(u,v)$ in $\cO$, the
  evaluation $H(u,v,\invS)$ of $H$ at $(u,v)$ is a well defined element
  of $\K$ if and only if $H$ is a regular fraction.

  The $\C$-algebra homomorphism $\phi:\mathcal{C} \rightarrow
  k(x,y), P(X,Y,t) \mapsto P(x,y,\invS)$ is well defined and is a field
  isomorphism which maps isomorphically $\C(t)$ onto $k=\C(S(x,y))$,
  $\C(X,t)$ onto $k(x)$ and $ \C(Y,t)$ onto $k(y)$.
\end{prop}
\begin{proof}
  Recall  that  by Theorem~\ref{thm:group_trans}, given a pair $(u,v) \in \cO$, there exists $\sigma \in G$
  such that $\sigma \cdot (x,y)= (u,v)$. The automorphism $\sigma$
  induces a $k$-algebra isomorphism between $k(x,y)$ and
  $k(u,v)$ so that the evaluation at $(x,y, \invS)$ composed by $\sigma$
  is the evaluation at $(u,v,\invS)$. Thereby, we only need to prove the
  first part of the proposition for the evaluation at $(x,y,\invS)$.

  Since $\Ktld(x,y,\invS)=0$, it is clear that one can not evaluate a fraction that is not regular. Thus, we only need to show that the evaluation of a regular fraction at $(x,y,\invS)$ is well defined.
  Let us write $H(X,Y,t) = \frac{A(X,Y,t)}{B(X,Y,t)}$ where $A(X,Y,t)$ and
  $B(X,Y,t)$ are relatively prime in $\C[X,Y,t]$, and the kernel polynomial $\Ktld(X,Y,t)$
  does not divide $B(X,Y,t)$. There exist two polynomials $U,V \in
  \C[X,Y,t]$ such that \[
    B(X,Y,t) U(X,Y,t) + \Ktld(X,Y,t) V(X,Y,t) = R(X,Y)
  \]
  with $R(X,Y) \in \C[X,Y]$ the resultant of $\Ktld(X,Y,t)$ and
  $B(X,Y,t)$ for the variable $t$. Since $\Ktld(X,Y,t)$ is an irreducible
  polynomial that does not divide
  $B(X,Y,t)$, the
  resultant $R(X,Y)$ is a nonzero polynomial. Since $x,y$ are algebraically independent over $\C$, one finds that $R(x,y) \neq 0$ and $K(x,y,\invS)=0$ which implies that $B\left(x,y,\invS\right) \neq 0$, so $H\left(x,y,\invS\right)$ is well defined.

  By Lemma~\ref{lem:ker_irred}, the kernel polynomial $\Ktld(X,Y,t)$ is
  irreducible as a
  polynomial in $\C(t)[X,Y]$.
  The ring
  $R=\C(t)[X,Y]/(\Ktld(X,Y,t))$ is therefore an integral domain. By
  \cite[page 9, (1K)]{Matsumura}, its fraction field is precisely
  $\mathcal{C}$. Now, the evaluation map from
  $\C(t)[X,Y]/(\Ktld(X,Y,t))$ to $k[x,y]$ is a ring isomorphism which
  maps isomorphically $\C(t)$ onto $k$.  The latter ring isomorphism
  extends to an isomorphism between the fraction fields $\mathcal{C}$ of
  $\C(t)[X,Y]/(\Ktld(X,Y,t))$ and the fraction field $k(x,y)$ of
  $k[x,y]$ which concludes the proof.
\end{proof}

If $H$ is a regular fraction, we denote $H_{(u,v)}$ its evaluation at
an element $(u,v)$ of the orbit and we can extend this evaluation by
$\C$-linearity to any $0$-chain $\gamma$.  We denote by $H_{\gamma}$
the corresponding element in $k(\cO)$. Such an evaluation is called an
\emph{orbit sum}.  We let the group $G$ act on $0$-chains by
$\C$-linearity, that is, $\sigma\left(\sum_{(u,v) \in \cO}
c_{(u,v)}(u,v) \right) = \sum_{(u,v) \in \cO} c_{(u,v)} \sigma \cdot
(u,v)$.  The following lemma shows that the evaluation morphism is
compatible with the action of $G$ on $k(\cO)$ and on $0$-chains.
\begin{lem}\label{lem::commutationgrp_eval_orb}
    Let $\sigma$ be an element of $G$, $\gamma$ be a $0$-chain, and $H(X,Y,t)$ be
    a regular fraction in $\C(X,Y,t)$. Then $\sigma(H_\gamma) = H_{\sigma \cdot \gamma}$.
\end{lem}
\begin{proof}
    Let $(u,v)$ be an element in the orbit.  Since $\sigma$ fixes $k = \C(S(x,y))$,
    we have  \[
        \sigma(H_{(u,v)}) = \sigma\left(H(u,v,\invS\right) =
        H\left(\sigma(u), \sigma(v), \invS\right)
    = H_{\sigma \cdot (u,v)}. \] The general case follows by $\C$-linearity.
\end{proof}

Two equivalent regular fractions have the same evaluation in
$k(\cO)$. Thereby, certain class of regular fractions can be
characterized by the Galoisian properties of their evaluation in
$k(\cO)$. This idea underlies the Galoisian study of invariants and
decoupling in Sections \ref{sect:ratinv} and \ref{sect:decoupling}. To
conclude, we want to compare the equivalence relation among regular
fractions that are elements of $\Cmul(X,Y)((t))$ and the
$t$-equivalence (see Section \ref{subsect:algebraicitystrategy} for
notation).

\begin{prop} \label{prop:link_pob_regular}
Let $F \in \Cmul(X,Y)((t))$ that is also a regular fraction in
$\C(X,Y,t)$. If $F$ is $t$-equivalent to $0$, that is, the
$t$-expansion of $F/\Ktld$ has poles of bounded order at $0$, then
the fraction $F / \Ktld$ is regular so that the regular fraction $F$
is equivalent to zero by definition.
  \end{prop}
\begin{proof}
 Our proof starts by following the lines of the proof of Lemma~2.6 in
\cite{BousquetMelouThreequadrant}. Assume that $F$ is $t$-equivalent
to $0$, so that there exists $H(X,Y,t) \in \Cmul(X,Y)((t))$ with poles
of bounded order at $0$ such
that \begin{equation}\label{eq:analyticequzero} F(X,Y)=\Ktld(X,Y,t)
H(X,Y,t).
\end{equation}

Analogous arguments to Lemma~2.6 in \cite{BousquetMelouThreequadrant}
show that there exists a root $\mathcal{X}$ of $K(.,Y,t)=0$ that is a
formal power series in $t$ with coefficients in an algebraic closure
of $\C(Y)$ and with constant term $0$. Since $H$ and $F$ have poles of
bounded order at $0$, one can specialize \eqref{eq:analyticequzero} at
$X=\mathcal{X}$ and find $F(\mathcal{X},Y,t)=0$. Writing $F
=\frac{P}{Q}$ where $P,Q \in \C[X,Y,t]$ are relatively prime, one
finds that $P(\mathcal{X},Y,t)=0$. Since $\Ktld(.,Y,t)$ is an
irreducible polynomial over $\C(Y,t)$ by Lemma~\ref{lem:ker_irred}, we
conclude that $\Ktld$ divides $P$. Because $P$ and $Q$ are relatively
prime, we find that $\Ktld$ doesn't divide $Q$ which concludes the
proof.
\end{proof}

Clearly, the regular fraction $\frac{\Ktld(X,Y,t)}{Y-t}$ is equivalent
to zero but not $t$-equivalent to zero, so the converse of
Proposition~\ref{prop:link_pob_regular} is false.  With the strategy
presented in Section \ref{subsect:alg_strat} in mind, we will use in
the next sections the notion of equivalence on regular fractions and
its Galoisian interpretation to produce pairs of \emph{Galois
invariants} and \emph{ Galois decoupling pairs}. For each pair of
Galois invariants and decoupling functions constructed for the models
presented in Section \ref{subsect:examples}, it happens that any
equivalence relation among these regular fractions is actually a
$t$-equivalence.  Unfortunately, we do not have any theoretical
arguments yet to explain this phenomenon.

The rest of the paper is devoted to the Galoisian interpretation of
the notions of invariants and decoupling.  Their construction relies
on the evaluation of regular fractions on suitable $0$-chains.

\section{Galois invariants} \label{sect:ratinv}

In this section, we prove that the finiteness of the orbit is
equivalent to the existence of a non-constant pair of Galois
invariants (see Theorem~\ref{thm:fried} below). This result
generalizes \cite[Theorem~7]{BBMR16} in the small steps case and was
proved in the more general context of finite algebraic correspondences
in \cite[Theorem~1]{FriedPoncelet}. Fried's framework is geometric,
but his proof is essentially Galois theoretic. We give here an
alternative presentation which does not require any algebraic
geometrical background. Moreover, we show in this section that if the
orbit is finite, the field of Galois invariants is of the form $k(c)$
for some element $c$ transcendental over $k$. In addition, we give an
algorithmic procedure to effectively construct $c$.

\subsection{Galois formulation of invariants}

In Section \ref{subsect:alg_strat}, we aimed at constructing
$t$-invariants that were rational fractions, that is, pairs
$(I(X,t),J(Y,t))$ satisfying an equation of the form $I(X,t) - J(Y,t)
= \Ktld(X,Y,t) R(X,Y,t)$ with $R$ having poles of bounded order at
zero ($I$ and $J$ are $t$-equivalent).  With the philosophy of Section
\ref{subsect:orbit sum} in mind, we introduce the weaker notion of
pair of \emph{Galois invariants} based on rational equivalence. Our
definition extends Definition~4.3 in \cite{BBMR16} to the large steps
context.

\begin{defi} \label{defi:rat_inv}
  Let $(I(X,t), J(Y,t))$ be a pair of rational fractions in
  $\C(X,t) \times \C(Y,t)$ (hence regular, as they are univariate).
  We say that this pair is a \emph{pair of Galois invariants} if
  there exists a regular fraction $R(X,Y,t)$
  such that $I(X,t) - J(Y,t) = \Ktld(X,Y,t) R(X,Y,t)$,
  that is, the regular fractions $I(X,t)$ and $J(Y,t)$ are equivalent.
\end{defi}

From Proposition~\ref{prop:link_pob_regular}, a pair of rational
$t$-invariants is a pair of Galois invariants. Therefore, it is
justified to look for a pair of Galois invariants first, and then to
check by hand if their difference is $t$-equivalent to $0$. Moreover,
the notion of pairs of Galois invariants is purely algebraic while the
notion of pairs of $t$-invariants involves some analytic
considerations which might be difficult to handle.  Using
Lemma~\ref{lem:liftingidentities}, the set of pairs of Galois
invariants corresponds to a subfield of $k(\cO)$ which can be easily
described.

\begin{prop} \label{prop:inv_to_ext}
  Let $P = (I(X,t),J(Y,t))$ be a pair of fractions in $\C(X,t)
  \times \C(Y,t)$. Then $P$ is a pair of Galois invariants if and only
  if the evaluations $I_{(x,y)}$ and $J_{(x,y)}$ are equal, and thus
  belongs to $k(x) \cap k(y) \subset k(\cO)$. Moreover, the pair $P$ is
  a constant pair of Galois invariants if and only if $I_{(x,y)} =
  J_{(x,y)}$ is in $k$.
\end{prop}

Therefore  we denote the  field $k(x) \cap k(y)$
as $\kinv$ and, by an abuse of terminology, call its elements  Galois invariants. The definition of the
group $G$ and the Galois correspondence applied to $k(\cO) |
k(x)$ and $k(\cO) | k(y)$ show that $f$ in $k(\cO)$ is a
Galois  invariant if and only if $f$ is fixed by $G$.  Moreover, Proposition~\ref{prop:inv_to_ext}
reduces the question of the existence of a  nonconstant  pair of Galois
invariants to the question of deciding whether $\kinv = k$ or not.

\subsection{Existence of nontrivial Galois invariants and finiteness of the orbit}\label{subsect:existencenontrivinv}
The existence  of a non-constant pair of  Galois invariants is equivalent to
the finiteness of the orbit  as proved in Theorem~1 and Lemma~p~470 in
\cite{FriedPoncelet} which holds also in positive characteristic and in a
higher dimensional context. Theorem~\ref{thm:fried} below is a rephrasing of
Fried's Theorem in pure Galois theoretic arguments.

\begin{thm} \label{thm:fried}
    The following are equivalent:
    \begin{enumerate}
        \item The orbit $\cO$ is finite.
        \item There exists a finite
          Galois extension $M$ of $k(x)$ and $k(y)$ such
            that $\Gal(M|k(x))$ and $\Gal(M|k(y))$ generate a finite
            group $\left<\Gal(M|k(x)),\Gal(M|k(y))\right>$ of automorphisms
            of~$M$.
        \item There exists a  nontrivial Galois invariant, that is, $k \subsetneq \kinv$.
    \end{enumerate}
\end{thm}
\begin{proof}

    (1) $\Rightarrow$ (2): Set $M = k(\cO)$.
    The group $G = \left<G_x,G_y\right>$ acts faithfully
    on the orbit, so it embeds as a subgroup of $S(\cO)$, the group
    of permutations of the pairs of the orbit $\cO$.
    The orbit is finite, therefore $G$ is finite.

    (2) $\Rightarrow$ (3): Write $H = \left<\Gal(M|k(x)),
      \Gal(M|k(y))\right>$.  By the same argument as in the beginning of Section
    \ref{sect:grouporbit}, the field $M^H$ is the field $\kinv$ of
    Galois invariants. Since $H$ is finite, the extension $M | \kinv$ is
    finite of degree $|H|$, hence the subextension $k(x) | \kinv$ is also
    finite. Since the extension $k(x) | k$ is transcendental by hypothesis
    on $\calW$, we conclude that $k \subsetneq \kinv$. Proposition
    \ref{prop:inv_to_ext} yields the existence of a pair of nontrivial
    Galois invariants.
    
    (3) $\Rightarrow$ (1): Let $(I(X,t),J(Y,t))$ be a pair of nontrivial
    Galois invariants. By the assumption on the model,
    $S(x,y)$ and $x$ are algebraically independent over $\C$. Since
    $I(x,\invS)$ is not in $\C(\invS)$
    by Lemma~\ref{lem:liftingidentities}, this implies that the extension
    $k(I(x,\invS))|k$ is transcendental. As the transcendence
    degree of $k(x)$ over $k$ is $1$, this implies that the extension
    $k(x) | k(I(x,\invS))$ is algebraic, hence $x$ is algebraic over $\kinv$,
    with minimal polynomial $P(X)$.

    The group $G$ leaves $\kinv$ fixed.  Thus
    the orbit of $x$ in $\K$ under the action of $G$ is a subset of
    the roots of $P(X)$. By Theorem~\ref{thm:group_trans}, the action of
    $G$ is transitive on the orbit, hence
    the set $G \cdot x = \{u \in \K \st \exists \sigma \in G, u =
    \sigma x\} = \{u \in \K \st \exists v \in \K,\, (u,v) \in \cO\}$ is
    finite.  As there are $\deg_Y \Ktld(X,Y,t)$ pairs of the orbit
    with first coordinate $u$ for each $u$ in $G \cdot x$, we
    conclude that $\cO$ is finite.
\end{proof}

In the rest of the paper, we assume that the orbit is finite. Theorem~\ref{thm:fried}
implies that the extension $k(\cO)| \kinv$ is finite and  Galoisian with  Galois group
$G = \left<G_x, G_y\right>$.

\subsection{Effective construction}\label{sect:effectiveinvariants}

In order to apply the algebraic strategy presented in Section
\ref{subsect:alg_strat}, we want to find explicit nonconstant rational
$t$-invariants. As already mentioned, we shall first construct explicitly
the field of Galois invariants and then search among these Galois
invariants the potential rational $t$-invariants.

In the small steps case, an orbit sum argument was used to
construct a pair of Galois invariants \cite[Theorem~4.6]{BBMR16}. This
construction generalizes mutatis mutandis to the large steps case, and is reproduced
here to show one way to exploit the group of the walk.

\begin{lem}\label{lem:definitionsinv}
    Let $\sinv$ be the $0$-chain $\frac{1}{|\cO|}\sum_{a \in \calO} a$. Then,
    for any regular fraction $H \in \Q(X,Y,t)$ the element $H_{\sinv}$ is a
    Galois invariant.
\end{lem}
\begin{proof}
  Let $H(X,Y,t)$ be a regular fraction. Since,  by Theorem~\ref{thm:group_trans}, the group $G$ acts faithfully on $\cO$,
  the $0$-chain
    $\sinv$ is invariant by the action of $G$. Thus, by Lemma
    \ref{lem::commutationgrp_eval_orb}, for all $\sigma$ in $G$,
    $\sigma\left(H_{\sinv}\right) = H_{\sigma \cdot \sinv} = H_{\sinv}$.
    Therefore, by the Galois correspondence, $H_{\sinv}$ is a Galois
    invariant.
  \end{proof}

Unfortunately, a non-constant regular fraction $H$ might have a constant
evaluation, that is,  $H_\sinv$ might belong to  $k$. Thus, one has to choose
carefully $H$ in order to avoid this situation which  is precisely the strategy
used in \cite[Theorem~4.6]{BBMR16}. Below,    we  describe  an alternative construction
which is easier to compute effectively and  yields a complete description of
the field $\kinv$.

Consider first this simple observation. Since $x$ is algebraic over $\kinv$,  we
can consider its minimal polynomial $\mu_{x}(Z)$ in $\kinv[Z]$. One of its
coefficients must be in $\kinv \setminus k$
because $x$ is transcendental over $k$.
Thus, such a coefficient is a non-trivial  Galois invariant.

A more sophisticated argument using a constructive version of Lüroth’s Theorem
says actually much more about such a coefficient.
\begin{thm}[Lüroth's Theorem \cite{Rotman}, Th. 6.66] \label{prop:gen_kinv}
    Let $k(x)$ be a field with $x$ transcendental over $k$
    and $k \subset K \subset k(x)$ a subfield.
    If $x$ is algebraic over $K$, then
    any  coefficient $c$ of its minimal polynomial $\mu_{x}(Z)$ over $K$ that is not in $k$
    is such that $K = k(c)$.
\end{thm}
Applying this result to the tower $k \subset \kinv \subset k(x)$,
not only can we find nontrivial  Galois invariants among the coefficients of
$\mu_x$, but any one of them generates the field of Galois
invariants. In one sense, these coefficients
contain all the information on the Galois   invariants attached to the model.
Therefore, all that remains is to compute effectively the polynomial $\mu_x(Z)$.

By irreducibility of the polynomial $\mu_x(Z)$ in $\kinv[Z]$, the
Galois group $G = \Gal(k(\cO) |\kinv)$ acts transitively on its
roots. By Theorem~\ref{thm:group_trans}, the orbit of $x$ under the
action of $G$ is the set of left coordinates of the orbit. Therefore,
$\mu_x(Z)$ is precisely the vanishing polynomial of the left
coordinates of the orbit,
which is exactly computed in the construction of the orbit in
\cite[Section 3.2]{bostan2018counting}.  We detail
this construction in Appendix~\ref{sect:decoupl_comp}.

In order to find an explicit pair of non-constant Galois invariants
$(I(X,t), J(Y,t))$, we only need to apply
Proposition~\ref{lem:liftingidentities} to lift to $\C(X,t)$ and to
$\C(Y,t)$ any non-constant coefficient of the polynomial $\mu_x(Z) \in
\kinv[Z]$. The lifts of the polynomial $\mu_x[Z]$ to $\C(X,t)[Z]$ and
to $\C(Y,t)[Z]$ can be computed directly when constructing the orbit,
see \ref{app:varietiesorbit}.

\begin{exa}[The model $\Gmod$]
  Consider the model $\Gmod$. Its orbit type is  $\cO_{12}$. We compute
  the lift of $\mu_x(Z)$ in $\C(X,t)[Z]$ as
    \begin{equation*}
      \begin{split}
        &Z^{6}\boxed{-\frac{\left(\lambda^{2}   \,X^{3}+X^{6}+1  \,X^{4}- X^{2}-1\right) t^{2}+X^{2} \lambda  \left(X^{2}-1 \right) t -X^{3}}{t^{2} X \left(X^{2}+1 \right)^{2}}}Z^5+\frac{t +\lambda}{t} Z^4 \\
        &-2 \frac{X^{6} t^{2} +\left(-\frac{\lambda^{2}  \,t^{2}}{2}+\frac{1}{2}\right) X^{5}+t \left( t +\lambda \right) X^{4}+\left(-t^{2}-\lambda  \, t \right) X^{2}-\frac{\left(\lambda^{2}  \,t^{2}-1\right) X}{2}-t^{2}}{t^{2} X \left(X^{2}+1 \right)^{2}}Z^3\\
        &-\frac{\left(  t +\lambda \right) Z^{2}}{t}-\frac{\left(\left(\lambda^{2}   \,X^{3}+X^{6}+\,X^{4}-X^{2}-1\right) t^{2}+X^{2} \lambda  \left(X^{2}-1 \right) t -X^{3}\right)}{t^{2} X \left(X^{2}+\right)^{2}}Z-1\\
        \end{split}
      \end{equation*}
      and in $\C(Y,t)[Z]$ as
      \begin{equation*}
        \begin{split}
        &Z^{6}+\boxed{\frac{-t \,Y^{4}+\lambda  t Y +Y^{3}+t}{t \,Y^{2}}}Z^5+\frac{t +\lambda}{t}Z^4
        -2\frac{\left(Y^{4} -\frac{1}{2} Y^{2} \lambda^{2}-Y \lambda -1\right) t^{2}-t \,Y^{3}+\frac{Y^{2}}{2}}{t^{2} Y^{2}}Z^3 \\
        &-\frac{\left(t +\lambda \right) }{t}Z^2+\frac{ \left(-  t \,Y^{4}+\lambda  t Y +Y^{3}+t \right) }{t \,Y^{2}}Z-1.
      \end{split}
    \end{equation*}

    The coefficient of $Z^5$ is nonconstant, hence we have the following pair of non-trivial Galois
    invariants $(I(X,t),J(Y,t))$
    \[\left(-\frac{\left(\lambda^{2} \,X^{3}+X^{6}+1  \,X^{4}- X^{2}-1\right) t^{2}+X^{2} \lambda  \left(X^{2}- \right) t -X^{3}}{t^{2} X \left(X^{2}+1 \right)^{2}},\frac{- t \,Y^{4}+\lambda  t Y +Y^{3}+t}{t \,Y^{2}}\right).\]
    We check that $\frac{I(X,t)-J(X,t)}{\Ktld(X,Y,t)}$ has poles of bounded order at $0$,
    hence  $(I(X,t),J(Y,t))$ is a pair of $t$-invariants.
    Moreover, Theorem~\ref{prop:gen_kinv} says that $\kinv = k\left(I(x,\invS)\right)$,
    so any pair of Galois invariants for $\Gmod$  is a  fraction
    in the pair $(I(X,t),J(Y,t))$.
\end{exa}

\begin{exa}
 The orbit type of the model with step polynomial
$S(X,Y)=X+\frac{X}{Y}+\frac{Y}{X^2}+\frac{1}{X^2}$ is $\cO_{18}$ (see
Figure~\ref{fig:sample_orbits}).  With our method, we find the
following pair of Galois invariants
\[\left(\frac{\left(-X^{9}-3 X^{3}+1\right) t^{2}+\left(X^{8}+X^{5}-2
X^{2}\right) t +X^{4}}{X^{6} t^{2}},\frac{\left(Y^{3}+3 Y +1\right)
\left(Y +1\right)^{3} t^{3}+Y^{4}}{Y^{2} t^{3} \left(Y
+1\right)^{3}}\right). \] One can also check by looking at the
$t$-expansions that it is a pair of $t$-invariants.
\end{exa}

\section{Decoupling} \label{sect:decoupling}

In this section, we study the Galoisian formulation of the  notion of decoupling introduced in Section \ref{subsect:algebraicitystrategy}.  In particular, assuming the
finiteness of the orbit, we show how the Galois decoupling of a rational
fraction $H(X,Y,t)$ can be, analogously to the small steps case,
tested and constructed if it exists via the evaluation on certain
$0$-chains on the orbit.

\subsection{Galois formulation of decoupling}

As in the previous section, we adapt the notion of decoupling
introduced in Section \ref{subsect:algebraicitystrategy} to our
Galoisian framework. The definition below is the straightforward
analogue of Definition~4.7 in \cite{BBMR16} for large steps models.

\begin{defi}[Galois decoupling of a fraction]
  Let $H(X,Y,t)$ be a regular  fraction in $\C(X,Y,t)$.
  A pair of fractions $(F(X,t),G(Y,t))$ in $\C(X,t) \times \C(Y,t)$
    is called a \emph{Galois decoupling pair} for the fraction $H$ if there exists
    a regular fraction $R(X,Y,t)$  satisfying
    \[H(X,Y,t) = F(X,t) + G(Y,t) + \Ktld(X,Y,t) R(X,Y,t).\]
    We call such an  identity a \emph{Galois decoupling} of the fraction $H$.
\end{defi}

Thanks to Proposition~\ref{prop:link_pob_regular}, if a regular  fraction admits a decoupling
with respect to the $t$-equivalence then it admits a Galois decoupling. Analogously to the notion of Galois invariants and as a corollary of
Proposition~\ref{lem:liftingidentities}, one can interpret the Galois decoupling
as an identity in the extension $k(\cO)$.

\begin{prop} \label{prop:decoupl_extension}
    Let $H$ be a regular fraction in $\C(X,Y,t)$. Then $H$ admits a
    Galois decoupling if and only if $H_{(x,y)}$ can be written as $f +
    g$ with $f$ in $k(x)$ and $g$ in $k(y)$.
\end{prop}

By an abuse of  terminology, we call any identity $H_{(x,y)} = f + g$ with
$f$ in $k(x)$ and $g$ in $k(y)$ a Galois decoupling of $H$. Furthermore,
these last two conditions can be reformulated algebraically via the Galois correspondence
applied to the extensions $k(\cO) | k(x)$ and $k(\cO) | k(y)$:
$H_{(x,y)} = f + g$ with $f$ fixed by $G_x$ and $g$ fixed by $G_y$.

Given a regular fraction $H$, one could try to use the normal basis theorem (see
\cite[chapter 6, § 13]{Lang}) to test the existence of a Galois decoupling
for $H$. The normal basis theorem  states that there exists a $\kinv$-basis of $k(\cO)$ of the form
$(\sigma(\alpha))_{\sigma \in G}$ for some $\alpha \in k(\cO)$.
The action of $G_x$ and $G_y$ on this basis is given by
permutation matrices, and thus the linear constraints for the Galois decoupling of $H_{(x,y)}$
is equivalent to a system of linear equations.
Unfortunately the computation of a normal basis requires \emph{a priori}
a complete knowledge of the Galois group $G$, whose computation is a difficult
problem. Therefore, we present in the rest of the section a construction of a Galois decoupling
test which relies entirely on the orbit and its Galoisian structure.

\subsection{The decoupling of $(x,y)$ in the orbit}

\begin{defi}
    Let $\alpha$ be a $0$-chain of the orbit. We say that
    $\alpha$ \emph{cancels decoupled fractions} if $H_{\alpha} = 0$
    for any regular fraction $H(X,Y,t)$ of $\C(X,t) + \C(Y,t)$.
\end{defi}

\begin{figure}[ht]
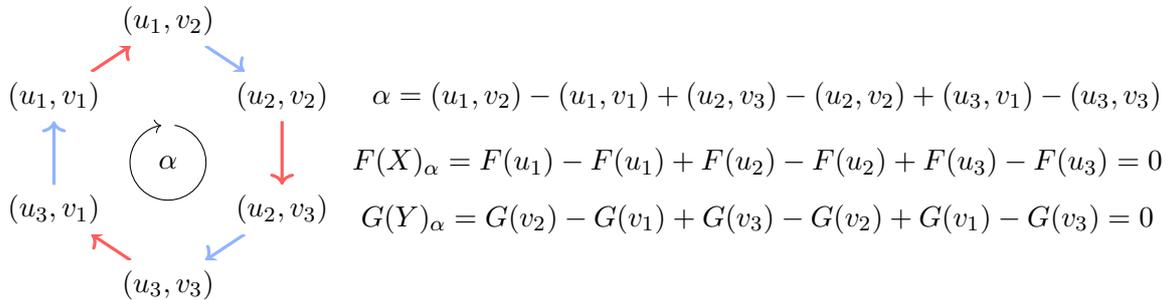

    \ctikzfig{cycle_alterne}
    \caption{The $0$-chain induced by a bicolored loop cancels decoupled fractions}
\label{fig:bicol_cycle}
\end{figure}

  We recall that a \emph{path} in the graph of the orbit is a sequence
of vertices $(a_1, a_2, \dots, a_{n+1})$ such that $a_i \sim a_{i+1}$
for all $0 \le i \le n$.  The length of $(a_1, a_2, \dots, a_{n+1})$ is the number of
adjacencies (that is $n$). A path is called a \emph{loop}
\footnote{We know that the  terminology \emph{loop} is unorthodox, however
 we follow \cite[Definition~1.8]{Giblin}.}
if $a_{n+1}=a_1$. A loop is called \emph{simple} if only its first and last vertices are equal.

\begin{exa} \label{exa:bicol_cycle}
  A \emph{bicolored loop} is a loop $(a_1, a_2, \dots, a_{2n+1})$ of
  even length such that for all $i$, $a_{2i} \sim^x a_{2i-1}$ and
  $a_{2i+1} \sim^y a_{2i}$. One associates to $ (a_1, a_2, \dots,
  a_{2n+1}) $ the $0$-chain \[ \alpha = \sum_{i=1}^{2n} (-1)^i a_i =
    \sum_{i=1}^{n} (a_{2i} - a_{2i-1}) = \sum_{i=1}^{n}
    (a_{2i}-a_{2i+1}).\]

  Taking $F(X,t)$ a regular fraction in $\C(X,t)$, one observes that
  for all $i$, $F_{a_{2i}} - F_{a_{2i-1}} = 0$, as vertices $a_{2i}$ and
  $a_{2i-1}$ share their first coordinate.  Symmetrically, taking
  $G(Y,t)$ a regular fraction in $\C(Y,t)$, $G_{a_{2i+1}} - G_{a_{2i}} =
  0$.  Therefore, $F_{\alpha} = G_{\alpha} = 0$. Hence, the
  \emph{$0$-chains induced by bicolored loops} cancel decoupled fractions.
  Figure~\ref{fig:bicol_cycle} illustrates this observation.
\end{exa}

Example \ref{exa:bicol_cycle} is fundamental for picturing the $0$-chains
  that cancel decoupled fractions because of the following stronger result:
\begin{prop} \label{prop:alt_cycles}
  A $0$-chain cancels decoupled fractions if and only if it can be
  decomposed as a $\C$-linear\footnote{Note that if the $0$-chain is
    with integer coefficients, one can choose the combination with integer
    coefficients as well.} combination of $0$-chains induced by bicolored
  loops.
\end{prop}
There exists an elementary graph theoretic proof of this
fact. However, we choose to postpone the proof of
Proposition~\ref{prop:alt_cycles} after the proof of
Theorem~\ref{thm:cycloloc}, which is an algebraic reformulation of the
condition for a $0$-chain to cancel decoupled fractions.

\begin{exa}\label{exa:obstructiondecouplingsquare}
  A straightforward application of this observation,
  is the following obstruction for the existence of a Galois decoupling of $ XY$.
  Consider an orbit whose graph contains a bicolored square
  (bicolored loop of length 4), with associated $0$-chain
  $\alpha = (u_1, v_1) - (u_1, v_2) + (u_2,v_2) - (u_2,v_1)$
  (thus with $u_1 \neq u_2$ and $v_1 \neq v_2$).
  The evaluation of $XY$ on this $0$-chain factors as $(XY)_{\alpha} = (u_1-u_2)(v_1-v_2)$,
  which is always nonzero.
  Therefore, if the orbit of a model $\calW$ contains a bicolored square, then
  $XY$ never admits a Galois decoupling and thereby a decoupling in the sense of the $t$-equivalence.
  Thus, we can conclude that for models with orbit $\widetilde{\cO_{12}}$
  (see Figure~\ref{fig:sample_orbits})
  or Hadamard (see Section~\ref{subsubsect:Hadamard}),
  or the ``Fan model'' (see Appendix~\ref{subsubsect:fanmod}),
  the fraction $XY$ never admits a decoupling.
\end{exa}

For now, we only saw that the canceling of a regular fraction on $0$-chains that cancel
decoupled fraction gives a necessary condition for the
Galois decoupling of this fraction. We prove in this section that this
condition is  in fact sufficient and that one only needs to consider the
evaluation on a single $0$-chain.

For small steps walks with finite orbit, there is only one bicolored
loop and thereby only one $0$-chain $\alpha$ induced by  the bicolored
loop.  Theorem~4.11 in \cite{BBMR16} shows that a regular fraction
admits a Galois decoupling if and only its evaluation on $\alpha$ is
zero.  More precisely, Bernardi, Bousquet-Mélou and Raschel proved an
explicit identity in the algebra of the group of the
walk. Rephrasing their equality in terms of $0$-chains in the orbit,
we introduce the notion of decoupling of the pair $(x,y)$ in the orbit
as follows:

\begin{defi}[Decoupling of $(x,y)$] \label{defi:decoupl_xy}
    We say that $(x,y)$ \emph{admits a decoupling in the orbit} if
    there exist $0$-chains $\wgammax$, $\wgammay$, $\alpha$ such that
    \begin{itemize}
        \item $(x,y) = \wgammax + \wgammay + \alpha$
        \item $\sigma_x \cdot \wgammax = \wgammax$ for all $\sigma_x \in G_x$
        \item $\sigma_y \cdot \wgammay = \wgammay$ for all $\sigma_y \in G_y$
        \item the $0$-chain $\alpha$ cancels decoupled fractions
        \end{itemize}

    In that case, we call the identity $(x,y) = \wgammax + \wgammay + \alpha$ a \emph{decoupling
    of $(x,y)$}.
\end{defi}
Note that if $(\wgammax,\wgammay,\alpha)$ is a decoupling of $(x,y)$
then the $0$-chain $\alpha$ is equal to $(x,y) -\wgammax -\wgammay$.
Hence, when giving such a decoupling, we will often state explicitly
only $\wgammax$ and $\wgammay$.
\begin{exa}
  For the orbit of the model $\Gmod$,
  a decoupling equation is as follows:
  $(x,y) = \wgammax + \wgammay + \alpha$ with
  \begin{align*}
    \wgammax &= \left(
               \frac{1}{2} \left((x,y) + (x,\overline{xy}) \right)
               -  \frac{1}{8} \left((z,y) + (-\overline{xy^2z},y) + (xy^2z,\overline{xy}) + (-\overline{z},\overline{xy}) \right) \right. \\
             & \left. + \frac{1}{8} \left(
               (xy^2z,\overline{yz}) + (z,\overline{yz})
               + (-\overline{xy^2z},-xyz)
               + (-\overline{z}, -xyz) \right) \right) \\
    \mbox{and }\wgammay &= \left(
                          \frac{1}{4} \left((x,y) + (z,y) + (-\overline{xy^2z},y)\right) - \frac{1}{4} \left( (x,\overline{xy}) + (z,\overline{yz}) + (-\overline{xy^2z}, -  xyz)\right) \right). \\
  \end{align*}

  This decoupling is constructed in
  Subsection~\ref{subsect:decouplingO12} and the $0$-chain $\alpha$ is
  represented in Figure~\ref{fig:bicol_cycles_gmod}. It is the sum of
  the two $0$-chains $\alpha_1$ and $\alpha_2$ induced by the two
  bicolored loops where the weights of the $0$-chains $\alpha_1$ and
  $\alpha_2$ are written in grey next to their corresponding vertex.

  \begin{figure}[h!]
    \scalebox{0.7}{\begin{tikzpicture}[scale=1.5]
	\begin{pgfonlayer}{nodelayer}
		\node [style=paire] (0) at (0, 4) {$(x,y_1)$};
		\node [style=paire] (1) at (-1.25, 5.75) {$(x_3,y_1)$};
		\node [style=paire] (2) at (1.25, 5.75) {$(x_4,y_1)$};
		\node [style=paire] (3) at (0, 1.25) {$(x,y)$};
		\node [style=paire] (4) at (-1, -0.75) {$(x_1,y)$};
		\node [style=paire] (5) at (1, -0.75) {$(x_2,y)$};
		\node [style=paire] (6) at (-3, -3) {$(x_1,y_2)$};
		\node [style=paire] (7) at (-5.5, -3.5) {$(x_3,y_2)$};
		\node [style=paire] (8) at (-3.5, -5.5) {$(x_5,y_2)$};
		\node [style=paire] (9) at (3, -3) {$(x_2,y_3)$};
		\node [style=paire] (10) at (3.5, -5.5) {$(x_5,y_3)$};
		\node [style=paire] (11) at (5.5, -3.5) {$(x_4,y_3)$};
		\node [style={clear_text}] (12) at (1.25, 1.25) {$1$};
		\node [style={clear_text}] (29) at (-1.25, 1.25) {$1$};
		\node [style={clear_text}] (13) at (2.5, -0.75) {$-1$};
		\node [style={clear_text}] (14) at (-2.5, -0.75) {$-1$};
		\node [style={clear_text}] (15) at (-4, -2.5) {$1$};
		\node [style={clear_text}] (22) at (4, -2.5) {$1$};
		\node [style={clear_text}] (16) at (-2.25, -5.5) {$0$};
		\node [style={clear_text}] (17) at (-7, -3.25) {$-1$};
		\node [style={clear_text}] (18) at (-2.5, 6) {$1$};
		\node [style={clear_text}] (19) at (-1.25, 4) {$-1$};
		\node [style={clear_text}] (28) at (1.25, 4) {$-1$};
		\node [style={clear_text}] (20) at (2.5, 6) {$1$};
		\node [style={clear_text}] (21) at (6.75, -3.25) {$-1$};
		\node [style={clear_text}] (23) at (2.25, -5.5) {$0$};
		\node [style=none] (30) at (-3, 1) {$\alpha_1$};
		\node [style=none] (31) at (3, 1) {$\alpha_2$};
	\end{pgfonlayer}
	\begin{pgfonlayer}{edgelayer}
		\draw [style=usimx] (4) to (3);
		\draw [style=usimx] (3) to (5);
		\draw [style=usimx] (5) to (4);
		\draw [style=usimx] (1) to (0);
		\draw [style=usimx] (0) to (2);
		\draw [style=usimx] (2) to (1);
		\draw [style=usimx] (6) to (7);
		\draw [style=usimx] (7) to (8);
		\draw [style=usimx] (8) to (6);
		\draw [style=usimx] (9) to (10);
		\draw [style=usimx] (10) to (11);
		\draw [style=usimx] (11) to (9);
		\draw [style=usimy] (4) to (6);
		\draw [style=usimy] (5) to (9);
		\draw [style=usimy, bend right, looseness=0.50] (8) to (10);
		\draw [style=usimy] (3) to (0);
		\draw [style=usimy, bend right=45, looseness=0.75] (1) to (7);
		\draw [style=usimy, bend left=45, looseness=0.75] (2) to (11);
		\draw [style=none,->] (30) +(80:1cm) arc(80:-260:1cm);
		\draw [style=none,->] (31) +(80:1cm) arc(80:-260:1cm);
	\end{pgfonlayer}
\end{tikzpicture}}
    \caption{A $0$-chain of $\calO_{12}$ characterizing decoupled
      fractions  for the model $\Gmod$
      where the weights are
      written in grey next to their corresponding vertex.}

    \label{fig:bicol_cycles_gmod}
  \end{figure}
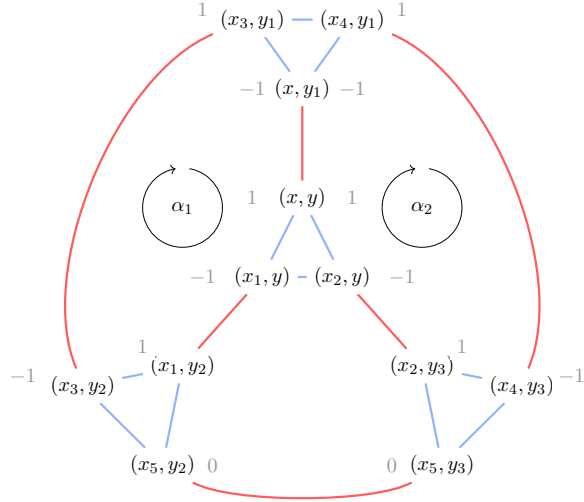
\end{exa}

The relation between the notion of decoupling of $(x,y)$ in the orbit and the notion of Galois decoupling is  detailed in the following proposition.
\begin{prop} \label{prop:crit_decoupl}
    Assume that $(x,y) = \wgammax + \wgammay + \alpha$ is a decoupling of $(x,y)$,
    and let $H(X,Y,t)$ be a regular fraction.
    Then the following assertions are equivalent:
    \begin{itemize}
    \item[(1)] $H$ admits a  Galois decoupling
    \item[(2)] $H_{\alpha} = 0$
    \item[(3)] $H_{(x,y)} = H_{\wgammax} + H_{\wgammay}$ is a Galois decoupling of $H$.
    \end{itemize}
\end{prop}
\begin{proof}
  (3) $\Rightarrow$ (1) is obvious.

  (1) $\Rightarrow$ (2): By definition of a Galois decoupling of $(x,y)$, $\alpha$
  cancels decoupled fractions.

  (2) $\Rightarrow$ (3): Evaluating $H$ on the decoupling of $(x,y)$ yields
  the identity $H_{(x,y)} = H_{\wgammax} + H_{\wgammay}$. Moreover,
  since $\wgammax$ (resp. $\wgammay$) is fixed by $G_x$ (resp. $G_y$), then
  Lemma~\ref{lem::commutationgrp_eval_orb} and the Galois correspondence
  in the extensions $k(\cO)|k(x)$ and $k(\cO)|k(y)$ ensure that $H_{\wgammax}$ and
  $H_{\wgammay}$ belong respectively to $k(x)$ and $k(y)$, hence
  $H_{\wgammax} + H_{\wgammay}$ is a Galois decoupling of $H$.
\end{proof}

Therefore, if we solve the  decoupling problem  of $(x,y)$ in the orbit, we also
solve the  Galois decoupling problem  for rational fractions: an explicit decoupling of $(x,y)$
will grant us with a simple test to check whether a regular fraction admits a Galois decoupling
(some orbit sum is zero), and an effective way to construct the associated Galois
decoupling based on orbit sum computations. We now state the main result of this section, whose proof will follow
from Theorem~\ref{thm:path_decoupl}.

\begin{thm}[Decoupling] \label{thm:decoupling_orb}
    If the orbit $\cO$ is finite, then $(x,y)$ always admits a decoupling in the orbit with rational coefficients.
\end{thm}

The rest of this section is dedicated to the proof of Theorem~\ref{thm:decoupling_orb}
and to the effective construction of the decoupling of $(x,y)$ in the orbit.

\subsection{Pseudo-decoupling}

We define here a more flexible notion of decoupling in the orbit
called \emph{pseudo-decoupling},
mainly used in the proof of the Theorem~\ref{thm:decoupling_orb}.
\begin{defi}[Pseudo-decoupling]
    Let $\gamma_x$ and $\gamma_y$ be two $0$-chains.
    We call the pair $(\gamma_x,\gamma_y)$ a \emph{pseudo-decoupling}
    of $(x,y)$ if for every regular fraction $H(X,Y,t)$
    that admits a Galois
    decoupling, the equation $H_{(x,y)} = H_{\gamma_x} + H_{\gamma_y}$ is
    a  Galois decoupling of $H$, that is, $H_{\gamma_x} \in k(x)$ and $H_{\gamma_y} \in k(y)$.
  \end{defi}

For instance, if $(x,y) = \wgammax + \wgammay + \alpha$ is a decoupling of $(x,y)$, then
the pair $(\wgammax, \wgammay)$ is a pseudo-decoupling of $(x,y)$
by Proposition~\ref{prop:crit_decoupl}.

Theorem \ref{thm:pseudec_is_dec} below shows how a pseudo-decoupling yields a decoupling. First  let us give some notation. Let $G'$ be a subgroup of $G$. We denote by $\left[G'\right]$
the formal sum $\frac{1}{|G'|} \sum_{\sigma \in G'} \sigma$.
From a Galois theoretic point of view, if $G'$ is the Galois group of some subextension $k(\cO)|M$,
then  $[G']$  is the trace of the field extension $k(\cO)|M$.
\begin{thm} \label{thm:pseudec_is_dec}
    If a pair $(\gamma_x, \gamma_y)$ is a pseudo-decoupling of $(x,y)$, then
    $(x,y)$ admits a decoupling of the form \[
        (x,y) = \wgammax + \wgammay + \alpha
    \]
    where $\wgammax =  [G_x] \cdot \gamma_x$ and
    $\wgammay = [G_y] \cdot \gamma_y$.
  \end{thm}
\begin{proof}
    By construction, the $0$-chains $\wgammax$ and $\wgammay$
    are fixed under the respective actions of $G_x$ and $G_y$.
    Therefore, we only need to prove that $\alpha$ cancels
    decoupled fractions, for which purpose we rewrite it as the sum of three terms
    \[\alpha = ((x,y) - \gamma_x - \gamma_y) + (\gamma_x - [G_x] \cdot \gamma_x)
    + (\gamma_y - [G_y] \cdot \gamma_y).\]

    Let $H$ be a regular fraction that admits a Galois decoupling. Then
    \begin{itemize}
        \item $H_{(x,y)} - H_{\gamma_x} - H_{\gamma_y} = 0$ by definition
            of the pseudo-decoupling $(\gamma_x, \gamma_y)$.
          \item  For $\sigma_x$ in $G_x$, we compute
                $H_{\gamma_x - \sigma_x \cdot \gamma_x}
                = H_{\gamma_x} - \sigma_x(H_{\gamma_x})$
                thanks to Lemma~\ref{lem::commutationgrp_eval_orb}.
                As $H_{\gamma_x}$ is in $k(x)$,
                it turns out that $H_{\gamma_x - \sigma_x \cdot \gamma_x}$ is zero.
            Since $\frac{1}{|G_x|} \sum_{\sigma_x \in G_x} (\gamma_x - \sigma_x \cdot \gamma_x)
            = \gamma_x - [G_x] \cdot \gamma_x$, we obtain that $\gamma_x - [G_x] \cdot \gamma_x$
            cancels $H$.
        \item The argument for $\gamma_y - [G_y] \cdot \gamma_y$ is similar.
    \end{itemize}
    Thus $H_{\alpha} = 0$,
    which concludes the proof.
\end{proof}

We finish this subsection with two important lemmas.

\begin{lem} \label{lem:trans_pseudodecoupl}
    If the pair $(\gamma_x, \gamma_y)$ is a pseudo-decoupling of $(x,y)$, and
    $\alpha$ and $\alpha'$ are $0$-chains that cancel decoupled
    fractions, then $(\gamma_x + \alpha, \gamma_y + \alpha')$ is also
    a pseudo-decoupling of $(x,y)$.
\end{lem}
\begin{proof}
  Let $H(X,Y,t)$ be a regular fraction that
  admits a Galois decoupling. By definition of
  $\alpha$ and $\alpha'$, we have
  $H_\alpha = H_{\alpha'} = 0$, which
    by linearity proves that $H_{\gamma_x+\alpha} = H_{\gamma_x}$ and
    $H_{\gamma_y + \alpha'} = H_{\gamma_y}$. Since $(\gamma_x, \gamma_y)$
    is a pseudo-decoupling of $(x,y)$, the equation
    $H_{(x,y)} = H_{\gamma_x} + H_{\gamma_y} = H_{\gamma_x + \alpha} + H_{\gamma_x + \alpha'}$
    is a  Galois decoupling of $H$ proving that $(\gamma_x + \alpha, \gamma_y + \alpha')$ is also
    a pseudo-decoupling of $(x,y)$.
 \end{proof}

\begin{lem} \label{lem:cocycle_cond}
    If two $0$-chains $\gamma_x$ and $\gamma_y$ satisfy the following conditions
    \begin{itemize}
        \item $(x,y) = \gamma_x + \gamma_y$
        \item for all $\sigma_x \in G_x$, the $0$-chain $\sigma_x \cdot \gamma_x - \gamma_x$
            cancels decoupled fractions
        \item for all $\sigma_y \in G_y$, the $0$-chain $\sigma_y \cdot \gamma_y - \gamma_y$
          cancels decoupled fractions
    \end{itemize}
    then $(\gamma_x, \gamma_y)$ is a pseudo-decoupling of $(x,y)$.
\end{lem}
\begin{proof}
    Let $H$ be a regular fraction which admits a Galois decoupling.
    As $H_{(x,y)} = H_{\gamma_x} + H_{\gamma_y}$ from the first point,
    one only needs to show that $H_{\gamma_x}$ is in $k(x)$ and that $H_{\gamma_y}$ is in $k(y)$.
    Let $\sigma_x$ be in $G_x$, then $\sigma_x(H_{\gamma_x}) =
    H_{\sigma_x \cdot \gamma_x} = H_{(\sigma_x \cdot \gamma_x - \gamma_x)
      + \gamma_x} = H_{\gamma_x}$ because $(\sigma_x \cdot \gamma_x -
    \gamma_x)$ cancels decoupled fractions. Therefore, the Galois
    correspondence proves that $H_{\gamma_x}$ is in $k(x)$.  The same
    argument proves that $H_{\gamma_y}$ is in $k(y)$.
\end{proof}

\subsection{Graph homology and construction of the decoupling}

Our construction of a decoupling relies on the graph structure of the
orbit $\cO$, and in particular on the formalism of graph homology.

\subsubsection{Basic graph homology}\label{subsec:basicgraphhomology}

We recall here the basic definitions of graph homology and the
properties that will be used in the construction of the decoupling
(see \cite{Giblin} for a comprehensive introduction to graph
homology).

\begin{defi}
  A \emph{graph} (undirected) is a pair $\Gamma = (V, E)$ where $V$
  is the set of vertices and $E \subset \left\{ \{a,a'\} \st a, a' \in
    V, a \neq a'\right\}$ is the set of edges. A \emph{subgraph} of
  $\Gamma$ is a graph $\Gamma' = (V',E')$ such that $V' \subset V$ and
  $E' \subset E$.

  An \emph{oriented graph} is a pair $\Gamma = (V,E^+)$ where $V$ is
  the set of vertices and $E^+ \subset \left\{(a,a') \st a,a' \in V, a
    \neq a'\right\}$ the set of \emph{arcs} (oriented edges) such that if
  $(a,a') \in E^+$ then $(a',a) \notin E^+$.  An \emph{orientation} of a
  graph $\Gamma = (V,E)$ is an oriented graph $\Gamma' = (V,E^+)$ such
  that the map $E^+ \rightarrow E$ which maps $(a,a')$ to $\{a,a'\}$ is
  a bijection.
\end{defi}
  Note that every graph can be given an orientation by freely choosing
an origin for each edge. Conversely, given an oriented graph
$\Gamma=(V, E^{+})$, one can consider the associated undirected graph
$(V,E)$ where $E= \{ \{a,a'\} \mbox{ such that } (a,a') \in E^+ \mbox{
or } (a',a) \in E^+ \}$. In what follows, the notions of graph
homomorphism, path, connected components concern the structure of
undirected graph.

\begin{exa} \label{exa:graph_orbit}
    The graphs considered here  are the graph induced
    by the orbit $(\cO, \sim)$ still denoted $\cO$,
    and the two subgraphs of the orbit restricted to each individual type of adjacency,
    which are $\cO^x = (\cO, \sim^x)$ and $\cO^y = (\cO, \sim^y)$.
\end{exa}

We now introduce the chain complex attached to an oriented graph.

\begin{defi}
    Let $\Gamma = (V, E^+)$ be an oriented graph and $K$ a field.
    The space $C_0(\Gamma)$ of $0$-chains of $\Gamma$ is the free $K$-vector space
    spanned by the vertices of $V$.
    Similarly, the space $C_1(\Gamma)$ of $1$-chains of $\Gamma$ is the free $K$-vector space
    spanned by the arcs of $E^+$. We turn $C_*(\Gamma)$ into a chain complex by defining
    the \emph{boundary} homomorphism, which is the $K$-linear map defined by
    \[
        \begin{array}{rccc}
            \partial \colon & C_1(\Gamma) & \longrightarrow & C_0(\Gamma) \\
                            & (a,a') \in E^+ & \longmapsto & a' - a
        \end{array}
    \]
\end{defi}

As the reader notices, the chain complex has only been defined for an
oriented graph.  Nonetheless, if $(V,E_1^+)$ and $(V,E_2^+)$ are two
orientations of a graph $\Gamma$, it is easy to see that the
associated chain complexes are isomorphic \cite[1.21
(3)]{Giblin}. When the context is clear, we shall abuse notation and
define a chain complex $C_{*}(\Gamma)$ of a graph $\Gamma$ as the
chain complex of the oriented graph $(V,E^+)$ where $E^+$ is an
arbitrary orientation of $\Gamma$.

We make the following convention.  Let $a$ and $a'$ be two adjacent
vertices of $\Gamma$. Given an orientation $E^+$ of $\Gamma$, we abuse
notation and denote by $(a,a')$ the $1$-chain
\[
  (a,a') = \left\{
    \begin{array}{cr} (a,a') & \text{if $(a,a')$ is in $E^+$} \\
-(a',a) & \text{otherwise}
    \end{array}\right.
\]
This notation is extremely convenient,
because for two adjacent vertices of $\Gamma$, the boundary
homomorphism always satisfies $\partial((a,a')) = a' - a$ and $(a,a')
= -(a',a)$.

\begin{defi}
 Let $\Gamma = (V,E^+)$ be an oriented graph.   A $1$-chain $c$ which satisfies $\partial(c) = 0$ is called a $\emph{$1$-cycle}$.
\end{defi}

\begin{exa}[$1$-chain induced by a path] \label{exa:pathandonechains}

Let $\Gamma=(V,E)$ be a graph and let $ (a_1, a_2, \dots, a_{n+1})$ be
a path in $\Gamma$, that is, a sequence of vertices such that $a_i$ is
adjacent to $a_{i+1}$ for $i=1, \dots,n$. Given an arbitrary
orientation $E^+$ of $\Gamma$, we define the $1$-chain $p = \sum_{i =
1}^{n} (a_i, a_{i+1})$, and we call it the \emph{$1$-chain induced by
the path $ (a_1, a_2, \dots, a_{n+1}) $}.  By telescoping,
$\partial(p) = a_{n+1} - a_1$, therefore if the path is a loop of
$\Gamma$ then $p$ is a $1$-cycle, hence the name. Every $1$-cycle is a
linear combination of $1$-cycles induced by the simple loops of the
graph, that is, loops with no repeated vertex (see
\cite[Theorem~1.20]{Giblin}).
\end{exa}

We recall that a graph is called \emph{connected} if any two vertices
are joined by a path. The reader should note that the notion of path
does not take into account a potential orientation of the edges.
Every finite graph is the disjoint union of finitely many connected
components which are maximal connected subgraphs. Any orientation of a
graph induces by restriction an orientation on its subgraphs and
thereby on its connected components. With this convention, it turns
out that the chain complex of a finite oriented graph is isomorphic to
the direct sum of the chain complexes of its connected
components. Hence, it is harmless to extend Theorem~1.23 in
\cite{Giblin} to the case of a non-connected graph.

\begin{prop} \label{prop:bound_connected}
    Let $\Gamma = (V, E)$ be a graph, and  let
    $(\Gamma_i = (V_i, E_i))_{i=1,\dots,r}$ be its connected components.
    Define the \emph{augmentation map}
    $\varepsilon \colon C_0(\Gamma) \rightarrow K^r$ by
    $\varepsilon(\sum_{a \in V} \lambda_a a)
    = (\sum_{a \in V_i} \lambda_a)_{i=1,\ldots,r}$.
    Then, $\Ker\, \varepsilon = \imag\, \partial$.
\end{prop}

Let $\Gamma=(V,E)$ be a graph and let $\sigma$ be a graph endomorphism
of $\Gamma$. Fixing an orientation $E^+$ on $\Gamma$, we let $\sigma$
act on the space of $0$ and $1$-chains by $K$-linearity via:
 \[\sigma \cdot a =\sigma(a) \mbox{ for any } a \mbox{ in } V \mbox{
and } \sigma \cdot(a,a')= (\sigma (a),\sigma (a')) \mbox{ for any }
(a,a') \mbox{ in } E^+. \] The reader should note that the action on
the space of $1$-chains uses the convention on the arc notation
introduced at the beginning of this  subsection. It is easily seen that the
action of a graph endomorphism of $\Gamma$ is compatible with the
boundary homomorphism of the chain complex $C_*(\Gamma)$.

\begin{prop} \label{prop:hom_bound}
  Let $\Gamma=(V,E)$ be a graph and $\sigma$ be a graph endomorphism
of $\Gamma$. Then $\sigma$ induces a \emph{chain map} on
$C_*(\Gamma)$, which means that the following diagram of $K$-linear
maps is commutative.

  \begin{center}
   \begin{tikzcd}
C_1(\Gamma) \arrow[r, "\partial"] \arrow[d,"\sigma"]
& C_0(\Gamma) \arrow[d, "\sigma"] \\
C_1(\Gamma) \arrow[r,  "\partial" ]
&  C_0(\Gamma)
\end{tikzcd}
  \end{center}

  \end{prop}

\subsubsection{The chain complex of the orbit and the
algebraic description of bicolored loops}

We now apply the homological formalism to the graphs associated with
the orbit $\cO$ with base field $\C$ (see
Example~\ref{exa:graph_orbit}). We fix once for all an
orientation on $\cO$ which induces an orientation on the subgraphs
$\cO^x$ and $\cO^y$. Quoting \cite[Remark 1.21]{Giblin},`` the choice
of this orientation is just a technical device introduced to enable
the computation of the boundary homomorphisms''.  We denote by
$\partial$ (resp. $\partial^x$, $\partial^y$) the boundary
homomorphism on the connected graph $\cO$ (resp. the non-connected
graphs $\cO^x$, $\cO^y$). Moreover, we denote by $\varepsilon$
(resp. $\varepsilon^x$, $\varepsilon^y$) the augmentation map defined
in Proposition~\ref{prop:bound_connected} for $\cO$ (resp. $\cO^x$,
$\cO^y$).

\begin{lem}\label{lem:decompositioncohomologie}
    The $\C$-vector space $C_1(\cO)$  is equal to $C_1(\cO^x) \oplus
    C_1(\cO^y)$ and  the boundary homomorphism $\partial$
    coincides with $\partial^x + \partial^y$ where one has extended
    $\partial^x$ (resp. $\partial^y$) by zero on $C_1(\cO^y)$ (resp.
    $C_1(\cO^x)$).
\end{lem}
\begin{proof}
    Every edge $\{a,a'\}$ of $\cO$ is either an $x$-adjacency
    or an $y$-adjacency, and not both. Therefore, the set of arcs of an
    orientation of $\cO$ is the disjoint union of the arcs of the orientations of
    $\cO^x$ and $\cO^y$, which thus induces a direct sum decomposition on the free
    vector space $C_1(\cO)$. The decomposition of the homomorphism $\partial$ follows directly.
\end{proof}

The action of the Galois group $G$ on the vertices of $\cO$ preserves the adjacency types
of the edges (see Lemma~\ref{lem:hom_adj}). Therefore $G$ acts by graph automorphisms on $\cO^x$ and
$\cO^y$. Thus, Proposition~\ref{prop:hom_bound} allows us to define the
action of $G$ on the chains of $\cO^x$ and $\cO^y$ in a compatible way
with the decomposition of Lemma~\ref{lem:decompositioncohomologie}.
\begin{prop} \label{prop:group_bound}
  Let $\sigma$ be in $G$. Then $\sigma$ induces automorphisms of the
  chain complexes $C_*(\cO), C_*(\cO^x)$ and $C_*(\cO^y)$ such that
  $\sigma \circ \partial^x=\partial^x \circ \sigma$ and $\sigma \circ
  \partial^y=\partial^y \circ \sigma$.
\end{prop}

The boundary homomorphisms $\partial^x,\partial^y$ allow us to rewrite
the $0$-chains induced by bicolored loops as boundaries.  If $\alpha$
is the $0$-chain associated to a bicolored loop as in Example
\ref{exa:bicol_cycle}, then it is easily seen that $\alpha =
\partial^x (p) = \partial^y(-p)$ with $p$ the $1$-chain as in Example
\ref{exa:pathandonechains}. The homology formalism generalizes the
above description to any $0$-chain that cancels decoupled fractions.

\begin{thm} \label{thm:cycloloc}
    Let $\alpha$ be a $0$-chain. Then the following statements are equivalent:
    \begin{enumerate}
        \item[(1)] $\alpha$ cancels decoupled fractions.
        \item[(2)] $\varepsilon^x(\alpha) = 0$ and $\varepsilon^y(\alpha) = 0$.
        \item[(3)] There exists a $1$-cycle $c$ of $\cO$ such that $\alpha = \partial^x(c)$.
        \item[(3')] There exists a $1$-cycle $c$ of $\cO$ such that $\alpha = \partial^y(c)$.
    \end{enumerate}
  \end{thm}
\begin{proof}
    (1) $\Rightarrow$ (2): Let $\alpha$ be a $0$-chain that cancels
    decoupled fractions. The connected components of the graph $\cO^x$ are
    of the form $\cO^x_u = \left\{(u',v') \in \cO \st u' = u\right\}$
    for the distinct left coordinates $u$ of $\cO$. Therefore, we decompose
    $\alpha = \sum_u \alpha_u$ where $\alpha_u = \sum_{v'} \lambda^u_{v'}
    (u,v')$ is a $0$-chain with vertices in $\cO^x_u$.
    Now, we consider the family of monomials $(X^i)_i$ which are obviously decoupled.
    Since $\alpha$ cancels decoupled fractions, the following holds for all $i$: \[
        0 = (X^i)_{\alpha} = \sum_u (X^i)_{\alpha_u}
        = \sum_u \sum_{v' \sts (u,v') \in \cO^x_u} \lambda^u_{v'} (X^i)_{(u,v')}
        = \sum_u \left(\sum_{v' \sts (u,v') \in \cO^x_u} \lambda^u_{v'}\right) u^i
        = \sum_u \varepsilon^x(\alpha)_u u^i.
      \]
    Because the elements $u$ are distinct, this is a Vandermonde system,
    in which the unknowns are the $\varepsilon^x(\alpha)_u$,
    therefore we deduce that they are all equal to $0$. Thus,
   $\varepsilon^x(\alpha) = 0$.
    The same argument yields $\varepsilon^y(\alpha) = 0$.

    (2) $\Rightarrow$ (3) and (3'): Assume that $\varepsilon^x(\alpha) = 0$ and
    $\varepsilon^y(\alpha) = 0$. By Proposition~\ref{prop:bound_connected},
    there  exist $c_x$ in $C_1(\cO^x)$ and $c_y$ in $C_1(\cO^y)$ such
    that $\partial^x(c_x) = \alpha$ and $\partial^y(c_y) = \alpha$. Moreover,\[
        \partial(c_x - c_y) = \partial(c_x) - \partial(c_y) = \partial^x(c_x) - \partial^y(c_y)
    = \alpha - \alpha = 0.\]
    Therefore, $c = c_x - c_y$ is a $1$-cycle of $\cO$ which satisfies $\partial^x(c) = \alpha$
    and $\partial^y(-c) = \alpha$.

    (3) $\Leftrightarrow$ (3'): Let $c$ be a $1$-cycle of $\cO$,
    then $\partial^x(c) = \partial(c) - \partial^y(c) = \partial^y(-c)$. This proves
    the equivalence.

    (3) $\Rightarrow$ (1): Assume that $\alpha = \partial^x(c) = \partial^y(-c)$ for $c$ a cycle of
    $\cO$.
    Now, let $e = ((u,v),(u,v'))$ be an arc of $\cO^x$ and take
$F(X,t) \in \C(X,t)$. Then
    $F_{\partial^x(e)} = F(u,\invS) - F(u,\invS) = 0$. Therefore, by $\C$-linearity,
    this implies that $F_{\alpha} = F_{\partial^x(c)} = 0$. Symmetrically, if $G(Y,t) \in \C(Y,t)$,
    then we deduce that $G_{\alpha} = G_{\partial^y(-c)} = 0$, which concludes the proof.
\end{proof}

We now apply this pleasant characterization to prove our earlier
claim that $0$-chains that cancel decoupled fractions are induced by $\C$-linear
combinations of $1$-cycles induced by bicolored loops.
\begin{proof}[Proof of Proposition~\ref{prop:alt_cycles}]
Let $\alpha$ be a $0$-chain which cancels decoupled fractions, then by
(3) of Theorem~\ref{thm:cycloloc}, we can write it $\alpha = \partial^x(c) =
- \partial^y(c)$ with $c$ a $1$-cycle of $\cO$. Since the $1$-cycles induced by the simple loops of $\cO$ generate
the $1$-cycles of $\cO$ (see \cite[Theorem~1.20]{Giblin}), we can assume
without loss of generality that $c$ is induced by a simple loop $p
= (a_1, a_2, \dots, a_n)$ of $\cO$.

Moreover, if consecutive arcs $e_i, \dots, e_{i+k-1} = (a_i, a_{i+1}), (a_{i+1}, a_{i+2}),
\dots, (a_{i+k-1}, a_{i+k})$ of $p$ are of the same adjacency type
(say $x$),
then since the monochromatic components of $\cO$ are
cliques, $(a_i,a_{i+k})$ is an arc of $\cO$.
Therefore, $\partial^x(e_i + \dots + e_{i+k-1}) = \partial^x(e_i +
\dots + e_{i+k-1} + (a_{i+k}, a_{i})) + \partial^x((a_i, a_{i+k}))$,
the first term being zero because it is the boundary of a
monochromatic cycle. The exact same reasoning
can be done for consecutive $y$-adjacencies.  Thus,  replacing  consecutive
arcs of the same adjacency type by one single arc of the same adjacency type, we can
assume without loss of generality
that $c$ is the $1$-chain induced by a simple bicolored loop. This proves that
$\alpha$ is the $0$-chain induced by a bicolored loop, finishing
the proof.
\end{proof}

\subsubsection{Construction of the decoupling}
We now use the results of the previous subsections to construct a
pseudo-decoupling of $(x,y)$ on a finite orbit $\cO$. For $p=(p_a)_{a
\in \cO}$ a family of $1$-chains, we consider the $0$-chains
  \[\gamma_x(p) = -\frac{1}{|\cO|} \sum_a \partial^y(p_a)
    \mbox{ and } \gamma_y(p) =   -\frac{1}{|\cO|} \sum_a \partial^x(p_a)\]
where all sums run over $\cO$. The $\C$-linearity of the boundary
homomorphisms implies that $\gamma_x$ and $\gamma_y$ are $\C$-linear
morphisms from $C_1(\cO)^{\cO}$ to $C_0(\cO)$.  We recall that $\sinv$
is the $0$-chain $\frac{1}{|\cO|}\sum_{a \in \cO} a$ defined in
Lemma~\ref{lem:definitionsinv}.

\begin{thm}[Decoupling theorem] \label{thm:path_decoupl}
  Let $p^x=(p^x_a)_{a \in \cO}$ and $p^y=(p^y_a)_{a \in \cO}$ be
  two families of $1$-chains,
  that are such that for all $a \in \cO$ one has
    \[ \varepsilon^x(\partial(p^x_a) + (x,y) - a)=0 \mbox{ and }
      \varepsilon^y(\partial(p^y_a)+(x,y) - a) = 0.\]
    Then, the pair $(\sinv + \gamma_x(p^x), \gamma_y(p^y))$ is a pseudo-decoupling of $(x,y)$.
\end{thm}
\begin{proof}
  Let $(p^x_a)_{a \in \cO}$ and $(p^y_a)_{a \in \cO}$ be two families that
  satisfy the conditions of the theorem. For all $a$ in $\cO$, we
  have $\varepsilon^x(\partial(p^x_a) + (x,y) - a) = 0$. By Proposition
  \ref{prop:bound_connected} applied to $\Gamma = \cO^x$, there exists
  $c^x_a \in C_1(\cO^x)$ such that $\partial(c^x_a) = \partial(p^x_a)- a + (x,y)$,
  which rewrites as $\partial(p^x_a - c^x_a) = a - (x,y)$. We denote by
  $c^x$ the family of $1$-chains $(c^x_a)_a$. Note that $\partial^y(c^x_a) = 0$ for all $a $ in $\cO$ so that   $\gamma_x(c^x) = 0$.
  Similarly, there exists a family  $c^y=(c^y_a )_{a \in \cO}$ in $C_1(\cO^y)^{\cO}$   such that
  we have  $\partial(p^y_a - c^y_a) = a - (x,y)$   for all $a$ in $\cO$ and
  $\gamma_y(c^y) = 0$.
  Therefore,  using the linearity of $\gamma_x$ and $\gamma_y$, we find
   \[(\sinv + \gamma_x(p^x), \gamma_y(p^y)) = (\sinv + \gamma_x(p^x-c^x), \gamma_y(p^x-c^x) + \gamma_y((p^y-c^y)-(p^x-c^x))).\]
   By  construction of $c^x$ and $c^y$,  the $1$-chain $(p^y_a-c^y_a)-(p^x_a-c^x_a)$ is a $1$-cycle for all $a$.
   Hence, by linearity of the boundary homomorphim $\partial$, the $1$-chain
   $t-\frac{1}{|\cO|} \sum_a \left(   (p^y_a-c^y_a)-(p^x_a-c^x_a) \right)$ is a $1$-cycle. Hence,
  by~(3) of Theorem~\ref{thm:cycloloc},
  the $0$-chain $\gamma_y((p^y-c^y)-(p^x-c^x))$ cancels decoupled fractions.

  Therefore, by Lemma~\ref{lem:trans_pseudodecoupl}, it only remains
  to show that the pair $(\sinv + \gamma_x(p^x-c^x), \gamma_y(p^x-c^x))$
  is a pseudo-decoupling of $(x,y)$. Denote by $q=(q_a)_{a \in \cO}$ the
  family of $1$-chains $p^x-c^x$. Then,
  \[  a - (x,y) = \partial(p^x_a-c^x_a) = \partial(q_a) = \partial^y(q_a) + \partial^x(q_a). \]
  Summing this identity over the orbit yields
  \[\sum_{a \in \cO} a - |\cO| (x,y) = \sum_{a \in \cO} \partial^y(q_a) + \sum_{a \in \cO} \partial^x(q_a), \]
  which can be rewritten as
    \[(x,y) = \left(\sinv + \gamma_x (q) \right) + \gamma_y(q). \]

    In order to conclude that the pair $(\sinv + \gamma_x (q), \gamma_y(q))$ is a pseudo-decoupling, we just need to check that,
    for all $\sigma_x$ in $G_x$, the $0$-chain $\sigma_x \cdot (\sinv + \gamma_x(q)) - (\sinv + \gamma_x(q))$
    cancels decoupled fractions, and that, for all $\sigma_y$ in $G_y$, the $0$-chain
    $\sigma_y \cdot \gamma_y(q) - \gamma_y(q)$ cancels decoupled fractions  and apply Lemma~\ref{lem:cocycle_cond}. Let $\sigma_x$ be in $G_x$. Then by compatibility of $G$ with the
    boundaries (Proposition~\ref{prop:group_bound}), we compute
    \begin{align*}
      \sigma_x \cdot \gamma_x(q) &= -\frac{1}{|\cO|} \sum_{a \in \cO} \partial^y(\sigma_x \cdot q_a) \\
                                 &= -\frac{1}{|\cO|} \sum_{a \in \cO} \partial^y(q_{\sigma_x \cdot a})
                                   - \frac{1}{|\cO|} \sum_{a \in \cO} \partial^y(\sigma_x \cdot q_a - q_{\sigma_x \cdot a}).
    \end{align*}

    The homomorphism $\sigma_x$ is a bijection on the vertices of $\calO$, so
    the first sum on the right hand-side is equal to $\gamma_x(q)$ so that
    \begin{equation}\label{eq:pseudodecouplingproof}
      \sigma_x \cdot \gamma_x(q) -\gamma_x(q) = \partial^y\left(  - \frac{1}{|\cO|} \sum_{a \in \cO} (\sigma_x \cdot q_a - q_{\sigma_x \cdot a}) \right).
    \end{equation}

    Now, that since $\sigma_x$ fixes $x$, we have  $\sigma_x \cdot (x,y) = (x,v)$ for some $v$.
    Thus there exists $c$ in $C_1(\cO^x)$ such that $\sigma_x \cdot (x,y) - (x,y) = \partial(c)$.
    Then, for all $a \in \cO$, we have  \[
        \partial(\sigma_x \cdot q_a - q_{\sigma_x \cdot a} + c) =
        (\sigma_x \cdot a - \sigma_x \cdot (x,y))
        - (\sigma_x \cdot a - (x,y))
        + (\sigma_x \cdot (x,y) - (x,y)) = 0.
      \]
      Therefore, the $1$-chain $\sigma_x \cdot q_a - q_{\sigma_x \cdot
        a} + c$ is a $1$-cycle for all $a$ so that $-\frac{1}{|\cO|} \sum_a
      \left( \sigma_x \cdot q_a - q_{\sigma_x \cdot a} + c \right)$ is also
      a $1$-cycle by linearity of the boundary homorphism $\partial$.
      Moreover, since $c$ is in $C_1(\cO^x)$, we have $\partial^y(\sigma_x
      \cdot q_a - q_{\sigma_x \cdot a}) = \partial^y(\sigma_x \cdot q_a -
      q_{\sigma_x \cdot a} + c)$ for all $a$ in $\cO$, so from
      \eqref{eq:pseudodecouplingproof}, we conclude that $\sigma_x \cdot
      \gamma_x(q) -\gamma_x(q) =\partial^y\left(-\frac{1}{|\cO|} \sum_a (
        \sigma_x \cdot q_a - q_{\sigma_x \cdot a} + c)
      \right)$. Theorem~\ref{thm:cycloloc} implies that $\sigma_x \cdot
      \gamma_x(q) -\gamma_x(q)$ cancels decoupled fractions.  Finally, as
      $\sinv$ is fixed by $\sigma_x$, we deduce that $\sigma_x \cdot (\sinv
      + \gamma_x(q)) - (\sinv + \gamma_x(q)) = \sigma_x \cdot \gamma_x(q) -
      \gamma_x(q)$ cancels decoupled fractions.  The proof for $\sigma_y
      \gamma_y(q) -\gamma_y(q)$ is completely analogous.
  \end{proof}

  We can now prove the existence of a decoupling of $(x,y)$ for any
finite orbit.
\begin{proof}[Proof of Theorem~\ref{thm:decoupling_orb}]
  The graph $\cO$ is connected. Hence, for every $a \in \cO$, there
  exists a path from $(x,y)$ to $a$. Denoting by $p_a^x=p_a^y$ the
  associated $1$-chain, we have $\partial(p^x_a) = a - (x,y)$ (see
  Example \ref{exa:pathandonechains}). Therefore, the families
  $(p^x_a)_{a \in \cO}$ and $(p^y_a)_{a \in \cO}$ satisfy the
  assumptions of Theorem~\ref{thm:path_decoupl} leading to the existence
  of a pseudo-decoupling.  Theorem~\ref{thm:pseudec_is_dec} establishes
  the existence of a decoupling obtained from a pseudo-decoupling
  concluding the proof of Theorem~\ref{thm:decoupling_orb}: if the orbit
  is finite, the pair $(x,y)$ always admits a decoupling in the orbit.
\end{proof}

In \cite[Definition~6.1]{BMEFHR}, the authors introduce the
notion of a multiplicative decoupling of a regular fraction. In our
context, we say that a regular fraction $H(X,Y)$ has a multiplicative
Galois decoupling if and only if there exists a positive integer $m$
such that \[H(X,Y)^m =F(X,t)G(Y,t) +\Ktld(X,Y,t)P(X,Y,t),\] for some
rational fractions $F(X,t),G(Y,t)$ and a regular fraction
$P(X,Y,t)$. 

Theorem~\ref{thm:decoupling_orb} yields a decoupling of $(x,y)$
with $0$-chains $\wgammax, \wgammay$ and $\alpha$ having rational
coefficients. Let $d$ be the common denominator of the rational
coefficients of $\wgammax, \wgammay$ and $\alpha$ which is easily seen
to divide the size of the orbit in the proof of Theorem~\ref{thm:path_decoupl}
when the input  $1$-chains in   $p^x$ and $p^y$ all have
integer coefficients. Then, the $0$-chains
$d\wgammax, d\wgammay, d\alpha$ have integer coefficients.
For such chains, one can define a multiplicative evaluation:

 For a $0$-chain $c = \sum_{u,v} c_{u,v} (u,v)$ with integer
coefficients, define \[H^{\text{mul}}_{c}= \prod_{u,v}
H(u,v,\invS)^{c_{u,v}}.\]

As a direct corollary of the existence of a decoupling in the orbit,
the following lemma gives an explicit procedure to test and construct,
when it exists, the multiplicative Galois decoupling of a regular
fraction $H$.

\begin{lem}
The following statements are equivalent:
\begin{itemize}
\item $H(X,Y,t)$ has a multiplicative Galois decoupling.
\item There exists a a positive integer $m$ such that ${\left( H^{\text{mul}}_{d \alpha} \right)}^m=1$.
\end{itemize}
\end{lem}
\begin{proof}
From Proposition~\ref{lem:liftingidentities}, the regular fraction
$H(X,Y,t)$ admits a multiplicative Galois decoupling if and only if
there exist a positive integer $m$, $f(x) \in k(x)$ and $g(y) \in
k(y)$ such that $H_{(x,y)}^m=f(x)g(y)$.

 Let us assume that $H$ admits a multiplicative Galois decoupling and
let $m$ be a positive integer such that $H(X,Y)^m =F(X)G(Y)
+K(X,Y,t)P(X,Y,t)$ for some rationals fractions $F,G$ and a regular
fraction $P$. By multiplicative evaluation of the previous identity on
$d\alpha$, we find that ${\left(H^{\text{mul}}_{d\alpha} \right)}^m=
{\left(F(X)^{\text{mul}}_{d \alpha } \right)}^m
{\left(G(Y)^{\text{mul}}_{d \alpha} \right)}^m.$ It is clear that
$d\alpha$ is a $0$-chain with integer coefficients that cancels
decoupled fractions. By Proposition~\ref{prop:alt_cycles}, the chain
$d \alpha$ is a $\Z$-linear combination of $0$-chains induced by
bicolored loops. One proves easily by a multiplicative analogue of
Example~\ref{exa:bicol_cycle} that if $\beta$ is a $0$-chain
induced by a bicolored loop then $F(X)^{\text{mul}}_{\beta}=
G(Y)^{\text{mul}}_{\beta}=1$ which concludes the proof of the first
implication.

Conversely, if there exists a positive integer $m$ such that
${\left(H^{\text{mul}}_{d \alpha}\right)}^m=1$, the decoupling
$d \cdot (x,y)=d \wgammax + d \wgammay + d \alpha$ yields by multiplicative
evaluation \[{\left(H^{\text{mul}}_{d(x,y)}\right)}^m={H_{(x,y)}}^{dm}
={\left(H^{\text{mul}}_{d\wgammax}\right)}^m {\left(H^{\text{mul}}_{d
\wgammay}\right)}^m.\] By definition of the decoupling of
$(x,y)=\wgammax +\wgammay +\alpha$, we find that $\sigma\cdot d
\wgammax= d \sigma\cdot \wgammax=d \wgammax$ for all $\sigma \in G_x$.
A multiplicative analogue of Lemma~\ref{lem::commutationgrp_eval_orb} implies easily that
$H^{\text{mul}}_{d \wgammax}$ is left fixed by $G_x$ so that
$H^{\text{mul}}_{d\wgammax}$ belongs to $k(x)$. A similar argument
shows that $H^{\text{mul}}_{d \wgammay}$ belong to $k(y)$ which
concludes the proof.
\end{proof}

\subsection{Effective construction} \label{subsect:effective_decoupling}

The evaluation of a regular fraction at a vertex of the orbit, that
is, at a pair of algebraic elements in $\K$ might be difficult from an
algorithmic point of view since this requires to compute in an
algebraic extension of $\Q(x,y)$. This is however the cost we may have
to pay in our decoupling procedure if we choose random families of
$1$-chains satisfying the assumptions of
Theorem~\ref{thm:path_decoupl}.

In this section, we show how, under mild assumption on the
\emph{distance transitivity} of the graph of the orbit, one can
construct a decoupling in the orbit expressed in terms of specific
$0$-chains that we call \emph{level lines}. These level lines regroup
vertices of the orbit that satisfy the same polynomial
relations. Therefore, one can use symmetric functions and efficient
methods from computer algebra to evaluate regular fractions on these
level lines (see Appendix~\ref{sect:decoupl_comp}).

\begin{defi}
  Let $a$ be a vertex
  of $\cO$. We define the \emph{$x$-distance}
  of $a$ to be $d_x(a) = \inf\{d(a,a') \st a' \sim^x (x,y)\}$,
  that is, the length of a shortest path in $\cO$ from $a$ to the clique $(x,\cdot)$.
  
  Such a shortest path $(g_0, g_1, \dots, g_r)$,
  that is, $g_r = a$, $g_0 \sim^x (x,y)$ and $d_x(a) = r$,
 is  called an \emph{$x$-geodesic} for $a$. Note that  we have $d_x(g_i) = i$ for all $i=0,\dots,r$.  We  denote by
  $\calP^x_a$ the set of $1$-chains associated with $x$-geodesics for
  $a$ as in Example \ref{exa:pathandonechains}.

  The \emph{$x$-level lines} $\levX_0, \levX_1, \dots$ are defined by
  $\levX_i = \{a \in \cO \st d_x(a) = i\}$, and we associate to the level line $\levX_i$   the $0$-chain
  $X_i = \sum_{a \in \levX i} a$. Analogously, we define the $y$-distance $d_y$,
  the set $\calP^y_a$ of $y$-geodesics for $a$, the $y$-level lines
  $\levY_0, \levY_1, \dots$, and denote by $Y_i$ the $0$-chain
  associated with the $y$-level line $\levY_i$.
\end{defi}

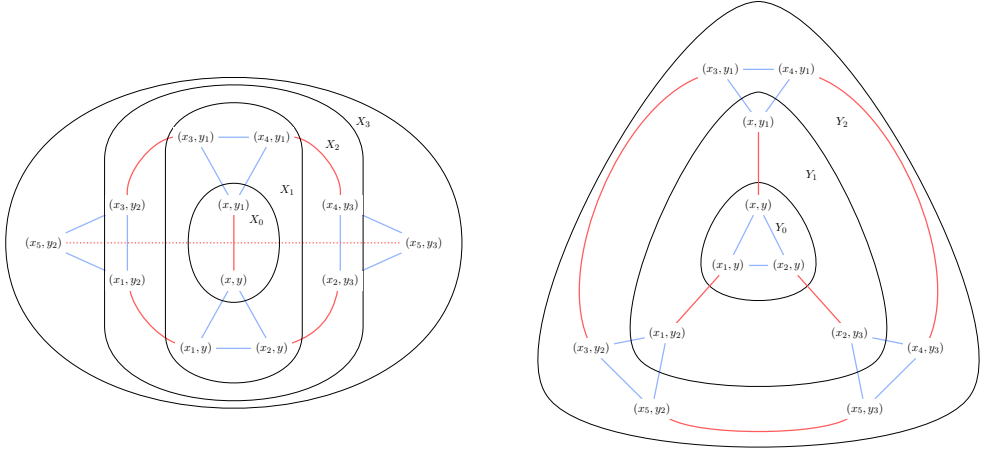
\begin{figure}[h]
  \centering
  \scalebox{0.4}{\begin{tikzpicture}[scale=2]
	\begin{pgfonlayer}{nodelayer}
		\node [style=paire] (0) at (9.5, 4) {$(x,y_1)$};
		\node [style=paire] (1) at (8.25, 5.75) {$(x_3,y_1)$};
		\node [style=paire] (2) at (10.75, 5.75) {$(x_4,y_1)$};
		\node [style=paire] (3) at (9.5, 1.25) {$(x,y)$};
		\node [style=paire] (4) at (8.5, -0.75) {$(x_1,y)$};
		\node [style=paire] (5) at (10.5, -0.75) {$(x_2,y)$};
		\node [style=paire] (6) at (6.5, -3) {$(x_1,y_2)$};
		\node [style=paire] (7) at (4, -3.5) {$(x_3,y_2)$};
		\node [style=paire] (8) at (6, -5.5) {$(x_5,y_2)$};
		\node [style=paire] (9) at (12.5, -3) {$(x_2,y_3)$};
		\node [style=paire] (10) at (13, -5.5) {$(x_5,y_3)$};
		\node [style=paire] (11) at (15, -3.5) {$(x_4,y_3)$};
		\node [style=paire] (12) at (-7.75, 1.25) {$(x,y_1)$};
		\node [style=paire] (13) at (-9, 3.5) {$(x_3,y_1)$};
		\node [style=paire] (14) at (-6.5, 3.5) {$(x_4,y_1)$};
		\node [style=paire] (15) at (-7.75, -1.25) {$(x,y)$};
		\node [style=paire] (16) at (-9, -3.5) {$(x_1,y)$};
		\node [style=paire] (17) at (-6.5, -3.5) {$(x_2,y)$};
		\node [style=paire] (18) at (-11.25, -1.25) {$(x_1,y_2)$};
		\node [style=paire] (19) at (-11.25, 1.25) {$(x_3,y_2)$};
		\node [style=paire] (20) at (-14, 0) {$(x_5,y_2)$};
		\node [style=paire] (21) at (-4.25, -1.25) {$(x_2,y_3)$};
		\node [style=paire] (22) at (-1.5, 0) {$(x_5,y_3)$};
		\node [style=paire] (23) at (-4.25, 1.25) {$(x_4,y_3)$};
		\node [style=none] (24) at (9.5, 2) {};
		\node [style=none] (25) at (7.75, -1.25) {};
		\node [style=none] (26) at (11.25, -1.25) {};
		\node [style=none] (27) at (9.5, 5) {};
		\node [style=none] (28) at (5.5, -3.75) {};
		\node [style=none] (29) at (13.5, -3.75) {};
		\node [style=none] (39) at (-0.25, 0) {};
		\node [style=none] (40) at (-15.25, 0) {};
		\node [style=none] (41) at (10.25, 0.5) {$Y_0$};
		\node [style=none] (42) at (11.25, 2.25) {$Y_1$};
		\node [style=none] (43) at (12.25, 4) {$Y_2$};
		\node [style=none] (44) at (-7, 0.75) {$X_0$};
		\node [style=none] (45) at (-6, 1.75) {$X_1$};
		\node [style=none] (46) at (-4.5, 3.25) {$X_2$};
		\node [style=none] (47) at (-3.5, 4) {$X_3$};
		\node [style=none] (48) at (9.5, 8) {};
		\node [style=none] (49) at (2.5, -5) {};
		\node [style=none] (50) at (16.5, -5) {};
		\node [style=none] (51) at (-6.25, 0) {};
		\node [style=none] (52) at (-9.25, 0) {};
		\node [style=none] (53) at (-10, 3) {};
		\node [style=none] (54) at (-5.5, 3) {};
		\node [style=none] (55) at (-5.5, -3) {};
		\node [style=none] (56) at (-10, -3) {};
		\node [style=none] (57) at (-12, 2.75) {};
		\node [style=none] (58) at (-3.5, 2.75) {};
		\node [style=none] (59) at (-3.5, -2.75) {};
		\node [style=none] (60) at (-12, -2.75) {};
	\end{pgfonlayer}
	\begin{pgfonlayer}{edgelayer}
		\draw [style={none}] (26.center)
			 to [bend left=60, looseness=0.75] (25.center)
			 to [bend left=60, looseness=0.75] (24.center)
			 to [bend left=60, looseness=0.75] cycle;
		\draw [style=usimx] (4) to (3);
		\draw [style=usimx] (3) to (5);
		\draw [style=usimx] (5) to (4);
		\draw [style=usimx] (1) to (0);
		\draw [style=usimx] (0) to (2);
		\draw [style=usimx] (2) to (1);
		\draw [style=usimx] (6) to (7);
		\draw [style=usimx] (7) to (8);
		\draw [style=usimx] (8) to (6);
		\draw [style=usimx] (9) to (10);
		\draw [style=usimx] (10) to (11);
		\draw [style=usimx] (11) to (9);
		\draw [style=usimy] (4) to (6);
		\draw [style=usimy] (5) to (9);
		\draw [style=usimy, bend right, looseness=0.50] (8) to (10);
		\draw [style=usimy] (3) to (0);
		\draw [style=usimy, bend right=45, looseness=0.75] (1) to (7);
		\draw [style=usimy, bend left=45, looseness=0.75] (2) to (11);
		\draw [style=usimx] (16) to (15);
		\draw [style=usimx] (15) to (17);
		\draw [style=usimx] (17) to (16);
		\draw [style=usimx] (13) to (12);
		\draw [style=usimx] (12) to (14);
		\draw [style=usimx] (14) to (13);
		\draw [style=usimx] (18) to (19);
		\draw [style=usimx] (19) to (20);
		\draw [style=usimx] (20) to (18);
		\draw [style=usimx] (21) to (22);
		\draw [style=usimx] (22) to (23);
		\draw [style=usimx] (23) to (21);
		\draw [style=usimy, bend left, looseness=0.75] (16) to (18);
		\draw [style=usimy, bend right] (17) to (21);
		\draw [style=usimy, dotted] (20) to (22);
		\draw [style=usimy] (15) to (12);
		\draw [style=usimy, bend right=45, looseness=0.75] (13) to (19);
		\draw [style=usimy, bend left=45, looseness=0.75] (14) to (23);
		\draw [style={none}] (29.center)
			 to [bend right=60, looseness=0.50] (27.center)
			 to [bend right=60, looseness=0.50] (28.center)
			 to [bend right=60, looseness=0.50] cycle;
		\draw [style={none}] (39.center)
			 to [bend left=90, looseness=1.25] (40.center)
			 to [bend left=90, looseness=1.25] cycle;
		\draw [style={none}] (50.center)
			 to [bend right=60, looseness=0.50] (48.center)
			 to [bend right=60, looseness=0.50] (49.center)
			 to [bend right=60, looseness=0.50] cycle;
		\draw [style={none}] (51.center)
			 to [bend left=90, looseness=2.25] (52.center)
			 to [bend left=90, looseness=2.25] cycle;
		\draw [style={none}] (53.center)
			 to [bend left=90, looseness=1.25] (54.center)
			 to (55.center)
			 to [bend left=90, looseness=1.25] (56.center)
			 to cycle;
		\draw [style={none}] (57.center)
			 to [bend left=90] (58.center)
			 to (59.center)
			 to [bend left=90] (60.center)
			 to cycle;
	\end{pgfonlayer}
\end{tikzpicture}}
\caption{The level lines for the orbit $\cO_{12}$} \label{fig:llines_o12}
\end{figure}

The level lines can be represented graphically, as in Figure~\ref{fig:llines_o12},
or in Section\ref{subsect:examples} or in the examples of Appendix~\ref{appendix:other_orbits}.
The level lines and geodesics are our key tools
to construct relevant collections of $1$-chains satisfying the
conditions of Theorem~\ref{thm:path_decoupl}. First, the boundaries of a geodesic are easy to express.
\begin{lem} \label{lem:alt_geodesic}
  Let $a$ be a vertex of $\cO$ and $(g_0,g_1,\dots,g_r)$ an
  $x$-geodesic for $a$. Then $g_{i} \sim^y g_{i-1}$ if and only if $i$
  is odd.
  Similarly, for $(g_0,g_1,\dots,g_r)$ an $y$-geodesic for $a$, then
  $g_{i} \sim^x g_{i-1}$ if and only if $i$ is odd.
\end{lem}
\begin{proof}
  Let $g = (g_0, g_1, \dots, g_r)$ be an $x$-geodesic of length $r$.
  Assume that there exists $i$ such that $g_i \sim^x g_{i+1} \sim^x g_{i+2}$.
  By transitivity of $\sim^x$,
  this implies that $g_i \sim^x g_{i+2}$, contradicting the minimality of the
  geodesic $g$. Similarly, if there exists $i$ such that $g_i \sim^y g_{i+1} \sim^y g_{i+2}$
  then $g_i \sim^y g_{i+2}$, also contradicting the minimality of the geodesic.
  Therefore, the adjacency types of the edges of the geodesic alternate. Finally, if $g_0 \sim^x g_1$, then this also contradicts the minimality of
  the geodesic because then $(x,y) \sim^x g_1$. This fixes the starting
  parity of the alternating adjacency types of edges of the geodesic,
  and thus $g_i \sim^y g_{i-1}$ if and only if $i$ is odd. The case of an $y$-geodesic is symmetric.
\end{proof}

\begin{cor} \label{cor:bound_geodesic}
 Let $a$ be a vertex  of $\cO$, $(g_0,g_1,\dots,g_r)$ an $x$-geodesic for $a$ and   $g$ its associated $1$-chain, then
  $\displaystyle \partial^y(g) = \sum_{\substack{1 \le i \le r \\ i
      \text{ odd }}} g_i - g_{i-1}$.
  Analogously, for $(g_0,g_1,\dots,g_r)$ a $y$-geodesic for $a$ then
  $\displaystyle \partial^x(g) = \sum_{\substack{1 \le i \le r \\ i
      \text{ odd }}} g_i - g_{i-1}$.
\end{cor}
Recall from Section \ref{subsec:basicgraphhomology} that any graph
automorphism $\tau$ of $\cO$ acts on the vertex $a$ of $\cO$
coordinate-wise and that we denote this action $\tau \cdot a$. We
extend the action of $\tau$ to any path $(a_1,\dots,a_{n+1})$ as
follows
\[ \tau \cdot (a_1,\dots,a_{n+1}) = (\tau\cdot a_1, \dots, \tau \cdot
a_{n+1}).\] Note that this action is compatible with the action of
graph automorphisms on $1$-chains defined in Section
\ref{subsec:basicgraphhomology}. Indeed, if $p$ is the $1$-chain
associated with the path $(a_1,\dots,a_{n+1})$ as in Example
\ref{exa:pathandonechains} then $\tau\cdot p$ is the $1$-chain
associated with the path $\tau \cdot (a_1,\dots,a_{n+1})$.

The following lemma shows that the geodesics and level lines satisfy
some stability properties with respect to the action of elements of
$G_x$ and $G_y$ viewed as subgroups of the group of graph
automorphisms of $\cO$.

\begin{lem} \label{lem:gxy_respect_lines}
  Let $\sigma_x$ be in $G_x$ and $a$ in $\cO$.
  Then $d_x(\sigma_x \cdot a) = d_x(a)$. Moreover, if $(g_0,\dots,g_r)$ is an $x$-geodesic for $a$, then
   $\sigma_x \cdot (g_0,\dots,g_r) $  an $x$-geodesic for $\sigma_x \cdot a$.
  Analogously, if $\sigma_y$ is in $G_y$ and $a$ in $\cO$, then $d_y(\sigma_y \cdot a) = d_y(a)$,
  and if $(g_0,\dots,g_r)$ an $y$-geodesic for $a$, so is $\sigma_y \cdot (g_0,\dots,g_r)$ for $\sigma_y \cdot a$.
\end{lem}
\begin{proof}
  Assume that $d_x(a) = r$. Then there exists an $x$-geodesic for
  $a$ that is  $(g_0, g_1, \dots, g_r)$ with $g_r=a$. Apply the graph automorphism
  $\sigma_x$ to each of the vertices of this path. Then
  $(\sigma_x \cdot g_0, \sigma_x \cdot g_1, \dots, \sigma_x \cdot g_r)$ with $   \sigma_x \cdot g_r= \sigma_x \cdot a$
  is a path of the orbit. By definition, $g_0 \sim^x (x,y)$, thus
  $\sigma_x \cdot g_0 \sim^x (x,y)$ since $x$ is fixed by $G_x$.
  Therefore,
  $d_x(\sigma_x \cdot a) \le r = d_x(a)$. Since $\sigma_x$ is an automorphism,
  we conclude that $d_x(\sigma_x \cdot a) = d_x(a)$. We finally deduce that $\sigma_x \cdot (g_0,\dots,g_r)$
  is a $x$-geodesic for $\sigma_x \cdot a$.
\end{proof}

This observation leads us to define two subgroups of
automorphisms of the graph $\cO$. We denote by $\Aut_x(\cO)$
(resp. $\Aut_y(\cO)$) the subgroup of graph automorphisms of $\cO$
that preserve the $x$ (resp. $y$)-distance and the  adjacency types
\footnote{One can show that this last condition is redundant with the condition on the distance preservation.}.
By definition, any element $\tau$ in $\Aut_x(\cO)$ maps an
$x$-geodesic for $a$ onto an $x$-geodesic for $\tau \cdot
a$. Moreover, a graph automorphism preserve the $x$-distance if and
only if it induces a bijective map from $\levX_i$ to itself for each
$i$.  Analogous results hold for $\Aut_y(\cO)$.

Lemma~\ref{lem:gxy_respect_lines} implies that $G_x$ (resp.  $G_y$) is
isomorphic to a subgroup of $\Aut_x(\cO)$ (resp. $\Aut_y(\cO)$). The
benefit of the groups $\Aut_x(\cO)$ and $\Aut_y(\cO)$ is that, unlike
$G_x$ and $G_y$, they only depend on the graph structure of the orbit,
and thus are more easily computable.  Note however that not all such
graph automorphisms come from a Galois automorphism (see for instance
the Hadamard example in Section~\ref{subsubsect:Hadamard}). We now
state an assumption on the \emph{distance transitivity} of the graph
of the orbit.

\begin{assumption} \label{conj:xyorb_levlines}
  Let $a$ and $a'$ be two pairs of $\cO$. If $d_x(a) = d_x(a')$, then
  there exists  $\sigma_x$ in $\Aut_x(\cO)$ such that $\sigma_x \cdot a = a'$.
  Similarly, if $d_y(a) = d_y(a')$, then there exists
   $\sigma_y$ in $\Aut_y(\cO)$ such that $\sigma_y(a) = a'$.
  In other words, $\Aut_x(\cO)$ (resp. $\Aut_y(\cO)$) acts
  transitively on $\levX_i$ (resp. $\levY_i$) for all $i$.
\end{assumption}

This assumption has been checked for all the finite orbit types
appearing for models with steps in $\{-1,0,1,2\}^2$ as well as for
Hadamard and Fan-models (see the examples in
Section~\ref{subsect:examples} or
Appendix~\ref{appendix:other_orbits}). However, Assumption
\ref{conj:xyorb_levlines} does not always hold as illustrated in the
following example.

\begin{exa}
  Consider the weighted model described by the Laurent polynomial
$S(X,Y) = \left(X+\frac{1}{X} + Y + \frac{1}{Y}\right)^2$.  The kernel
polynomial $\Ktld$ is an irreducible polynomial of degree $4$ in $X$
and in $Y$.  Therefore, the cardinal of $\levY_0$ is $4$ and the only
right coordinate of the elements in $\levY_0$ is $y$. Moreover, each
element of $\levY_0$ is $x$-adjacent to three distinct elements in
$\levY_1$ so the cardinality of $\levY_1$ is $12$.  Now, it is easily
seen that the right coordinates of vertices in $\levY_0~\cup~\levY_1$
are the roots of the polynomial
$\Res(\Ktld(X,y,\invS),\Ktld(X,Y,\invS),X)$. Since $x$ and $y$ are
algebraically independent over $\C$, its irreducible factors in
$\C(x,y)[Y]$ are $\left(Y y -1\right)$, $\left(-y +Y \right)$ ,
$\left(Y^{2} x y +2 Y \,x^{2} y +Y x \,y^{2}+Y x +2 Y y +x y \right)$
and $\left(Y^{2} x y -2 Y \,x^{2} y -3 Y x \,y^{2}-3 Y x -2 Y y +x y
\right)$.  This proves that the cardinality of the set $\mathcal{V}$
of right-coordinates of elements in $\levY_1$ is $5$.
  
   If Assumption \ref{conj:xyorb_levlines} were true for this model
then the transitive action of $\Aut_y(\cO)$ on $\levY_1$ implies that
the sets $K_v = \left\{(u,w) \st \, w=v \mbox{ and } (u,w) \in \levY_1
\right\} \subset \levY_1$ for $v$ in $\mathcal{V}$ are all in bijection.
Indeed, $K_v$ is equal to $\{a \in \cO \st a \sim^y (u,v)\} \cap \levY_1$ for
some $(u,v) \in K_v$.
Therefore, as Assumption \ref{conj:xyorb_levlines} provides $\sigma_y$ in
$\Aut_y(\cO)$ such that $\sigma_y \cdot (u,v) = (u',v') \in K_{v'}$,
its restriction to $K_v$ gives an embedding into $K_{v'}$, because
$\sigma_y$ preserves the $y$-adjacencies and the $y$-distance. 
By symmetry, this proves that $K_v$ and $K_{v'}$ are in bijection.
Since these sets form
a partition of $\levY_1$, this would imply that the cardinality of
$\mathcal{V}$ (5) divides the cardinality of $\levY_1$ (12). A contradiction.
\end{exa}
  
We now show that Assumption~\ref{conj:xyorb_levlines} is sufficient
for $(x,y)$ to admit a decoupling in terms of level lines.

\begin{lem}[Under Assumption \ref{conj:xyorb_levlines}]  \label{lem:bij_desc_paths}
  Let $a$ and $a'$ to be two vertices with $d_x(a) = d_x(a')$. Then
  there is a bijection between $\calP^x_a$ and
  $\calP^x_{a'}$. Analogously, if $a$ and $a'$ satisfy $d_y(a) =
  d_y(a')$, then there is a bijection between $\calP^y_a$ and
  $\calP^y_{a'}$.
\end{lem}
\begin{proof}
  Use Assumption~\ref{conj:xyorb_levlines} to produce $\sigma_x$
  in $\Aut_x(\cO)$ such that $\sigma_x (a) = a'$. This $\sigma_x$
  induces a bijection between $\calP^x_a$ and $\calP^x_{\sigma_x \cdot a} = \calP^x_{a'}$
  by Lemma~\ref{lem:gxy_respect_lines} and the compatibility between the action of $\sigma_x$ on $x$-geodesics
  and its action on the associated $1$-chains.
\end{proof}

The following theorem gives a decoupling of $(x,y)$ in terms of level lines.

\begin{thm}[Under Assumption \ref{conj:xyorb_levlines}] \label{thm:decoupl_levlines}
  Define the following $0$-chains:
  \[
    \begin{array}{lcr}
      {\displaystyle \gamma_x = -\frac{1}{|\cO|} \sum_{i \ge 1} |\levX_i| \sum_{\substack{1 \le j \le i \\ j \text{ odd}}} \left(\frac{X_j}{|\levX_j|} - \frac{X_{j-1}}{|\levX_{j-1}|} \right)}
      &
        \text{and}
      &
        {\displaystyle \gamma_y = -\frac{1}{|\cO|} \sum_{i \ge 1} |\levY_i| \sum_{\substack{1 \le j \le i \\ j \text{ odd}}} \left(\frac{Y_j}{|\levY_j|}  - \frac{Y_{j-1}}{|\levY_{j-1}|}\right)}
    \end{array}.
  \]
  Then $(x,y) = \left(\sinv + \gamma_x\right)+ \gamma_y + \alpha$ is a
  decoupling of $(x,y)$ in the orbit (with $\sinv = \frac{1}{|\cO|} \sum_{a \in \cO} a$).
\end{thm}
\begin{proof}
  Consider the two families of $1$-chains $(p^x_a)_{a \in \cO}$ and $(p^y_a)_{a \in \cO}$ defined
  for $a$ in $\cO$ as
  \begin{align*}
    p^x_a = \frac{1}{|\calP^x_a|} \sum_{g \in \calP^x_a} g & \quad \mbox{ and } \quad
                                                             p^y_a = \frac{1}{|\calP^y_a|} \sum_{g \in \calP^y_a} g.
  \end{align*}

  For all $g=(g_0,\dots, g_r)$ in
  $\calP^x_a$, we have  $\partial(g) = a - g_0$ with $g_0 \sim^x (x,y)$.  Then,
  $\varepsilon^x(\partial(g) - a + (x,y)) = 0$.  Thus, we find by
  linearity that $\varepsilon^x(\partial(p^x_a) - a + (x,y)) = 0$. The same
  argument shows that $\varepsilon^y(\partial(p^y_a) - a + (x,y)) =
  0$. Therefore, both families of $1$-chains $(p^x_a)_{a \in \cO}$ and
  $(p^y_a)_{a \in \cO}$ satisfy the conditions of Theorem~\ref{thm:path_decoupl},
  which thus states that if we take
  \begin{align*}
    \gamma_x = - \frac{1}{|\cO|} \sum_{a \in \cO} \partial^y(p^x_a) & \quad \text{ and } \quad
                                                                      \gamma_y = -\frac{1}{|\cO|} \sum_{a \in \cO} \partial^x(p^y_a),
  \end{align*}
  then the pair $(\sinv + \gamma_x, \gamma_y)$
  is a pseudo decoupling. As the geodesics are stable under the action of
  their respective Galois groups by Lemma~\ref{lem:gxy_respect_lines}, it is also a decoupling.

  Therefore, we are left to prove that 
  $\gamma_x$ and $\gamma_y$ admit the announced (pleasant) expressions.
  We only treat the case of $\gamma_x$, the case of $y$ being  totally symmetric.

  First, note that, by Lemma~\ref{lem:bij_desc_paths}, the cardinality
  of $\calP^x_a$ (resp. $\calP^y_a$) depend only on the $x$-distance
  (resp. $y$-distance) of $a$. For $i$ a non-negative integer, we denote
  by $m^x_i$ (resp. $m^y_i$) the cardinality of $\calP^x_a$
  (resp. $\calP^y_a$) for any $a$ such that $d_x(a)=i$
  (resp. $d_y(a)=i$).  The expression of the boundary of a geodesic
  (Lemma \ref{cor:bound_geodesic}) combined with the partition of $\cO$
  into $x$-level lines yields
  \[
    \gamma_x = -\frac{1}{|\cO|} \sum_{i \ge 0} \sum_{a \in \levX_i} \partial^y(p^x_a)
    = -\frac{1}{|\cO|} \sum_{i \ge 0} \frac{1}{m^x_i} \sum_{a \in \levX_i} \sum_{g \in \calP^x_a} \sum_{\substack{j \text{ odd}\\ j \le i}} \left(g_j - g_{j-1}\right).
  \]
  If we denote
  \[
    S^i_j = \frac{1}{m^x_i} \sum_{a \in \levX_i} \sum_{g \in \calP^x_a} g_j,
  \]
  then $\gamma_x$ rewrites as
  \[
    \gamma_x = -\frac{1}{|\cO|} \sum_{i \ge 1} \sum_{\substack{j \text{ odd} \\ j \le i}}
    S^i_j - S^i_{j-1}.
  \]

  First, observe that, for any $x$-geodesic $(g_0,\dots,g_i)$, the
  $j$-th component $g_j$ has $x$-distance $j$, so the vertices appearing
  in $S^i_j$ with nonzero coefficients are in $\levX_j$. Thus, we can
  write \[S^i_j = \sum_{b \in \levX_j} \lambda^{i,j}_b b. \]
  
  Let $\sigma_x$ be in $\Aut_x(\cO)$. Remind that $\sigma_x$ induces a
  bijection on each $x$-level line and maps bijectively $\calP^x_a$ and
  $\calP^x_{\sigma_x \cdot a}$ for all $a$. Thus, we find

  \begin{align*}
    \sigma_x \cdot S^i_j &=& \frac{1}{m^x_i} \sum_{a \in \levX_i}
                             \sum_{g \in \calP^x_a} \sigma_x \cdot g_j
    &=& \frac{1}{m^x_i} \sum_{a \in \levX_i}
        \sum_{g \in \calP^x_a} (\sigma_x \cdot g)_j
    &=& \frac{1}{m^x_i} \sum_{a \in \levX_i}
        \sum_{g \in \calP^x_{\sigma_x \cdot a}} g_j
    &=& S^i_j.
  \end{align*}

  Under Assumption \ref{conj:xyorb_levlines}, the group $\Aut_x(\cO)$
  acts transitively on $\levX_j$.  Since $ S^i_j$ is fixed by the action
  of $\Aut_x(\cO)$, one concludes easily that all the coefficients
  $\lambda^{i,j}_b$ are equal to some scalar $\lambda^i_j$ and that
  $S^i_j = \lambda^i_j X_j \quad (*)$. To compute the value of
  $\lambda^i_j$, we recall the existence of the augmentation morphism
  ${\varepsilon :C_0 (\cO) \rightarrow \C}$ which associates to a
  $0$-chain the sum of its coefficients.  We apply $\varepsilon$ to each
  side of $(*)$.  On the one hand, $\varepsilon(S^i_j) = \sum_{a \in
    \levX_i} \frac{1}{|\calP^x_a|} \sum_{g \in \calP^x_a} 1 = \sum_{a \in
    \levX_i} 1 = |\levX_i|$. On the other hand, $\varepsilon(\lambda^i_j
  X_j) = \lambda^i_j |\levX_j|$. Therefore, we deduce $\lambda^i_j =
  \frac{|\levX_i|}{|\levX_j|}$ and the announced expression for the
  decoupling follows.
\end{proof}

To conclude, we have defined in this section a
\emph{distance-transitivity} property that is only
graph-theoretic. When this property is satisfied by the orbit-type, it
leads to a decoupling expressed in terms of level lines. As described
in Appendix~\ref{sect:decoupl_comp}, the evaluation of a regular
fraction on a level line is efficient from an algorithmic point of
view and so is our procedure for the Galois decoupling of a regular
fraction. In the following section and in Appendix
\ref{appendix:other_orbits}, we easily check
Assumption~\ref{sect:decoupl_comp} on various orbit-types and produce
the associated decoupling in terms of level-lines.

\subsection{Examples}\label{subsect:examples}

In this last subsection and Appendix~\ref{appendix:other_orbits}, we
check Assumption~\ref{conj:xyorb_levlines} and unroll the construction
of the decoupling of the previous section for all the finite
orbit-types of models with steps in $\{-1,0,1,2\}^2$, namely with
orbits $\cO_{12}$, $\cO_{18}$, $\widetilde{\cO_{12}}$ as well as for
the cyclic models, the Hadamard models and the fan models. We
summarize the results of this section and Appendix
\ref{appendix:other_orbits} on the decoupling of $XY$ in the following
proposition.

\begin{prop}\label{prop:XYdecouple}
  The regular fraction $XY$ does not decouple for any weighted models
  with orbit-types Hadamard (see below) nor for the family of the
  fan-models (see Appendix~\ref{appendix:other_orbits}).  The regular
  fraction $XY$ does not decouple for unweighted models with steps in
  $\{-1,0,1,2 \}^2$ with orbit-types $\cO_{18}$, $\widetilde{\cO_{12}}$.
  The fraction $XY$ decouples for the model $\Gmod$ with any $\lambda$.
\end{prop}

\subsubsection{Cyclic orbit}

Assume that the orbit is a cycle of size $2n$, which is the orbit-type
of any small-steps model with finite orbit.  The graph of the orbit
looks as follows, where we have labeled vertices from $0$ to $2n-1$.
We represent both $x$-level lines and $y$-level lines.
\begin{center}
  \scalebox{1}{
    \begin{tikzpicture}[scale=1.5]
	\begin{pgfonlayer}{nodelayer}
		\node [style=paire] (0) at (-9.75, -1) {$0$};
		\node [style=paire] (1) at (-9.75, 1) {$1$};
		\node [style=paire] (2) at (-7.75, -1) {$2$};
		\node [style=paire] (3) at (-7.75, 1) {$3$};
		\node [style=paire] (4) at (-3.75, -1) {$2n-4$};
		\node [style=paire] (5) at (-3.75, 1) {$2n-3$};
		\node [style=paire] (6) at (-1.75, -1) {$2n-2$};
		\node [style=paire] (7) at (-1.75, 1) {$2n-1$};
		\node [style=none] (8) at (-6.75, -1) {};
		\node [style=none] (9) at (-6.75, 1) {};
		\node [style=none] (10) at (-5, -1) {};
		\node [style=none] (11) at (-5, 1) {};
		\node [style=none] (12) at (-3.75, 1) {};
		\node [style=none] (13) at (-10.25, -1.5) {};
		\node [style=none] (14) at (-8.75, -1.5) {};
		\node [style=none] (15) at (-10.25, 2.5) {};
		\node [style=none] (16) at (-8.75, 2.5) {};
		\node [style=none] (17) at (-10.25, -1.5) {};
		\node [style=none] (18) at (-6.75, -1.5) {};
		\node [style=none] (19) at (-10.25, 2.5) {};
		\node [style=none] (20) at (-6.75, 2.5) {};
		\node [style=none] (21) at (-10.25, -1.5) {};
		\node [style=none] (22) at (-2.75, -1.5) {};
		\node [style=none] (23) at (-10.25, 2.5) {};
		\node [style=none] (24) at (-2.75, 2.5) {};
		\node [style=none] (25) at (-10.25, -1.5) {};
		\node [style=none] (26) at (-0.75, -1.5) {};
		\node [style=none] (27) at (-10.25, 2.5) {};
		\node [style=none] (28) at (-0.75, 2.5) {};
		\node [style=none] (29) at (-5, 2.5) {};
		\node [style=none] (30) at (-10.25, 2.5) {};
		\node [style=none] (31) at (-5, -1.5) {};
		\node [style=none] (32) at (-10.25, -1.5) {};
		\node [style=none] (33) at (-9.5, 1.75) {$X_0$};
		\node [style=none] (34) at (-7.75, 1.75) {$X_1$};
		\node [style=none] (35) at (-4, 1.75) {$X_{n-2}$};
		\node [style=none] (36) at (-1.75, 1.75) {$X_{n-1}$};
		\node [style=paire] (37) at (1.25, -1) {$0$};
		\node [style=paire] (38) at (1.25, 1) {$2$};
		\node [style=paire] (39) at (3.25, -1) {$1$};
		\node [style=paire] (40) at (3.25, 1) {$4$};
		\node [style=paire] (41) at (7.25, -1) {$2n-5$};
		\node [style=paire] (42) at (7.25, 1) {$2n-2$};
		\node [style=paire] (43) at (9.25, -1) {$2n-3$};
		\node [style=paire] (44) at (9.25, 1) {$2n-1$};
		\node [style=none] (45) at (4.25, -1) {};
		\node [style=none] (46) at (4.25, 1) {};
		\node [style=none] (47) at (6, -1) {};
		\node [style=none] (48) at (6, 1) {};
		\node [style=none] (50) at (0.75, -1.5) {};
		\node [style=none] (51) at (2.25, -1.5) {};
		\node [style=none] (52) at (0.75, 2.5) {};
		\node [style=none] (53) at (2.25, 2.5) {};
		\node [style=none] (54) at (0.75, -1.5) {};
		\node [style=none] (55) at (4.25, -1.5) {};
		\node [style=none] (56) at (0.75, 2.5) {};
		\node [style=none] (57) at (4.25, 2.5) {};
		\node [style=none] (58) at (0.75, -1.5) {};
		\node [style=none] (59) at (8.25, -1.5) {};
		\node [style=none] (60) at (0.75, 2.5) {};
		\node [style=none] (61) at (8.25, 2.5) {};
		\node [style=none] (62) at (0.75, -1.5) {};
		\node [style=none] (63) at (10.25, -1.5) {};
		\node [style=none] (64) at (0.75, 2.5) {};
		\node [style=none] (65) at (10.25, 2.5) {};
		\node [style=none] (66) at (6, 2.5) {};
		\node [style=none] (67) at (0.75, 2.5) {};
		\node [style=none] (68) at (6, -1.5) {};
		\node [style=none] (69) at (0.75, -1.5) {};
		\node [style=none] (70) at (1.5, 1.75) {$Y_0$};
		\node [style=none] (71) at (3.25, 1.75) {$Y_1$};
		\node [style=none] (72) at (7, 1.75) {$Y_{n-2}$};
		\node [style=none] (73) at (9.25, 1.75) {$Y_{n-1}$};
	\end{pgfonlayer}
	\begin{pgfonlayer}{edgelayer}
		\draw [style=usimy] (0) to (1);
		\draw [style=usimx] (7) to (6);
		\draw [style=usimy] (4) to (6);
		\draw [style=usimy] (7) to (5);
		\draw [style=usimx] (0) to (2);
		\draw [style=usimx] (1) to (3);
		\draw [style={none}] (16.center)
			 to (14.center)
			 to (13.center)
			 to (15.center)
			 to cycle;
		\draw [style={none}] (20.center)
			 to (18.center)
			 to (17.center)
			 to (19.center)
			 to cycle;
		\draw [style={none}] (24.center)
			 to (22.center)
			 to (21.center)
			 to (23.center)
			 to cycle;
		\draw [style={none}] (28.center)
			 to (26.center)
			 to (25.center)
			 to (27.center)
			 to cycle;
		\draw [style={none}] (32.center)
			 to (30.center)
			 to (29.center)
			 to (31.center)
			 to cycle;
		\draw [style=usimy] (3) to (9.center);
		\draw [style=usimy] (2) to (8.center);
		\draw [style=usimx] (12.center) to (11.center);
		\draw [style=usimx] (10.center) to (4);
		\draw [dashed] (9.center) to (11.center);
		\draw [dashed] (10.center) to (8.center);
		\draw [style=usimx, in=270, out=90] (37) to (38);
		\draw [style=usimy] (44) to (43);
		\draw [style=usimx] (41) to (43);
		\draw [style=usimx] (44) to (42);
		\draw [style=usimy] (37) to (39);
		\draw [style=usimy] (38) to (40);
		\draw [style={none}] (53.center)
			 to (51.center)
			 to (50.center)
			 to (52.center)
			 to cycle;
		\draw [style={none}] (57.center)
			 to (55.center)
			 to (54.center)
			 to (56.center)
			 to cycle;
		\draw [style={none}] (61.center)
			 to (59.center)
			 to (58.center)
			 to (60.center)
			 to cycle;
		\draw [style={none}] (65.center)
			 to (63.center)
			 to (62.center)
			 to (64.center)
			 to cycle;
		\draw [style={none}] (69.center)
			 to (67.center)
			 to (66.center)
			 to (68.center)
			 to cycle;
		\draw [style=usimx] (40) to (46.center);
		\draw [style=usimx] (39) to (45.center);
		\draw [style=usimy] (47.center) to (41);
		\draw [dashed] (46.center) to (48.center);
		\draw [dashed] (47.center) to (45.center);
		\draw [style=usimy] (48.center) to (42);
	\end{pgfonlayer}
\end{tikzpicture}
  }
\end{center}
Each of the $x$-level lines has $2$ elements, so does any $y$-level
line.  The reader can check that the permutation $ \sigma^x = (0,1)
(2,3) \dots (2i,2i+1) \dots (2n-1,2n-2)$ which corresponds to a
horizontal reflection on the figure on the left-hand side, induces a
graph automorphism of $\Aut_x(\cO)$, that is preserving the
$x$-distance and the type adjacencies.  Moreover, $\sigma^x$ acts
transitively on each $x$-level line. As the situation is completely
symmetric for $y$-level lines, this proves
Assumption~\ref{conj:xyorb_levlines} for cyclic orbits.  In this
section, we take the convention that the exponents on the permutation
indicate which type of level lines these automorphisms
stabilize. According to Theorem~\ref{thm:decoupl_levlines}, we find:

\begin{align*}
  (x,y) &=\left( \sinv-\frac{1}{2n} \sum_{j\text{ odd}} (n-j) \left(X_j - X_{j-1}\right)\right)
        - \left( \frac{1}{2n} \sum_{j\text{ odd}} (n-j)\left( Y_j - Y_{j-1} \right) \right)
          + \alpha.
\end{align*}
 In the above equation and in the rest of the section, we only give
the explicit expressions of $\wgammax, \wgammay$ and we write them
between parenthesis according to their order in the expression $(x,y)=
\wgammax +\wgammay + \alpha$. The above decoupling equation
corresponds to the decoupling construction obtained for small steps
walks in \cite[Theorem~4.11]{BBMR16}.

\subsubsection{The case of $\cO_{12}$}\label{subsect:decouplingO12}
Below are the $x$ and $y$-level lines  for the orbit type $\cO_{12}$:
\begin{center}
  \scalebox{0.5}{
    \begin{tikzpicture}[scale=1.5]
	\begin{pgfonlayer}{nodelayer}
		\node [style=label] (0) at (8, 4) {$3$};
		\node [style=label] (1) at (6.75, 5.75) {$6$};
		\node [style=label] (2) at (9.25, 5.75) {$7$};
		\node [style=label] (3) at (8, 1.25) {$0$};
		\node [style=label] (4) at (7, -0.75) {$1$};
		\node [style=label] (5) at (9, -0.75) {$2$};
		\node [style=label] (6) at (5, -3) {$4$};
		\node [style=label] (7) at (2.5, -3.5) {$11$};
		\node [style=label] (8) at (4.5, -5.5) {$10$};
		\node [style=label] (9) at (11, -3) {$5$};
		\node [style=label] (10) at (11.5, -5.5) {$9$};
		\node [style=label] (11) at (13.5, -3.5) {$8$};
		\node [style=label] (12) at (-8.75, 1.5) {$3$};
		\node [style=label] (13) at (-10, 3.75) {$6$};
		\node [style=label] (14) at (-7.5, 3.75) {$7$};
		\node [style=label] (15) at (-8.75, -1.25) {$0$};
		\node [style=label] (16) at (-9.75, -3.25) {$1$};
		\node [style=label] (17) at (-7.75, -3.25) {$2$};
		\node [style=label] (18) at (-12.25, -1) {$4$};
		\node [style=label] (19) at (-12.25, 1.5) {$11$};
		\node [style=label] (20) at (-15, 0.25) {$10$};
		\node [style=label] (21) at (-5.25, -1) {$5$};
		\node [style=label] (22) at (-2.5, 0.25) {$9$};
		\node [style=label] (23) at (-5.25, 1.5) {$8$};
		\node [style=none] (24) at (8, 2) {};
		\node [style=none] (25) at (6.25, -1.25) {};
		\node [style=none] (26) at (9.75, -1.25) {};
		\node [style=none] (27) at (8, 5) {};
		\node [style=none] (28) at (4, -3.75) {};
		\node [style=none] (29) at (12, -3.75) {};
		\node [style=none] (33) at (-8.75, 3) {};
		\node [style=none] (34) at (-8.75, -2.25) {};
		\node [style=none] (35) at (-8.75, 4.75) {};
		\node [style=none] (36) at (-8.75, -4.25) {};
		\node [style=none] (37) at (-8.75, 5.25) {};
		\node [style=none] (38) at (-8.75, -4.75) {};
		\node [style=none] (39) at (-8.75, 6) {};
		\node [style=none] (40) at (-8.75, -5.5) {};
		\node [style=none] (44) at (-8, 1.25) {$X_0$};
		\node [style=none] (45) at (-7, 2.25) {$X_1$};
		\node [style=none] (46) at (-5.5, 3.5) {$X_2$};
		\node [style=none] (47) at (-4.5, 4.5) {$X_3$};
		\node [style=none] (48) at (8, 8) {};
		\node [style=none] (49) at (1, -5) {};
		\node [style=none] (50) at (15, -5) {};
		\node [style=none] (51) at (8.75, 0.5) {$Y_0$};
		\node [style=none] (52) at (9.75, 2) {$Y_1$};
		\node [style=none] (53) at (10.75, 3.75) {$Y_2$};
	\end{pgfonlayer}
	\begin{pgfonlayer}{edgelayer}
		\draw [style={none}] (26.center)
			 to [bend left=60, looseness=0.75] (25.center)
			 to [bend left=60, looseness=0.75] (24.center)
			 to [bend left=60, looseness=0.75] cycle;
		\draw [style={none}] (29.center)
			 to [bend right=60, looseness=0.50] (27.center)
			 to [bend right=60, looseness=0.50] (28.center)
			 to [bend right=60, looseness=0.50] cycle;
		\draw [style={none}] (33.center)
			 to [bend left=90] (34.center)
			 to [bend left=90] cycle;
		\draw [style={none}] (35.center)
			 to [bend left=90] (36.center)
			 to [bend left=90] cycle;
		\draw [style={none}] (37.center)
			 to [bend left=90, looseness=1.75] (38.center)
			 to [bend left=90, looseness=1.75] cycle;
		\draw [style={none}] (39.center)
			 to [bend left=90, looseness=2.00] (40.center)
			 to [bend left=90, looseness=2.00] cycle;
		\draw [style={none}] (50.center)
			 to [bend right=60, looseness=0.50] (48.center)
			 to [bend right=60, looseness=0.50] (49.center)
			 to [bend right=60, looseness=0.50] cycle;
		\draw [style=usimx] (4) to (3);
		\draw [style=usimx] (3) to (5);
		\draw [style=usimx] (5) to (4);
		\draw [style=usimx] (1) to (0);
		\draw [style=usimx] (0) to (2);
		\draw [style=usimx] (2) to (1);
		\draw [style=usimx] (6) to (7);
		\draw [style=usimx] (7) to (8);
		\draw [style=usimx] (8) to (6);
		\draw [style=usimx] (9) to (10);
		\draw [style=usimx] (10) to (11);
		\draw [style=usimx] (11) to (9);
		\draw [style=usimy] (4) to (6);
		\draw [style=usimy] (5) to (9);
		\draw [style=usimy, bend right, looseness=0.50] (8) to (10);
		\draw [style=usimy] (3) to (0);
		\draw [style=usimy, bend right=45, looseness=0.75] (1) to (7);
		\draw [style=usimy, bend left=45, looseness=0.75] (2) to (11);
		\draw [style=usimx] (16) to (15);
		\draw [style=usimx] (15) to (17);
		\draw [style=usimx] (17) to (16);
		\draw [style=usimx] (13) to (12);
		\draw [style=usimx, in=-120, out=60] (12) to (14);
		\draw [style=usimx] (14) to (13);
		\draw [style=usimx] (18) to (19);
		\draw [style=usimx] (19) to (20);
		\draw [style=usimx] (20) to (18);
		\draw [style=usimx] (21) to (22);
		\draw [style=usimx] (22) to (23);
		\draw [style=usimx, in=90, out=-90] (23) to (21);
		\draw [style=usimy, bend left, looseness=0.75] (16) to (18);
		\draw [style=usimy, bend right] (17) to (21);
		\draw [style=usimy, dashed] (20) to (22);
		\draw [style=usimy] (15) to (12);
		\draw [style=usimy, bend right=45, looseness=0.75] (13) to (19);
		\draw [style=usimy, bend left=45, looseness=0.75] (14) to (23);
	\end{pgfonlayer}
\end{tikzpicture}}
\end{center}
Consider the following permutations of the vertices of the orbit: $
\tau^{x,y} = (1\,2) (4\,5) (6\,7) (9\,10) (8\,11)$ the vertical
reflection on both sides, $ \tau^x\,= (0\,3) (1\,6) (2\,7) (4\,11)
(5\,8)$ the horizontal reflection on the left-hand side, $\tau^y =
(0\,1\,2) (3\,4\,5) (6\,10\,8) (7\,11\,9)$ the $\frac{2\pi}{3}$
rotation on the right -hand side. The reader can check that these
automorphisms are elements of $\Aut_x(\cO)$ or $\Aut_y(\cO)$ according
to their exponents and that their action on the corresponding level
lines is transitive.

Therefore Assumption~\ref{conj:xyorb_levlines} holds for the orbit
type $\cO_{12}$. The cardinality of $\cO$ is $12$ and one can write
$\sinv = \frac{1}{12} \left(X_0 + X_1 + X_2 + X_3\right)$.  Thus,
according to Theorem~\ref{thm:decoupl_levlines}, the decoupling
equation is
\begin{align*}
  (x,y) &= \frac{2}{12} \left(\frac{X_2}{4} - \frac{X_3}{2}\right) + \frac{4+4+2}{12} \left(
          \frac{X_0}{2} - \frac{X_1}{4}\right)
          + \frac{3+6}{12} \left(\frac{Y_0}{3} - \frac{Y_1}{3}\right)
          + \sinv + \alpha \\
        &= \left(\frac{X_0}{2}-\frac{X_1}{8}+\frac{X_2}{8}\right)+\left(\frac{Y_0}{4}-\frac{Y_1}{4}\right) + \alpha.
\end{align*}

\subsubsection{Hadamard models}\label{subsubsect:Hadamard}

The notion of \emph{Hadamard} models has been introduced by Bostan,
Bousquet-Mélou and Melczer who proved that these models are always
$D$-finite \cite[Proposition~21]{bostan2018counting}. Hadamard models
are characterized by the shape of their Laurent polynomial: $S(X,Y) =
P(X) Q(Y) + R(X)$ for $P$, $Q$ and $R$ three Laurent polynomials.

The Hadamard models form an interesting class because their orbit is
always finite and in the form of a cartesian product. Indeed,
\cite[Proposition~22]{bostan2018counting} yields the existence of
distinct elements $x_0, \dots, x_{n-1}$ and $y_0, \dots, y_{m-1}$ in
$\K$ with $x_0=x$ and $y_0=y$ such that $\cO = \{x_i \st 1 \le 0 \le
n-1\} \times \{y_j \st 0 \le j \le m-1\}$.  As a consequence, the
orbits of the Hadamard models, even though their size might be
arbitrarily large, are always of diameter two. This means that the
distance between any two vertices is at most two as illustrated below:

\begin{center}
  \scalebox{0.5}{
    \begin{tikzpicture}[scale=3]
	\begin{pgfonlayer}{nodelayer}
		\node [style=label] (0) at (-7, 2) {$(x_0,y_0)$};
		\node [style=label] (1) at (-3, 2) {$(x_0,y_j)$};
		\node [style=label] (2) at (-7, 0) {$(x_i,y_0)$};
		\node [style=label] (3) at (-3, 0) {$(x_i,y_j)$};
		\node [style=none] (6) at (-9, 1) {};
		\node [style=none] (7) at (-1, 1) {};
		\node [style=none] (8) at (-8, 2) {};
		\node [style=none] (9) at (-2, 2) {};
		\node [style=none] (10) at (-4, 1.5) {$X_0$};
		\node [style=none] (11) at (-2, 0.25) {$X_1$};
		\node [style=label] (12) at (3, 2) {$(x_0,y_0)$};
		\node [style=label] (13) at (7, 2) {$(x_i,y_0)$};
		\node [style=label] (14) at (3, 0) {$(x_0,y_j)$};
		\node [style=label] (15) at (7, 0) {$(x_i,y_j)$};
		\node [style=none] (16) at (1, 1) {};
		\node [style=none] (17) at (9, 1) {};
		\node [style=none] (18) at (2, 2) {};
		\node [style=none] (19) at (8, 2) {};
		\node [style=none] (20) at (6, 1.5) {$Y_0$};
		\node [style=none] (21) at (8, 0.25) {$Y_1$};
	\end{pgfonlayer}
	\begin{pgfonlayer}{edgelayer}
		\draw [style={none}] (6.center)
			 to [bend right=90] (7.center)
			 to [bend left=270] cycle;
		\draw [style=usimy] (0) to (1);
		\draw [style=usimy] (2) to (3);
		\draw [style=usimx] (0) to (2);
		\draw [style=usimx] (1) to (3);
		\draw [style={none}] (8.center)
			 to [bend right=90, looseness=0.50] (9.center)
			 to [bend right=90, looseness=0.50] cycle;
		\draw [style={none}] (16.center)
			 to [bend right=90] (17.center)
			 to [bend left=270] cycle;
		\draw [style=usimx] (12) to (13);
		\draw [style=usimx] (14) to (15);
		\draw [style=usimy] (12) to (14);
		\draw [style=usimy] (13) to (15);
		\draw [style={none}] (18.center)
			 to [bend right=90, looseness=0.50] (19.center)
			 to [bend right=90, looseness=0.50] cycle;
	\end{pgfonlayer}
\end{tikzpicture}
  }
\end{center}

These orbit-types are very symmetric. The $x$-level lines $\levX_0$ is
$\{ (x,y_j) \st 0 \le j \le m-1 \}$ while $\levX_1=\{ (x_i,y_j) \st 0
\le j \le m-1 \mbox{ and } 1 \le i \le n-1 \}$. Thus, $|\levX_0|=m$
and $|\levX_1|=(n-1)m$. It is easy to prove that any element of
$\Aut_x(\cO)$ is of the form $ \phi^x_{\sigma,\tau} \colon
(x_i,y_j)\mapsto (\sigma(x_i),\tau(y_j)),$ for $\tau$ a permutation of
the set $\{ y_j \st 0 \le j \le m-1 \}$ and $\sigma$ a permutation of
$\{x_i \st 0 \le i \le n-1 \}$ such that $\sigma(x)=x$.  An analogous
description holds for the $y$-level lines and $\Aut_y(\cO)$ proving
that the Hadamard models satisfy Assumption~\ref{conj:xyorb_levlines}
and that $\Aut_x(\cO) \simeq S_{n-1} \times S_m$ and $\Aut_y(\cO)
\simeq S_{n} \times S_{m-1}$. Theorem~\ref{thm:decoupl_levlines} gives
the following decoupling:
\begin{align*}
  (x,y) &= \frac{m(n-1)}{nm} \left(\frac{X_0}{m} - \frac{X_1}{m(n-1)}\right)
          + \frac{n(m-1)}{nm} \left(\frac{Y_0}{n} - \frac{Y_1}{n(m-1)}\right) + \omega + \alpha \\
        &= \left(\frac{1}{m} X_0\right) + \left(\frac{m-1}{nm} Y_0 - \frac{1}{nm} Y_1\right) + \alpha
          = \left(\frac{1}{m} X_0\right) + \left(\frac{1}{n} Y_0 - \omega\right) + \alpha,
\end{align*}
 with $\sinv=\frac{1}{mn}(Y_0 +Y_1)$.
Note that any Hadamard model where $\deg_X \Ktld > 1$ and $\deg_Y \Ktld > 1$ always
contains a bicolored square, so the fraction $XY$ never admits a decoupling
(see Example~\ref{exa:obstructiondecouplingsquare}).

The complete description of the groups $\Aut_x(\cO)$ and $\Aut_y(\cO)$ obtained above is particularly useful 
to construct examples of orbits whose graph automorphisms are not necessarily Galois automorphisms as illustrated below.

\begin{exa}
Consider the nontrivial unweighted model defined by
$S(X,Y)=\left(X+\frac{1}{X}\right)\left(Y^n + \frac{1}{Y^n}\right)$.
Then by Proposition~22 in \cite{bostan2018counting}, the orbit has the
form
\[ \{ x, \frac{1}{x} \} \times \{ \zeta^i y, \zeta^i \frac{1}{y}\mbox{
for } i=0, \dots,n-1 \} \] where $\zeta$ is a primitive $n$-th root of
unit. Hence, the extension $k(\cO)$ equals $\C(x,y)=k(x,y)$.  Consider
the tower of field extension $k(x) \subset k(x, y^n) \subset
k(x,y)$. Since $k(x)$ coincides with $\C(x,y^n + \frac{1}{y^n})$ and
$k(x, y^n)$ with $\C(x,y^n)$, the multiplicativity of the degree of a
field extension yields
\[[k(\cO):k(x)] = [\C(x,y):\C(x,y^n)] \times [\C(x,y^n):\C(x,y^n +
\frac{1}{y^n})] = n \times 2.\] Indeed, since $x$ and $y$ are
algebraically independent over $\C$, the element $y^n$ is not a $m$-th
power in $\C(x,y^n)$ for $m$ dividing $n$. Thus, the minimal
polynomial of $y$ over the field $\C(x,y^n)$ is $Y^n-y^n$ so that
$[\C(x,y):\C(x,y^n)]$ equals $n$. Moreover, since $y^n$ does not
belong to $\C(x, y^n + \frac{1}{y^n})$, its minimal polynomial over
the later field is $Y^2 -( y^n + \frac{1}{y^n})Y +1$. Thus, $G_x
\subsetneq \Aut_x(\cO)$ because $G_x$ is a dihedral group of size $2n$
and $\Aut_x(\cO)$ is $S_{2n}$ by the above description.
\end{exa}

\section{The algebraic kernel curve and its covering}\label{ap:geometry}

In this section, we present an informal discussion on the geometric
framework for walks confined in a quadrant.  For small steps walks,
this approach was developed in \cite{KurkRasch,DHRS,
DreyfusHardouinRoquesSingerGenuszero} and allowed these authors to
construct analytic weak invariants \cite{RaschelJEMS, BBMR16},
difference equations \cite{KurkRasch,
DreyfusHardouinRoquesSingerGenuszero} as well as efficient algorithms
to compute the order of the group or some decoupling in the infinite
group case \cite{HardouinSingerselecta}.

For small steps models, this geometric framework amounts to interpret
the $X$ and $Y$-symmetries of the polynomial $\widetilde{K}(X,Y,t)$ as
automorphisms of a certain algebraic curve. For large steps models, we
shall see that this interpretation is still valid when the orbit is
finite but might be no longer true for an infinite orbit. Our
intention in this section is to introduce a geometric framework and
not to give a complete and systematic study of this geometric setting
for large steps walks which is a whole subject in its own right.

Though the kernel polynomial $\widetilde{K}(X,Y,t)$ is irreducible
over $\Q(t)[X,Y]$, it might be reducible over $\overline{\Q(t)}$, the
algebraic closure of $\Q(t)$. For small steps walks, Proposition~1.2
in \cite{DreyfusHardouinRoquesSingerGenuszero2} characterizes the
models, called \emph{degenerate}, whose associated kernel polynomial
is reducible over $\overline{\Q(t)}$.  These small steps models
correspond to the univariate cases described in
Section~\ref{subsect:algebraicitystrategy} plus the two cases where
the step polynomial $S(X,Y)$ is either a Laurent polynomial in $XY$ or
in $X/Y$.  The generating function $Q(X,Y,t)$ of a degenerate model
with small steps is always algebraic over $\Q(X,Y,t)$. One could
wonder if the degenerate models in the large steps situation still
coincide with the univariate cases described in
Section~\ref{subsect:algebraicitystrategy} and are therefore
algebraic. The question of the reducibility of the kernel polynomial
over $\overline{\Q(t)}$ requires some substantial work and we leave it
for further articles. Thus, we assume from now on that the kernel
polynomial is irreducible of positive degree $d_x=m_x +M_x$
(resp. $d_y=m_y +M_y$) in $X$ (resp. in $Y$) in the notation of
Section \ref{subsect:algebraicitystrategy}.

 Let us fix once for all a complex transcendental value for $t$ so
that $\Qb(t)$ embeds into $\C$. We denote by $\P^1(\C)$ the complex
projective line, that is, the set of equivalence classes
$[\alpha_0:\alpha_1]$ of elements $(\alpha_0,\alpha_1) \in \C^2$ up to
multiplication by a non-zero scalar. The projective line $\P^1(\C)$
can be identified to $\C \cup \{ \infty \}$ where $\C=\{ [\alpha_0:1]
\mbox{ with } \alpha_0 \in \C \}$ and $\infty$ is the point $[1:0]$.
We define the kernel curve $\Et$ as follows.
\[
\Et = \{([x_{0}:x_{1}],[y_{0}:y_{1}]) \in \P^1(\C) \times \P^1(\C) \ \vert \ \overline{K}(x_0,x_1,y_0,y_1,t) = 0\},
\]
where $\overline{K}(x_0,x_1,y_0,y_1,t)$ is the homogeneous polynomial
defined by
$x_1^{d_x}y_1^{d_y}\widetilde{K}(\frac{x_0}{x_1},\frac{y_0}{y_1},t)$
(see \cite[Section 2]{DreyfusHardouinRoquesSingerGenuszero2} for the
small steps case). The kernel curve $\Et \subset \P^1(\C) \times
\P^1(\C)$ is a projective algebraic curve. It is naturally equipped
with two projections $\pi_x :\Et \rightarrow \P^1_x,
([x_{0}:x_{1}],[y_{0}:y_{1}]) \mapsto [x_0:x_1]$ and $\pi_y :\Et
\rightarrow \P^1_y, ([x_{0}:x_{1}],[y_{0}:y_{1}]) \mapsto [y_0:y_1]$
where the notation $\P^1_x,\P^1_y$ emphasizes the variable on which
one projects. The curve $\Et$ is irreducible by our assumption on
$\widetilde{K}$. If we denote by $\mathrm{Sing}$ its singular locus,
that is, the set of points of $\Et$ at which the tangent is not
defined, its genus is given by the formula

\begin{equation}\label{equn:virtualgenus}
g(\Et)= 1+ d_xd_y -d_x-d_y -\sum_{P \in \mathrm{Sing}}  \sum_i \frac{m_i(P)(m_i(P)-1)}{2},
\end{equation}
where $m_i(P)$ is a positive integer standing for the multiplicity of
a point $P$, {that is, for every $\ell < m_i(P)$, the partial
derivatives of $\overline{K}$ of order $\ell$ vanish at $P$}
\cite[Exercise 5.6, Page 231-232 and Example 3.9.2, Page 393]{Hart}.
\begin{exa}\label{exa:algebraiccurveGmod}
  The kernel polynomial associated to the model $\Gmod$ is
  $\widetilde{K}(X,Y,t)= XY-t(1+XY^2+X^2+X^3Y^2+ \lambda X^2Y)$ One can
  easily check that the algebraic curve $\Et$ corresponding to the model
  $\Gmod$ is smooth\footnote{This means that $\Et$ has no singular
    point. Otherwise, one says that the curve is singular.}, so that its
  genus equals $2=1 +3.2 -3-2$.
\end{exa}

If the curve $\Et$ is smooth, it can be endowed with a structure of
compact Riemann surface (see \cite[Example 1.46]{Girondo}).  We recall
that the function field $\C(E)$ of an irreducible projective curve $E$
defined by some irreducible polynomial $F(X,Y)$ is the fraction field
of the $\C$-algebra $\C[X,Y]/(F)$ where $(F)$ is the polynomial ideal
generated by $F$.  The following categories are equivalent
\begin{itemize}
\item the  category of  smooth projective curves $E$  over $\C$ and non-constant morphisms,
\item  the category   of finitely generated field extensions $\C(E)$ of $\C$ of transcendence degree one and morphisms of field extensions,
\item the category of compact Riemann surfaces $E$ and their morphisms \cite[Remark 1.94 and Proposition~1.95]{Girondo}.
\end{itemize}

When the projective curve $E_t$ is singular, any automorphism of its
function field corresponds to a birational transformation of the curve
$E_t$ but, for simplicity of presentation, we assume from now on that
$\Et$ is smooth. The above equivalence of categories applied to the
two projections $\pi_x,\pi_y$ implies that $\C(\Et)$ is a field
extension of $\C(x)=\C(\P^1_x)$ and of $\C(y)=\C(\P^1_y)$.
 
  When the model is with small steps, the curve $\Et$ is of genus one
if $\Et$ is smooth (see
\cite[Proposition~2.1]{DreyfusHardouinRoquesSingerGenuszero2}). Moreover,
the field $\C(\Et)$ is an extension of degree $2$ of the fields
$\C(x)$ and $\C(y)$ and thereby a Galois extension of these two
fields. The Galois groups $\Gal(\C(\Et)|\C(x))$
(resp. $\Gal(\C(E_t)|\C(y))$) are cyclic of order two. Their
generators correspond via the aforementioned equivalence of categories
to two automorphisms $\Phi,\Psi$ of $\Et$ which are respectively the
deck transformations of the projections from $\Et$ to $\P^1_x$ and to
$\P^1_y$. These two automorphisms coincide on a Zariski open set of
$\Et \cap \C^2$ with the two birational involutions defined in
Section~\ref{sect:orbit}.

When the model has at least one large step, that is, $d_x$ or $d_y$ is
strictly greater than $2$, and the curve $\Et$ is irreducible and
smooth, \eqref{equn:virtualgenus} yields that the genus of $\Et$ is
strictly greater than one.  Hurwitz's Theorem
\cite[Theorem~2.41]{Girondo} implies that the group of automorphisms
of $\Et$, as the group of automorphism of any algebraic curve of genus
strictly greater than one, is of finite order bounded by
$84(g(\Et)-1)$. Moreover, the function field $\C(\Et)$ is in general
not a Galois extension of $\C(x)$ and of $\C(y)$ as illustrated in the
following example.

\begin{exa}
  In the notation of Example \ref{ex:gessel2_alg_6}, the field
  $\C(\Et)=\C(x,y)$ associated to the model $\Gmod$ is not Galois and is
  a proper subextension of the Galois extension $\C(\cO)=\C(x,y,z)$.
\end{exa}

If the genus of the curve $\Et$ is strictly greater than one, the same
holds for any any cover $\mathcal{M}$ of $\Et$ so that the group of
automorphisms $ \Aut(\mathcal{M})$ of any cover $\mathcal{M}$ of $\Et$
is finite \cite[Theorem~1.76]{Girondo}. Therefore if the curve $\Et$
is smooth, irreducible and the model has at least one large step,
Theorem~\ref{thm:fried}\footnote{which still holds if one replaces $k$
by $\C$.} shows that the existence of a Galois extension $M$ of
$\C(x)$ and $\C(y)$ is equivalent to the finiteness of the orbit of
the model. Indeed, the condition that $\Gal(M|\C(x))$ and
$\Gal(M|\C(y))$ generate a finite group of $\Aut(M)$ is automatic
since $\Aut(M)$ which is isomorphic to $\Aut(\mathcal{M})$ by the
above equivalence of category is finite. Applying once again the
equivalence of category, one finds that the Galois group $\Gal(M
|\C(x))$ (resp. $\Gal(M |\C(y))$) corresponds to the group of deck
transformations of $\pi_x$ (resp. $\pi_y$). We summarize this
discussion in the following Theorem.

\begin{thm}\label{thm:finiteorbitGaloiscovering}
  Assume that the model has  at least  one large step and that the curve $\Et$ is irreducible and smooth. The following statements are equivalent:

  \begin{itemize}
  \item the orbit of the walk is finite,
  \item there exists a cover $\mathcal{M}$ of $\Et$ which is a Galois cover of $\P^1_x$ and $\P^1_y$.
  \end{itemize}
  In that case, the group of the walk $G$ is isomorphic to the group of automorphims of $\mathcal{M}$ generated by the
  deck transformations of the covers $\mathcal{M} \rightarrow \P^1_x$ and $\mathcal{M} \rightarrow \P^1_y$.
\end{thm}

Under the assumption of Theorem~\ref{thm:finiteorbitGaloiscovering},
one can generalize the notion of group of the walk defined by the two
birational involutions $\Phi,\Psi$ for small steps models (see Section
\ref{sect:orbit} ) to the large step framework if and only if the
orbit of the walk is finite. If the orbit is finite, the group of the
walk is generated by the deck transformations of the two projections
of $\mathcal{M}$ onto $\P^1_x,\P^1_y$. It is in general no longer a
group of automorphisms of the kernel curve $\Et$, unless $\Et$ equals
$\mathcal{M}$, which happens only in very restricted situations.  If
the orbit is infinite and the curve $\Et$ is of genus greater than
one, one cannot realize the group of the walk as a group of
automorphisms of an algebraic curve which covers $\Et$.

\section*{Acknowledgement}
We thank Andrea Sportiello for suggesting the addition of weights to
the models $\calG_0$ and $\calG_1$ (of which the model
$\Gmod$ of Example~\ref{ex:gessel2_alg_1} is a generalization), Alin
Bostan for his advice on formal computation over the orbit, and
Mireille Bousquet-Mélou for her  inspiring  guidance, suggestions and review of
preliminary versions.

\appendix 
\section{Solving the model $\Gmod$} \label{sect:bmj}

In Section~\ref{subsect:algebraicitystrategy}, we illustrate how the
construction of pair of Galois invariants and Galois decoupling pairs
for the model $\Gmod$ allows us to construct explicit equations in one
catalytic variable satisfied by the sections $Q(X,0)$ and
$Q(0,Y)$. Theorem~3 in \cite{BMJ} implies that these sections are
algebraic which yields the algebraicity of the generating function
$Q(X,Y)$.  However, \cite{BMJ} actually gives a general method to
obtain explicit polynomial equations for the solutions of equations in
one catalytic variable.

In this section, we follow this method  to provide
an explicit polynomial equation for the excursion generating function $Q(0,0)$
attached to the model $\Gmod$. All the computations can be found in the Maple
worksheet (also on this \href{https://www.labri.fr/perso/pbonnet/}{webpage}) and we give here their guidelines.

We start from the functional equation obtained for $Q(0,Y)$,
because it is the simplest of the two and we recall below
its canonical form, with $F(Y) = Q(0,Y)$:
\begin{align*}
  F(Y) &= 1 +  t \bigg( t^{2} Y F \! \left(Y\right)\left({\Delta^{(1)}F} \! \left(Y \right)\right)^{2} +  \lambda  t F \! \left(Y\right){\Delta^{(1)} F} \! \left(Y \right)+t \left({\Delta^{(1)}F} \! \left(Y \right)\right)^{2} \\
       &+2  t F \! \left(Y\right){\Delta^{(2)}F} \! \left(Y \right) +Y F \! \left(Y \right)+\lambda {\Delta^{(2)}F} \! \left(Y \right) +2 {\Delta^{(3)}F} \! \left(Y \right) \bigg),
\end{align*}
where $\Delta$ is the \emph{discrete derivative}: $\Delta G(Y)=\frac{G(Y)-G(0)}{Y}$. Besides $F(Y)$, there are three unknown functions: $F(0)$ (the excursions series),
$F'(0)$ and $F''(0)$. The above equation can hence be rewritten as
\begin{equation} \label{eq:cat_1}
  P(F(Y),F(0),F'(0),F''(0),t,Y) = 0,
\end{equation}
with $P(x_0,x_1,x_2,x_3,t,Y)$ a polynomial with coefficients in $\Q(\lambda)$.

The method of Bousquet-Mélou and Jehanne consists in constructing more
equations from \eqref{eq:cat_1}. For that purpose, we search for
fractional power series \footnote{A fractional power series is an
element of $ \C[[t^{1/d}]]$ for some positive integer $d$.} $Y_i$'s
that are solutions of \eqref{eq:cat_1} and of the following equation
\begin{equation} \label{eq:cat_2}
  \left(\partial_{x_0}P\right)(F(Y_i),F(0),F'(0),F''(0),t,Y_i) = 0.
\end{equation}
Then the paper \cite{BMJ} points out  that any such  solution is also a solution of the following equation
\begin{equation} \label{eq:cat_3}
  \left(\partial_{Y}P\right)(F(Y_i),F(0),F'(0),F''(0),t,Y_i) = 0.
\end{equation}
Moreover, these solutions are double roots of
$D(F(0),F'(0),F''(0),t,Y)$ the discriminant of $P$ with respect to
$x_0$ \cite[Theorem 14]{BMJ}.  If there are enough fractional power
series $Y_i$'s (at least the number of unknown functions), then the
result of \cite{BMJ} provides ``enough'' independent polynomial
equations $P_i(X_0,X_1,X_2)$ relating the unknown functions (here
$F(0)$, $F'(0)$ and $F''(0)$) so that the dimension of the polynomial
ideal generated by the $P_i$'s is zero. This shows that one can
eliminate these multivariate polynomial equations to find a one
variable polynomial equation for each of the unknown series.

Eliminating $F'(0)$ between \eqref{eq:cat_1} and \eqref{eq:cat_3},
one finds a first equation between $Y_i$ and $F(Y_i)$:
\begin{equation} \label{eq:cat_4}
\begin{aligned}
-2 F({Y_i}) t \,{Y_i}^{4} & +{F(0)}^{2} \,t^{2} {Y_i} -4
{F(0)} {F({Y_i})} \,t^{2} {Y_i} +3 {Y_i} \,t^{2}
{F({Y_i})}^{2}-{F(0)} \lambda t {Y_i} \\ &+{F({Y_i})}
\lambda t {Y_i} +{F({Y_i})} \,{Y_i}^{3}-2 {F'(0)} t {Y_i} -{Y_i}^{3}-4 t
{F(0)} +4 {F({Y_i})} t = 0.
\end{aligned}
\end{equation}
Now, eliminating $F(Y_i)$ between \eqref{eq:cat_4} and \eqref{eq:cat_1},
and removing the trivially nonzero factors, we obtain
the following polynomial equation for the $Y_i$'s:
\begin{equation} \label{eq:cat_5}
2 t \,{Y_i}^{4}+\lambda  t {Y_i} -{Y_i}^{3}+2 t = 0.
\end{equation}
Using Newton polygon's method, we find  that, among the four roots of the 
irreducible polynomial above, exactly three are fractional power series
$Y_1$, $Y_2$ and $Y_3$  that are not  formal power series. The last root,  denoted $Y_0$, is a
Laurent series with a simple pole at $t=0$. Moreover, \eqref{eq:cat_5} yields
\[
  t = \frac{Y_0}{2 {Y_0}^4 + \lambda Y_0 + 2},
\]
so that $\Q(\lambda, t) \subset \Q(\lambda,Y_0)$.
Replacing $t$ by the above expression in \eqref{eq:cat_5} and factoring by $Y-Y_0$, we   obtain the minimal polynomial $M(Y_0,Y)$ 
satisfied by the series  $Y_1$, $Y_2$, $Y_3$ over $\Q(\lambda,Y_0)$ 
 as:
\begin{equation} \label{eq:cat_6}
M(Y_0,Y) = 2{Y_0}^3Y^3 - {Y_0}^2 \lambda Y - Y_0 \lambda Y^2 - 2 {Y_0}^2 - 2 Y_0 Y - 2 Y^2.
\end{equation}

This polynomial of degree $3$ is irreducible over the field $\Q(\lambda,Y_0)
\subset \Q(\lambda) ((t))$ because otherwise one of the series
$Y_i$'s would belong to $\Q(\lambda,Y_0)$ which is impossible
since the $Y_i$'s are not Laurent series in $t$. Since
$\Q(\lambda,Y_0, F(0),F'(0),F''(0)) \subset \Q(\lambda)
((t))$, the same argument shows that $M(Y_0,Y)$ remains irreducible
over $\Q(\lambda,Y_0, F(0),F'(0),F''(0))$.

Now, since the $Y_i$'s are double roots of
$D(F(0),F'(0),F''(0),t(Y_0),Y)$, the polynomial $M(Y_0,Y)^2$ must
divide $D(F(0),F'(0),F''(0),t(Y_0),Y)$ so that the remainder $R(Y)$ in
the euclidian division of $D(F(0),F'(0),F''(0),t(Y_0),Y)$ by
$M(Y_0,Y)^2$ should be identically zero. The polynomial $R(Y)$ has
degree at most $6$ (the discriminant has degree $12$ and $M(Y_0,Y)^2$
has degree $6$), and we write it as \[ R(Y) = e_0 + e_1 Y + e_2 Y^2 +
e_3 Y^3 + e_4 Y^4 + e_5 Y^5 + e_6 Y^6
\] with $e_i$ a polynomial in $Y_0$,$F(0)$,$F'(0)$ and $F''(0)$.
Hence, each of its coefficient gives an equation $e_i = 0$ on the
unknown functions in terms of $Y_0$.  We first eliminate $F''(0)$
between $e_0$ and $e_1$ which yields an equation $e_6$ between $Y_0$,
$F(0)$, $F'(0)$. We get another such equation $e_7$ by eliminating
$F''(0)$ between $e_0$ and $e_2$. Finally, eliminating $F'(0)$ between
$e_6$ and $e_7$ yields an equation $e_7$ over $\Q(\lambda)$ between
$Y_0$ and $F(0)$. The polynomial defining the equation $e_7$ factors
into two nontrivial irreducible factors. To decide which of these
factors is a polynomial equation for $F(0)$, we compute the first
terms of the $t$-expansion $F(0)=Q(0,0,t)$ (which is easy from the
functional equation for $Q(X,Y)$) and of $Y_0(t)$ (thanks to the
Newton method) and we plug these approximations in the two factors of
$e_7$.  One finds that $F(0)$ is algebraic of degree $8$ over
$\Q(\lambda)(Y_0)$. One eliminates $Y_0$ thanks to its functional
equation and, thanks to Maple, one verifies that $F(0)$ is algebraic
of degree $32$ over $\Q(t)$ (see the Maple \href{https://www.labri.fr/perso/pbonnet/}{worksheet}).  This gives the
following result:
\begin{prop}\label{prop:appendixminipolygmod}
  The series $Q(0,0)$ is algebraic of degree $8$ over $\Q(\lambda)(Y_0)$
  (for any $\lambda$).
  Hence, as $Y_0$ is of degree $4$ over $\Q(\lambda)(t)$, we conclude
  that $Q(0,0)$ is an algebraic series of degree $32$ over $\Q(\lambda)(t)$.
\end{prop}

We note that any step of our procedure remains valid if one
specializes $\lambda$ to $0$ and $1$ so that the
excursion series $Q(0,0)$ of the models $\calG_0$ and $\calG_1$
remains algebraic of degree $32$ over $\Q(\lambda)(Y_0)$.  

\section{Formal computation of decoupling with level lines} \label{sect:decoupl_comp}

As explained in Section \ref{subsect:effective_decoupling}, the
evaluation of a regular fraction at an arbitrary pair of elements in
the orbit is expensive from a computer algebra point of view.  We
describe below a family of $0$-chains called \emph{symmetric chains}
which are easy to evaluate on. We will then show that the level lines
introduced in Section\ref{subsect:effective_decoupling} can be described
explicitely in terms of these symmetric chains. Thus, under
Assumption~\ref{conj:xyorb_levlines},
Theorem~\ref{thm:decoupl_levlines} yields an expression of the
decoupling in the orbit in terms of symmetric chains which provides a
powerful implementation of the computation of the Galois decoupling of
a fraction (see the Sage worksheet (on this \href{https://www.labri.fr/perso/pbonnet/}{webpage}).).

\subsection{ Symmetric chains on the orbit }\label{app:varietiesorbit}
\begin{defi}
  Let $P(X)$ be a square-free polynomial in $\C(x,y)[Z]$.
  We define two finite subsets of $\K \times \K$ to be
  $V^1(P) = \{(u,v) \in \K \times \K \st P(u) = 0 \land S(x,y) = S(u,v) \}$ and
  $V^2(P) = \{(u,v) \in \K \times \K \st P(v) = 0 \land S(x,y) = S(u,v) \}$.
\end{defi}

We recall here a well known corollary of the theory of symmetric
polynomials (see \cite[Theorem~6.1]{Lang}).  Let $P(X)$ be a
polynomial with coefficients in a field $L$ and let $x_1, \dots, x_n$
be its roots taken with multiplicity in some algebraic closure of $L$. If $H(X)$ is a rational
fraction over $L$ with denominator relatively prime to $P(X)$, then
the sum $\sum_{i} H(x_i)$ is a well defined element of $L$.  There are
numerous effective algorithms to compute such a  sum based on resultants,
trace of a companion matrix, Newton formula\dots (see for example
\cite{bostan:inria}).

We extend these methods to the computation of $s = \sum_{(u,v) \in V^1(P)} H(u,v,\invS)$ for
$P$ a square-free polynomial such that $V^1(P) \subset \cO$ and $H(X,Y,t)$ a
regular fraction as follows. By definition of $V^1(P)$, we can rewrite $s$ as the double sum \[
  s = \sum_{u \sts  P(u) =0 } \, \sum_{v \sts \Ktld(u,v,\invS) = 0} H(u,v,\invS).
\]

Consider the sum $ \sum_{v \sts \Ktld(x,v,\invS)=0 } H(x,v,\invS)$. It is
a well-defined element of $k(x)$ which can be computed efficiently
since it is a symmetric function on the roots of the square-free
polynomial $\Ktld(x,Y,\invS)$. Let $\Sigma(X)$ be  in $k(X)$
such that \[\Sigma(x)=\sum_{v \sts \Ktld(x,v,\invS)=0 }
H(x,v,\invS).\] Since the group of the orbit $G$ acts transitively on
the orbit and preserves the adjacencies, it is easily seen that, for
any right coordinate of the orbit $u$, the sum
$\sum_{v \sts \Ktld(u,v,\invS)=0 } H(u,v,\invS)$ coincides with $\Sigma(u)$.  Then,
$s=\sum_{u \sts P(u) =0} \Sigma(u)$ is of the desired form and can also
be computed efficiently since it is a symmetric function on the roots
of the square-free polynomial $P$. The process is symmetric for
$V^2(P)$. These observations  motivate the following definition.
\begin{defi}
  A    \emph{symmetric chain} is   a $\C$-linear combination of $0$-chains of the form
  $\sum_{a \in V^i(P)} a$ with $P$ a square-free polynomial  such that
  $V^i(P) \subset \cO$.
\end{defi}

From the above discussion, any regular fraction $H(X,Y,t)$ can be
evaluated on a symmetric chain in an efficient way.

\subsection{Level lines as symmetric chains}
We now motivate the choice of level lines  introduced in  Section~\ref{subsect:effective_decoupling}, by showing
they are symmetric chains  which one can construct efficiently. We recall that the \emph{square-free part} of a polynomial $P$ in $K[Z]$
is the product of its distinct irreducible factors and  can be computed
as $P/\GCD(P,P')$.

Now, let $P $ be a polynomial in $\C(x,y)[Z]$. Then we denote by $R_{\Ktld,X}(P)$ 
the square-free part of $\Res(\Ktld(X,Z,\invS),P(X),X)$   in $\C(x,y)[Z]$.
Similarly, we define $R_{\Ktld,Y}(P) $ to
be the square-free part of $\Res(\Ktld(Z,Y,\invS),P(Y),Y)$ in $\C(x,y)[Z]$.
The following lemmas are straightforward so that we omit their proofs.

\begin{lem} \label{lem:adjunc_res}
  Let $P(Z)$ be a polynomial in $\C(x,y)[Z]$. Then,
  \[V^2(R_{\Ktld,X}(P)) = \{a \in \K \times \K \st \exists a' \in V^1(P),\; a \sim^y a'\} \]
    and 
  \[V^1(R_{\Ktld,Y}(P)) = \{a \in \K \times \K \st \exists a' \in V^2(P),\; a \sim^x a'\}.\]
\end{lem}

\begin{lem} \label{lem:adj_levlines}
  Let $i$ be a positive integer. Any element $a$ of $\levX_{i}$ is adjacent
  to some element $a'$ of $\levX_{i-1}$. Moreover, if $i$ is odd then $a \sim^y a'$ and
  if $i$ is even then $a \sim^x a'$.
\end{lem}

Now, we construct by induction a sequence of square-free polynomials
$(P^x_j(Z))_j \in \C(x,y)[Z]$ which satisfy the equations
\[V^1(P^x_{2i}) = \levX_{2i} \cup \levX_{2i-1} \mbox{ and }
V^2(P^x_{2i+1}) = \levX_{2i+1} \cup \levX_{2i} \mbox{  for all } i.\]

We set $P^x_0(Z) = Z-x$ so that $V^1(P^x_0) = \levX_0 \subset \cO$. Now, assume that  we have constructed the polynomials $P^x_j(Z)$   for  $j=0, \dots 2i$.
 By Lemma~\ref{lem:adjunc_res} and
the induction hypothesis, $V^2(R_{\Ktld,X}(P^x_{2i}))$ is composed
of all the vertices that are $y$-adjacent to some vertex in
$\levX_{2i} \cup \levX_{2i-1}$. Moreover, by the induction hypothesis,
$V^2(P^x_{2i-1})=\levX_{2i-1} \cup \levX_{2i}$. Hence, by
Lemma~\ref{lem:adj_levlines}
we find that \[
V^2(R_{\Ktld,X}(P^x_{2i})) \setminus V^2(P^x_{2i-1}) =\levX_{2i+1} \cup
\levX_{2i} . \]

Hence, if we define $P^x_{2i+1}$ to be $R_{\Ktld,X}(P^x_{2i})$
divided by its greatest common divisor with $P^x_{2i-1}$, then $P^x_{2i+1}$
is square-free, and the above equation ensures that
$V^2(P^x_{2i+1})  = \levX_{2i+1} \cup \levX_{2i}$. We construct
$P^x_{2i+2}$  using  similar arguments. Analogously, one can construct a sequence of square-free polynomials
$(P^y_j(Z))_j \in \C(x,y)[Z]$ which satisfy
\[V^1(P^y_{2i}) = \levY_{2i} \cup \levY_{2i-1} \mbox{ and }
V^2(P^y_{2i+1}) = \levY_{2i+1} \cup \levY_{2i} \mbox{  for all } i.\]
starting from $P^y_0(Z) = Z-y$.

As the $x$-level lines are disjoint sets of vertices, the $0$-chain associated
with $\levX_{i+1} \cup \levX_{i}$ is just the sum $X_{i+1} + X_{i}$. Hence,
as $X_0$ and all $X_{i+1} + X_{i}$ are symmetric chains, all the $X_i$
are symmetric chains as well. The same argument holds for $y$-level lines.
Note that, as expected, the coefficients of the $P^x_i$ are actually
in $k(x)$ and the coefficients of the $P^y_i$ are in $k(y)$. By Proposition~\ref{lem:liftingidentities}, one can identify $k(x)$ (resp. $k(y)$) with $\C(X,t)$ (resp. $\C(Y,t)$) by identifying $\invS$ with $t$, $x$ with $X$ and $y$ with $Y$ so that  the coefficients $P^x_i$ (resp. $P^y_i$) can be considered in  $\C(X,t)$ (resp. $\C(Y,t)$).

\section{Some more decoupling of orbit types} \label{appendix:other_orbits}

\subsection{The case of $\widetilde{\cO_{12}}$}
We represent below the $x$ and $y$-level lines for the orbit type $\widetilde{\cO_{12}}$:
\begin{center}
  \scalebox{0.6}{
  \begin{tikzpicture}[scale=1.5]
	\begin{pgfonlayer}{nodelayer}
		\node [style=label] (0) at (-6.5, 3.5) {$8$};
		\node [style=label] (1) at (-8, 3.5) {$3$};
		\node [style=label] (2) at (-7.25, 2) {$0$};
		\node [style=label] (3) at (-3.25, 3.5) {$11$};
		\node [style=label] (4) at (-4.75, 0.5) {$7$};
		\node [style=label] (5) at (-11.25, 3.5) {$9$};
		\node [style=label] (6) at (-9.75, 0.5) {$4$};
		\node [style=label] (7) at (-8, 0.5) {$1$};
		\node [style=label] (8) at (-6.5, 0.5) {$2$};
		\node [style=label] (9) at (-9, -1) {$5$};
		\node [style=label] (11) at (-7.25, -3.75) {$10$};
		\node [style=label] (12) at (-5.5, -1) {$6$};
		\node [style=none] (13) at (-7.25, 2.5) {};
		\node [style=none] (14) at (-8.5, 0.25) {};
		\node [style=none] (15) at (-6, 0.25) {};
		\node [style=none] (16) at (-7.25, 4.5) {};
		\node [style=none] (17) at (-10.25, -1) {};
		\node [style=none] (18) at (-4.25, -1) {};
		\node [style=none] (19) at (-2, 4.25) {};
		\node [style=none] (20) at (-12.5, 4.25) {};
		\node [style=none] (21) at (-7.25, -5) {};
		\node [style=none] (22) at (-6.5, 1.5) {$X_0$};
		\node [style=none] (23) at (-5.25, 2.5) {$X_1$};
		\node [style=none] (24) at (-4, 3.25) {$X_2$};
		\node [style=label] (25) at (8, 3.5) {$2$};
		\node [style=label] (26) at (6.5, 3.5) {$1$};
		\node [style=label] (27) at (7.25, 2) {$0$};
		\node [style=label] (28) at (11.25, 3.5) {$6$};
		\node [style=label] (29) at (9.75, 0.5) {$7$};
		\node [style=label] (30) at (3.25, 3.5) {$5$};
		\node [style=label] (31) at (4.75, 0.5) {$4$};
		\node [style=label] (32) at (6.5, 0.5) {$3$};
		\node [style=label] (33) at (8, 0.5) {$8$};
		\node [style=label] (34) at (5.5, -1) {$9$};
		\node [style=label] (35) at (7.25, -3.75) {$10$};
		\node [style=label] (36) at (9, -1) {$11$};
		\node [style=none] (37) at (7.25, 2.5) {};
		\node [style=none] (38) at (6, 0.25) {};
		\node [style=none] (39) at (8.5, 0.25) {};
		\node [style=none] (40) at (7.25, 4.5) {};
		\node [style=none] (41) at (4.25, -1) {};
		\node [style=none] (42) at (10.25, -1) {};
		\node [style=none] (43) at (12.5, 4.25) {};
		\node [style=none] (44) at (2, 4.25) {};
		\node [style=none] (45) at (7.25, -5) {};
		\node [style=none] (46) at (8, 1.5) {$Y_0$};
		\node [style=none] (47) at (9.25, 2.5) {$Y_1$};
		\node [style=none] (48) at (10.5, 3.25) {$Y_2$};
	\end{pgfonlayer}
	\begin{pgfonlayer}{edgelayer}
		\draw [style=usimx] (0) to (1);
		\draw [style=usimx] (1) to (2);
		\draw [style=usimx] (2) to (0);
		\draw [style=usimx] (7) to (6);
		\draw [style=usimx] (6) to (9);
		\draw [style=usimx] (9) to (7);
		\draw [style=usimy] (2) to (7);
		\draw [style=usimy] (7) to (8);
		\draw [style=usimy] (8) to (2);
		\draw [style=usimy] (5) to (1);
		\draw [style=usimy] (1) to (6);
		\draw [style=usimy] (6) to (5);
		\draw [style=usimy, in=180, out=0] (0) to (3);
		\draw [style=usimy] (0) to (4);
		\draw [style=usimy] (4) to (3);
		\draw [style=usimy] (9) to (11);
		\draw [style=usimy] (11) to (12);
		\draw [style=usimy] (9) to (12);
		\draw [style=usimx] (4) to (8);
		\draw [style=usimx] (12) to (8);
		\draw [style=usimx] (4) to (12);
		\draw [style=usimx, bend left=45] (3) to (11);
		\draw [style=usimx, bend left=45] (11) to (5);
		\draw [style=usimx, bend left=45] (5) to (3);
		\draw [style={none}] (13.center)
			 to [bend right=60, looseness=0.75] (14.center)
			 to [bend right=45, looseness=0.75] (15.center)
			 to [bend right=60, looseness=0.75] cycle;
		\draw [style={none}] (16.center)
			 to [bend right=60, looseness=0.75] (17.center)
			 to [bend right=45, looseness=0.75] (18.center)
			 to [bend right=60, looseness=0.75] cycle;
		\draw [style={none}] (19.center)
			 to [bend left=60] (21.center)
			 to [bend left=60] (20.center)
			 to [bend left=60] cycle;
		\draw [style=usimy] (25) to (26);
		\draw [style=usimy] (26) to (27);
		\draw [style=usimy] (27) to (25);
		\draw [style=usimy] (32) to (31);
		\draw [style=usimy] (31) to (34);
		\draw [style=usimy] (34) to (32);
		\draw [style=usimx] (27) to (32);
		\draw [style=usimx] (32) to (33);
		\draw [style=usimx] (33) to (27);
		\draw [style=usimx] (30) to (26);
		\draw [style=usimx] (26) to (31);
		\draw [style=usimx] (31) to (30);
		\draw [style=usimx, in=180, out=0] (25) to (28);
		\draw [style=usimx] (25) to (29);
		\draw [style=usimx] (29) to (28);
		\draw [style=usimx] (34) to (35);
		\draw [style=usimx] (35) to (36);
		\draw [style=usimx] (34) to (36);
		\draw [style=usimy] (29) to (33);
		\draw [style=usimy] (36) to (33);
		\draw [style=usimy] (29) to (36);
		\draw [style=usimy, bend left=45] (28) to (35);
		\draw [style=usimy, bend left=45] (35) to (30);
		\draw [style=usimy, bend left=45] (30) to (28);
		\draw [style={none}] (37.center)
			 to [bend right=60, looseness=0.75] (38.center)
			 to [bend right=45, looseness=0.75] (39.center)
			 to [bend right=60, looseness=0.75] cycle;
		\draw [style={none}] (40.center)
			 to [bend right=60, looseness=0.75] (41.center)
			 to [bend right=45, looseness=0.75] (42.center)
			 to [bend right=60, looseness=0.75] cycle;
		\draw [style={none}] (43.center)
			 to [bend left=60] (45.center)
			 to [bend left=60] (44.center)
			 to [bend left=60] cycle;
	\end{pgfonlayer}
\end{tikzpicture}
  }
\end{center}

We find the following automorphisms: $\tau^{x,y} = (1\,2) (3\,8) (4\,7)
(5\,6) (9\,11)$ the vertical reflection, $\tau^x = (0\,1\,2) (3\,5\,7)
(4\,6\,8) (9\,10\,11)$ the $\frac{2\pi}{3}$ rotation. One can check
that their action is transitive on the $x$-level lines.  As the
situation is completely symmetric for $y$-level lines,
Assumption~\ref{conj:xyorb_levlines} holds for this orbit type. Thus,
according to Theorem~\ref{thm:decoupl_levlines} and taking $\sinv =
\frac{1}{12} \left(X_0 + X_2 + X_3\right)$, the decoupling equation is
\begin{align*}
  (x,y) &= \frac{6+3}{12} \left(\frac{X_0}{3} - \frac{X_1}{6}\right)
          + \frac{6+3}{12} \left(\frac{Y_0}{3} - \frac{Y_1}{6}\right)
          + \sinv + \alpha \\
        &=\left(\frac{X_0}{3}-\frac{X_1}{24}+\frac{X_2}{12}\right)+\left(\frac{Y_0}{4}-\frac{Y_1}{8}\right) + \alpha.
\end{align*}

\subsection{The case of $\cO_{18}$}
We represent below the $x$ and $y$-level lines for the orbit type $\cO_{18}$.
\begin{center}
  \scalebox{0.5}{
  \begin{tikzpicture}[scale=1.5]
	\begin{pgfonlayer}{nodelayer}
		\node [style=label] (0) at (7.825, 1.25) {};
		\node [style=label] (1) at (6.825, -0.75) {$1$};
		\node [style=label] (2) at (8.825, -0.75) {$2$};
		\node [style=label] (3) at (10.575, -5) {$9$};
		\node [style=label] (4) at (10.825, -2.75) {$5$};
		\node [style=label] (5) at (12.825, -3.75) {$10$};
		\node [style=label] (6) at (13.325, 1.75) {$17$};
		\node [style=label] (7) at (14.825, 1) {$15$};
		\node [style=label] (8) at (13.825, 3) {$16$};
		\node [style=label] (9) at (8.825, 5.25) {$11$};
		\node [style=label] (10) at (6.825, 5.25) {$6$};
		\node [style=label] (11) at (7.825, 3.25) {$3$};
		\node [style=label] (12) at (2.325, 1.75) {$14$};
		\node [style=label] (13) at (1.825, 3) {$12$};
		\node [style=label] (14) at (0.825, 1) {$13$};
		\node [style=label] (15) at (5.075, -5) {$8$};
		\node [style=label] (16) at (4.825, -2.75) {$4$};
		\node [style=label] (17) at (2.825, -3.75) {$7$};
		\node [style=label] (18) at (-12.425, 1.25) {$3$};
		\node [style=label] (19) at (-13.425, 3.25) {$6$};
		\node [style=label] (20) at (-11.425, 3.25) {$11$};
		\node [style=label] (21) at (-8.425, 7.25) {$17$};
		\node [style=label] (22) at (-9.425, 4.25) {$16$};
		\node [style=label] (23) at (-6.425, 5.25) {$15$};
		\node [style=label] (24) at (-8.425, -7.25) {$9$};
		\node [style=label] (25) at (-6.425, -5.25) {$10$};
		\node [style=label] (26) at (-9.425, -4.25) {$5$};
		\node [style=label] (27) at (-11.425, -3.25) {$2$};
		\node [style=label] (28) at (-13.425, -3.25) {$1$};
		\node [style=label] (29) at (-12.425, -1.25) {$0$};
		\node [style=label] (30) at (-16.425, -7.25) {$8$};
		\node [style=label] (31) at (-15.425, -4.25) {$4$};
		\node [style=label] (32) at (-18.425, -5.25) {$7$};
		\node [style=label] (33) at (-16.425, 7.25) {$14$};
		\node [style=label] (34) at (-15.425, 4.25) {$12$};
		\node [style=label] (35) at (-18.425, 5.25) {$13$};
		\node [style=label] (36) at (7.825, 1.25) {$0$};
		\node [style=none] (37) at (7.825, 2.25) {};
		\node [style=none] (38) at (6.075, -1.5) {};
		\node [style=none] (39) at (9.575, -1.5) {};
		\node [style=none] (40) at (7.825, 4.25) {};
		\node [style=none] (41) at (4.325, -3.5) {};
		\node [style=none] (42) at (11.325, -3.5) {};
		\node [style=none] (46) at (7.825, 9.25) {};
		\node [style=none] (47) at (-0.425, -7) {};
		\node [style=none] (48) at (15.825, -7) {};
		\node [style=none] (49) at (-12.425, 2.25) {};
		\node [style=none] (50) at (-12.425, -2.25) {};
		\node [style=none] (51) at (-12.425, 4.5) {};
		\node [style=none] (52) at (-12.425, -4.5) {};
		\node [style=none] (53) at (-12.425, 7.25) {};
		\node [style=none] (54) at (-12.425, -7.25) {};
		\node [style=none] (55) at (-12.425, 9.25) {};
		\node [style=none] (56) at (-12.425, -9.25) {};
		\node [style=none] (57) at (8.75, 1) {$Y_0$};
		\node [style=none] (58) at (9.25, 2.5) {$Y_1$};
		\node [style=none] (59) at (9.75, 4.25) {$Y_2$};
		\node [style=none] (60) at (10.75, 6.75) {$Y_3$};
		\node [style=none] (61) at (-12, 0) {$X_0$};
		\node [style=none] (62) at (-10.5, 1.25) {$X_1$};
		\node [style=none] (63) at (-8.75, 3) {$X_2$};
		\node [style=none] (64) at (-7, 4.25) {$X_3$};
		\node [style=none] (65) at (7.825, 7) {};
		\node [style=none] (66) at (2.575, -5.25) {};
		\node [style=none] (67) at (13.075, -5.25) {};
	\end{pgfonlayer}
	\begin{pgfonlayer}{edgelayer}
		\draw [style=usimy] (0) to (11);
		\draw [style=usimy, dotted] (3) to (12);
		\draw [style=usimy, dotted] (6) to (15);
		\draw [style=usimy] (2) to (4);
		\draw [style=usimy] (5) to (7);
		\draw [style=usimy] (8) to (9);
		\draw [style=usimy] (13) to (10);
		\draw [style=usimy] (14) to (17);
		\draw [style=usimy] (16) to (1);
		\draw [style=usimx] (1) to (0);
		\draw [style=usimx] (0) to (2);
		\draw [style=usimx] (2) to (1);
		\draw [style=usimx] (4) to (5);
		\draw [style=usimx] (5) to (3);
		\draw [style=usimx] (3) to (4);
		\draw [style=usimx] (16) to (15);
		\draw [style=usimx] (15) to (17);
		\draw [style=usimx] (17) to (16);
		\draw [style=usimx] (14) to (12);
		\draw [style=usimx] (12) to (13);
		\draw [style=usimx] (13) to (14);
		\draw [style=usimx] (11) to (9);
		\draw [style=usimx] (9) to (10);
		\draw [style=usimx] (10) to (11);
		\draw [style=usimx] (6) to (7);
		\draw [style=usimx] (7) to (8);
		\draw [style=usimx] (8) to (6);
		\draw [style=usimy] (18) to (29);
		\draw [style=usimy, dotted] (21) to (30);
		\draw [style=usimy, dotted] (24) to (33);
		\draw [style=usimy] (20) to (22);
		\draw [style=usimy, bend left=15] (23) to (25);
		\draw [style=usimy] (26) to (27);
		\draw [style=usimy] (31) to (28);
		\draw [style=usimy, bend left=15] (32) to (35);
		\draw [style=usimy] (34) to (19);
		\draw [style=usimx] (19) to (18);
		\draw [style=usimx] (18) to (20);
		\draw [style=usimx] (20) to (19);
		\draw [style=usimx] (22) to (23);
		\draw [style=usimx] (23) to (21);
		\draw [style=usimx] (21) to (22);
		\draw [style=usimx] (34) to (33);
		\draw [style=usimx] (33) to (35);
		\draw [style=usimx] (35) to (34);
		\draw [style=usimx] (32) to (30);
		\draw [style=usimx] (30) to (31);
		\draw [style=usimx] (31) to (32);
		\draw [style=usimx] (29) to (27);
		\draw [style=usimx] (27) to (28);
		\draw [style=usimx] (28) to (29);
		\draw [style=usimx] (24) to (25);
		\draw [style=usimx] (25) to (26);
		\draw [style=usimx] (26) to (24);
		\draw [style={none}] (39.center)
			 to [bend left=45] (38.center)
			 to [bend left=60, looseness=0.75] (37.center)
			 to [bend left=60, looseness=0.75] cycle;
		\draw [style={none}] (42.center)
			 to [bend left=60, looseness=0.50] (41.center)
			 to [bend left=75, looseness=0.50] (40.center)
			 to [bend left=75, looseness=0.50] cycle;
		\draw [style={none}] (48.center)
			 to [bend left=45, looseness=0.75] (47.center)
			 to [bend left=60, looseness=0.75] (46.center)
			 to [bend left=60, looseness=0.75] cycle;
		\draw [style={none}] (49.center)
			 to [bend left=90, looseness=0.75] (50.center)
			 to [bend right=270, looseness=0.75] cycle;
		\draw [style={none}] (51.center)
			 to [bend left=90, looseness=1.25] (52.center)
			 to [bend left=90, looseness=1.25] cycle;
		\draw [style={none}, bend left=90, looseness=1.25] (53.center) to (54.center);
		\draw [style={none}, bend left=90, looseness=1.25] (54.center) to (53.center);
		\draw [style={none}, bend left=90, looseness=1.50] (55.center) to (56.center);
		\draw [style={none}, bend left=90, looseness=1.50] (56.center) to (55.center);
		\draw [style={none}] (67.center)
			 to [bend left=60, looseness=0.50] (66.center)
			 to [bend left=75, looseness=0.50] (65.center)
			 to [bend left=75, looseness=0.50] cycle;
	\end{pgfonlayer}
\end{tikzpicture}
  }
\end{center}
We present some elements belonging to the groups $\Aut_x(\cO)$ and
$\Aut_y(\cO)$:\\ $\tau^{xy} =
(1\,2)(6\,11)(4\,5)(7\,10)(8\,9)(13\,15)(14\,17)(12\,16)$ the vertical
reflection, \\ $\tau^y = (0\,1\,2) (3\,4\,5) (6\,7\,9) (8\,10\,11)
(12\,13\,14) (15\,16\,17)$ the $\frac{2\pi}{3}$ rotation for $d_y(v)
\le 2$ + rotating each "ear", \\ $ \tau^x_1 = (0\,3) (1\,6) (2\,11)
(4\,12) (5\,16) (7\,13) (8\,14) (9\,17) (10\,15)$ the horizontal
reflection,\\ $ \tau^x_2 = (15\,17) (8\,10) (4\,5) (7\,9) (13\,14)
(1\,2)$ the pinching of the upper "arms".\\ The reader can check that
these elements act transitively on their respective level lines which
proves Assumption~\ref{conj:xyorb_levlines} for $\cO_{18}$. Thus,
according to Theorem~\ref{thm:decoupl_levlines} and taking $\sinv =
\frac{1}{18} \left(X_0 + X_1 + X_2 + X_3\right)$, the decoupling
equation is
\begin{align*}
  (x,y) &= \frac{8}{18} \left(\frac{X_2}{4} - \frac{X_3}{8}\right)
          + \frac{4+4+8}{18} \left(\frac{X_0}{2} - \frac{X_1}{4}\right)
          + \frac{6}{18} \left(\frac{Y_2}{6} - \frac{Y_3}{6}\right)
          + \frac{3+6+6}{18} \left(\frac{Y_0}{3} - \frac{Y_1}{3}\right)
          + \sinv + \alpha \\
        &= \left(\frac{X_0}{2}-\frac{X_1}{6}+\frac{X_2}{6}\right)
          +\left(\frac{5 Y_0}{18}-\frac{5 Y_1}{18}+\frac{Y_2}{18}-\frac{Y_3}{18}\right) + \alpha.
\end{align*}

\subsection{Fan models} \label{subsubsect:fanmod}
We study a class of models derived from the ones arising in the
enumeration of plane bipolar orientations (see \cite{BM2021plane}). The \emph{fan} models are derived from those  introduced in
\cite[Equation $(7)$]{BM2021plane} by a  horizontal reflection.

\begin{defi}
    For $i \ge 0$, define $V_i(X,Y) = \sum_{0 \le j \le i} X^j Y^{i-j}$. If $z_1, \dots, z_p$ are
    complex weights, with $z_p$ being nonzero,
    we define the \emph{$p$-fan} to be the model with step polynomial
    \[S(X,Y) = \frac{1}{XY} + \sum_{i \le p} z_i V_i(X,Y).\]
  \end{defi}
By \cite[Proposition~3]{bostan2018counting}, the orbits
of models related to one another by a reflection are isomorphic so
that one can directly use the orbit computations of
Proposition~4.4 in \cite{BM2021plane} to compute the orbit of a
$p$-fan.
\begin{prop}
    Let $x_0, \dots, x_p$ be defined as the roots of the equation $S(X,y) = S(x,y)$  with $x_0 = x$
    and $x_{p+1} = y$.
    Moreover, for $0 \le i \le p+1$, denote
    $y_i = x_i$.
   
    In particular, $y_{p+1} = y$.
    Then the pairs $(x_i, y_j)$ with $i \neq j$ form the orbit
    of the walk for the $p$-fan.
\end{prop}
Note that all these models have small backwards steps and that they
all have an $X/Y$ symmetry. As a result, the orbit is of size
$(p+2)(p+1)$, and the cardinalities of the level lines are $|\levX_0|
= p+1$, $|\levX_2| = p+1$ and $|\levX_1| = p(p+1)$. The $y$-level
lines are symmetric. Below is a depiction of this orbit type, with the
indices $i$ and $j$ satisfy $0 < i \neq j < p+1$. Note that the orbit
of the $p$-fan contains a bicolored square, hence no decoupling of
$XY$ is possible (see Example~\ref{exa:obstructiondecouplingsquare}).
\begin{center}
    \begin{tikzpicture}
	\begin{pgfonlayer}{nodelayer}
		\node [style=paire] (0) at (-5, 5) {$(x_0,y_{p+1})$};
		\node [style=paire] (1) at (0, 5) {$(x_i,y_{p+1})$};
		\node [style=paire] (2) at (0, 0) {$(x_i,y_j)$};
		\node [style=paire] (3) at (0, -5) {$(x_{p+1},y_j)$};
		\node [style=paire] (4) at (5, -5) {$(x_{p+1},y_0)$};
		\node [style=paire] (5) at (5, 0) {$(x_i,y_0)$};
		\node [style=paire] (6) at (-5, 0) {$(x_0,y_j)$};
		\node [style=none] (9) at (3, 9.75) {};
		\node [style=none] (11) at (-11, 3) {};
		\node [style=none] (12) at (-11, -3) {};
		\node [style=none] (14) at (-5, 7.75) {};
		\node [style=none] (20) at (-5, 7.75) {$\levX_0$};
		\node [style=none] (15) at (0, 7.75) {$\levX_1$};
		\node [style=none] (16) at (5, 7.75) {$\levX_2$};
		\node [style=none] (17) at (-9, 5) {$\levY_0$};
		\node [style=none] (18) at (-9, 0) {$\levY_1$};
		\node [style=none] (19) at (-9, -5) {$\levY_2$};
		\node [style=none] (22) at (3, -8) {};
		\node [style=none] (23) at (-3, -8) {};
		\node [style=none] (24) at (8, -3) {};
		\node [style=none] (25) at (8, 3) {};
		\node [style=none] (27) at (-3, 9.75) {};
	\end{pgfonlayer}
	\begin{pgfonlayer}{edgelayer}
		\draw [style=usimy] (0) to (6);
		\draw [style=usimy] (1) to (2);
		\draw [style=usimy] (2) to (5);
		\draw [style=usimy] (3) to (4);
		\draw [style=usimx] (6) to (2);
		\draw [style=usimx] (2) to (3);
		\draw [style=usimx] (0) to (1);
		\draw [style=usimx] (5) to (4);
		\draw [style=usimx] (6) to (3);
		\draw [style=usimy] (1) to (5);
		\draw [style={light_grid}] (11.center) to (25.center);
		\draw [style={light_grid}] (12.center) to (24.center);
		\draw [style={light_grid}] (23.center) to (27.center);
		\draw [style={light_grid}] (9.center) to (22.center);
	\end{pgfonlayer}
\end{tikzpicture}
\end{center}
The groups  $\Aut_x(\cO)$ and $\Aut_y(\cO)$ contain in particular
the following family of automorphisms $ \phi^x_{\sigma,\tau} \colon
(x_i,y_j) \mapsto (\sigma(x_i),\tau(y_j))$ with $\sigma$ and $\tau$
some permutations such that $\sigma(x_0) = x_0$, $
\phi^y_{\sigma,\tau} \colon (x_i,y_j) \mapsto (\sigma(x_i),\tau(y_j))$
with $\sigma$ and $\tau$ some permutations such that $\tau(y_{p+1}) =
y_{p+1}$. This family of automorphisms acts transitively on the level
lines proving Assumption~\ref{conj:xyorb_levlines}. Thus using
Theorem~\ref{thm:decoupl_levlines} we obtain the decoupling equation
of $(x,y)$ as
\begin{align*}
  (x,y) &= \frac{(p+1)+p(p+1)}{(p+1)(p+2)} \left(\frac{X_0}{p+1} - \frac{X_1}{p(p+1)}\right)
          + \frac{(p+1)+p(p+1)}{(p+1)(p+2)} \left(\frac{Y_0}{p+1} - \frac{Y_1}{p(p+1)}\right)
          + \sinv + \alpha \\
    &=\left(\frac{X_0}{p +1}-\frac{X_1}{p \left(p +1\right) \left(p +2\right)}+\frac{X_2}{\left(p +1\right) \left(p +2\right)}\right)+\left(\frac{Y_0}{ p +2} -\frac{Y_1}{p \left(p +2\right)}\right)+\alpha.
\end{align*}

\section{Computation of a Galois group : Hadamard models} \label{sec:groupe-hadam-empl}

Consider $S(X,Y)=P(X)Q(Y)+R(X)$ a Hadamard model, with $PR$ and $Q$
nonconstant Laurent polynomials over $\C$. We first note that the
pair $(\frac{t-R(X)}{P(X)},Q(Y))$ is a pair of nontrivial Galois
invariants, hence the orbit of a Hadamard model is always finite by
Theorem~\ref{thm:fried}.  One can also easily describe its field of Galois invariants.

\begin{prop}
The field of Galois invariants of a Hadamard model given by $S(X,Y)=P(X)Q(Y)+R(X)$ coincides with $k(Q(y))$.
\end{prop}
\begin{proof}
Writing $Q(Y) = A(Y)/B(Y)$ with $A$ and $B$ relatively
prime, we know that the right coordinates of the orbit are the
roots of the polynomial $\mu_y(Y) = B(Y)-A(Y)Q(y) \in k(Q(y))[Y] \subset \kinv[Y]$.
Thus, by Section \ref{sect:effectiveinvariants}, the coefficients of this
polynomial generate the field of Galois invariants, implying that
$k(Q(y)) \subset \kinv \subset k(Q(y))$, which shows the claim.
\end{proof}

 The form of the step polynomial of Hadamard models is a strong constraint on the orbit  which has the form of a Cartesian product as described 
 below. 

\begin{prop}[Proposition~3.22 in
\cite{bostan2018counting}]\label{prop:orbitHadamard}
The orbit of a Hadamard model given by $S(X,Y)=P(X)Q(Y)+R(X)$ is of the form $\bold{x} \times \bold{y}$ where  $\bold{x} = x_0,\dots,x_{m-1}$ the
$m$ distinct solutions $x_i$ of $P(X)Q(y)+R(X)=P(x)Q(y)+R(x)$ and 
 $\bold{y} = y_0,\dots,y_{n-1}$ the $n$
distinct solutions $y_i$ of $Q(y)=Q(Y)$.  Hence, the field $k(\cO)$
is equal to $\C(\bold{x},\bold{y})$. 
\end{prop}

 Our goal in the rest of this section is to give an explicit
description of the group of the walk for a Hadamard model when the
step polynomial is of the form $S(X,Y)=Q(Y) +R(X)$ or $P(X)Q(Y)$. In
that situation, we shall prove that the group of the walk is a direct
product of two simple Galois groups.

\begin{prop}\label{prop:GaloisgroupsplitHadamard}
  Consider a Hadamard model with step polynomial of the form $Q(Y) +P(X)$ or $P(X)Q(Y)$. The following holds.
  \begin{itemize}
  \item The field $\kinv$ is  $\C(P(x),Q(y))$.
  \item In the notation of Proposition  \ref{prop:orbitHadamard}, the elements of $\bold{x}$ satisfy $P(x_i)=P(x)$  and the field extensions
    $\C(\bold{x})|\C(P(x))$ and $\C(\bold{y}) | \C(Q(y))$ are both Galois
    with respective Galois groups $H_x$ and $H_y$.
  \item The group of the walk $\Gal(k(\cO)|\kinv)$ is
    isomorphic to $H_x \times H_y$.
  \end{itemize}
\end{prop}

Before proving Proposition \ref{prop:GaloisgroupsplitHadamard}, we recall some terminology. We say that two field
extensions $L|K$ and $M|K$, subfields of a common field $\Omega$, are
\emph{algebraically independent} if any finite set of elements of $L$,
that are algebraically independent over $K$, remains algebraically
independent over $M$. We say that $L|K$ and $M|K$ are \emph{linearly
disjoint} over $K$ if any finite set of elements of $L$, that are
$K$-linearly independent, are linearly independent over $M$. The \emph{ field
compositum } of $L$ and $M$ is the smallest subfield of $\Omega$ that
contains $L$ and $M$. Finally, we say that $L|K$ is a regular field
extension if $K$ is relatively algebraically closed in $L$ and $L|K$
is separable. We recall that $K$ is relatively algebraically closed in
$L$ if any element of $L$ that is algebraic over $K$ belongs to
$K$. Note that in our setting, all fields are in characteristic zero
so $L|K$ is always separable.

\begin{proof}
  The proof of the first two items is obvious. First, let us prove that
  $\C(\bold{x}, Q(y))| \C(P(x),Q(y))$ is Galois with Galois group
  isomorphic to $H_x$.  We remark that since $x$ and $y$ are
  algebraically independent over $\C$, the field extension
  $\C(P(x),Q(y)) |\C(P(x))$ is purely transcendental of transcendence
  degree one, hence regular. Since $\C(\bold{x}) |\C(P(x))$ is an
  algebraic extension, the element $Q(y)$ remains transcendental over
  $\C(\bold{x})$.  Thus, the field extensions $\C(\bold{x})$ and
  $\C(P(x),Q(y))$ are algebraically independent over $\C(Q(y))$. Thus,
  by Lemma~2.6.7 in \cite{FriedJarden}, the fields $\C(\bold{x}) $ and
  $\C(P(x),Q(y))$ are linearly disjoint over $\C(P(x)$.  Then, the field
  $\C(\bold{x},Q(y))$ that is the compositum of $\C(\bold{x}) $ and
  $\C(P(x),Q(y))$, is Galois with Galois group isomorphic to $H_x$ (see
  page 35 in \cite{FriedJarden}). Analogously, one can prove that
  $\C(\bold{y},P(x)) | \C(P(x))$ is Galois with Galois group $H_y$.
  
  To conclude, we note that the field extension $\C(\bold{x}) |
  \C$ is regular of transcendence degree~$1$. Since $x$ is
  transcendental over $\C(\bold{y})$, the fields extensions
  $\C(\bold{x})$ and $\C(\bold{y})$ are algebraically independent over
  $\C$ and therefore linearly disjoint over $\C$ by Lemma~2.6.7 in
  \cite{FriedJarden}. By the tower property of the linear disjointness
  (Lemma~2.5.3 in \cite{FriedJarden}), we find that $\C(\bold{x},Q(y))$
  is linearly disjoint from $\C(\bold{y})$ over $\C(Q(y))$. Using once
  again the tower property, we conclude that $\C(\bold{x},Q(y))$ and
  $\C(\bold{y},P(x))$ are linearly disjoint over
  $\kinv=\C(P(x),Q(y))$. Lemma~2.5.6 implies that the following restriction map
  is a group isomorphism:\\

  \resizebox{.9\hsize}{!}{$
    \begin{array}{rcl}
      G=\Gal(\C(\bold{x},\bold{y}) |\C(P(x),Q(y))) & \longrightarrow & \Gal( \C(\bold{x},Q(y))| \C(P(x),Q(y))) \times \Gal(\C(\bold{y},P(x))| \C(P(x),Q(y))) \\
      \sigma & \longmapsto & (\sigma|_{\C(\bold{x},Q(y))}, \sigma|_{\C(\bold{y},P(x))}).
    \end{array}
    $}
  \\

  By the above, we conclude
  that $G$ is isomorphic to $H_x \times H_y$.
\end{proof}

\section{Other algebraic models}\label{sect:otheralgmodels}

In the classification of  models with small steps, four of them were proved algebraic. Among them,
the so called \emph{Gessel Model}, given by the Laurent polynomial $(1+1/Y)/X + (1+Y) X$.
It was a notoriously difficult model to study, and the first known proof of
algebraicity of its full generating function used heavy
computer algebra (see~\cite{KauersBostan}). Among other proofs of this result,
one relied on the general strategy developped in~\cite{BBMR16} as
presented in Section~\ref{subsect:alg_strat}.
It is noteworthy that no purely combinatorial proof of this
result yet exists.

In a private communication, Mireille Bousquet-Mélou suggested that we explore with our tools a  new family of large steps models
$(\calH_n)_n$ which she expected to have a finite orbit for any non-negative integer $n$. These models  are 
obtained from the Gessel model by stretching the two rightmost steps. 
More precisely, they are defined through the following Laurent polynomial:
\[
  H_n(X,Y) = (1+1/Y)/X + (1+Y) X^n.
\]
Unfortunately, unlike the Gessel model ($\calH_1$), we
checked that the fraction $XY$ does not admit a decoupling
(for $n \le 4$), so it seems unlikely that any of these larger
models is algebraic.

Now,  Proposition~7.3 in~\cite{BBMR16} implies for the Gessel model with small
steps that the fractions of the form $X^a Y^b$ with $a,b \ge 1$ that admit a decoupling
are those that satisfy $a=b$ or $a=2b+1$. This includes the fraction $XY$ (corresponding
to the starting point $(0,0)$),
but also other starting points, lying on two lines.
This result lead us to look for such points, trying to recover this
pattern for the higher $\calH_n$. To this end, we investigate  systematically the Galois 
decoupling of monomials with exponents near the origin 
allowing to formulate this conjecture.
\begin{conj}\label{conj:conj_alg_models}
  \begin{enumerate}
  \item For $n \ge 2$ and $a, b \ge 1$, the orbit of $\calH_n$ is finite and the fraction $X^a Y^b$ admits a $t$-decoupling with respect to
  the model $\calH_n$ if and only if $(a,b) = (n,1)$ or $(a,b) = ((n+1)k, k)$ for some $k$.

\item For  $(a,b)$ as above, 
  the generating functions for walks on $\calH_n$ starting at $(a-1,b-1)$ are algebraic.
  \end{enumerate}
\end{conj}

Regarding part (1) of Conjecture~\ref{conj:conj_alg_models}, it is easy to
prove that $X^n Y$ admits a decoupling with respect to $\calH_n$ for all $n$,
through the following identity: \[
  X^n Y = -\frac{1}{X} + \frac{Y}{t(Y+1)} - \frac{XY(1 - t H_n(X,Y))}{tX(Y+1)}.
\]
For the other mentioned exponents above, we did not find a general
argument. We checked that the first part of the conjecture holds for
$n \le 5$ and $0 \le a,b \le 10$.

Regarding part (2), we apply the strategy presented in Section~\ref{subsect:alg_strat}
 to prove the algebraicity of the models
$\calH_n$ with starting points $(n,0)$ and $(n-1,0)$ and $n \le 3$.
This illustrates the robustness of the strategy  and the significance of our systematic method
to test decoupling, allowing to formulate such conjectures (see the Maple worksheet (also on this \href{https://www.labri.fr/perso/pbonnet/}{webpage}).).

\bibliographystyle{alpha}
\bibliography{gal_walks}

\end{document}